\documentclass[11pt]{amsart}
\usepackage[utf8]{inputenc}
\usepackage{amssymb,amscd,amsmath, young,mathrsfs}
\usepackage{amsthm}
\usepackage{array,booktabs}
\usepackage{mathrsfs, mathdots} 
\usepackage{enumitem}
\usepackage{mathtools}
\usepackage{mathabx}
\usepackage{float}
\usepackage[usenames]{color}
\usepackage{soul}
\usepackage{array,longtable}
\usepackage{faktor}
\usepackage{tikz}
\usepackage{tikz-cd}
\usepackage{forest}

\newcolumntype{C}{>{$}c<{$}}
\newcolumntype{R}{>{$}r<{$}}
\newcolumntype{L}{>{$}l<{$}}
\theoremstyle{plain}
\newtheorem{thm}{Theorem}[section]
\newtheorem{lem}[thm]{Lemma}
\newtheorem{pro}[thm]{Proposition}
\newtheorem{cor}[thm]{Corollary}

\newtheorem{dfn}[thm]{Definition}

\newtheorem{lemma}[thm]{Lemma}

\theoremstyle{remark}
\newtheorem{rem}[thm]{Remark}

\newtheorem*{question*}{Question}

\newtheorem*{acknowledgements}{Acknowledgements}

\newcommand{\N}{\mathbb{N}}
\newcommand{\Z}{\mathbb{Z}}
\newcommand{\Q}{\mathbb{Q}}
\newcommand{\F}{\mathbb{F}}

\newcommand{\Fq}{\mathbb{F}_q}
\def \Paskunas {Pa\u{s}k\={u}nas}

\newcommand{\tensor}{\otimes}

\renewcommand{\epsilon}{\varepsilon}

\renewcommand{\phi}{\varphi}

\newcommand{\mcO}{\mathcal{O}}

\newcommand{\dpr}{{\prime \prime}}
\newcommand{\Nt}{\widetilde{\mathbb{N}}}
\newcommand{\Sym}{\mathrm{Sym}}

\DeclareMathOperator{\GL}{GL}

\begin{document}
\title[Principal series of $\mathrm{GL}_2$ over finite rings]{On the structure of modular principal series representations of $\GL_2$ over some finite rings}
%\date{\today}
\author{Michael M.~Schein}  \address{Department of Mathematics, Bar-Ilan University, Ramat Gan 5290002, Israel}
\author{Re'em Waxman} \address{Department of Mathematics, Bar-Ilan University, Ramat Gan 5290002, Israel}
\maketitle

\thispagestyle{empty}
\begin{abstract}
The submodule structure of mod $p$ principal series representations of $\mathrm{GL}_2(k)$, for $k$ a finite field of characteristic $p$, was described by Bardoe and Sin and has played an important role in subsequent work on the mod $p$ local Langlands correspondence.  The present paper studies the structure of mod $p$ principal series representations of $\mathrm{GL}_2(\mathcal{O} / \mathfrak{m}^n)$, where $\mathcal{O}$ is the ring of integers of a $p$-adic field $F$ and $\mathfrak{m}$ its maximal ideal.  In particular, the multiset of Jordan-H\"older constituents is determined.

In the case $n = 2$, more precise results are obtained.  If $F / \Q_p$ is totally ramified, the submodule structure of the principal series is determined completely.  Otherwise the submodule structure is infinite.  When $F$ is ramified but not totally ramified, the socle and radical filtrations are determined and a specific family of submodules, providing a filtration of the principal series with irreducible quotients, is studied; this family is closely related to the image of a functor of Breuil.  In the case of unramified $F$, the structure of a particular submodule of the principal series is studied; this provides a more precise description of the structure of a module constructed by Breuil and \Paskunas in the context of their work on diagrams giving rise to supersingular mod $p$ representations of $\mathrm{GL}_2(F)$.
\end{abstract}

\section{Introduction}
\subsection{Tame principal series}
Let $k$ be a finite field and $n \geq 2$.  Let $P \leq \mathrm{GL}_n$ be a parabolic subgroup with Levi subgroup $M$ isomorphic to $\mathrm{GL}_{n-1} \times \mathrm{GL}_1$, and let $\chi: P(k) \to \overline{k}$ be a character; observe that $\chi$ necessarily factors through $M(k)$.  If $\chi$ also factors through projection to the second component of $M(k)$, then, motivated by applications to coding theory, Bardoe and Sin~\cite{BS/00} determined the submodule structure of $\mathrm{Ind}_{P(k)}^{\mathrm{GL}_n(k)} \chi$ nearly three decades ago.  In particular, they determined the structure of the principal series representation $I(\chi) = \mathrm{Ind}_{B(k)}^{\mathrm{GL}_2(k)} \chi$, where $B \leq \mathrm{GL}_2$ is a Borel subgroup.  It turns out that each Jordan-H\"older constituent $\sigma$ of $I(\chi)$ appears with multiplicity one.  Thus there is a unique submodule of $I(\chi)$ with irreducible cosocle isomorphic to $\sigma$, and any submodule of $I(\chi)$ is a sum of submodules of this type; hence $I(\chi)$ has only finitely many submodules.

The structural results of Bardoe and Sin found a dramatic new application in the work of Breuil and {\Paskunas}~\cite{BP/12} towards an expected mod $p$ local Langlands correspondence for $\mathrm{GL}_2(F)$, where $F / \Q_p$ is a finite unramified extension. Recall that when $F = \Q_p$ this correspondence was described in a functorial way in a series of works including~\cite{Breuil/03, Colmez/10}, but far less is known when $F \neq \Q_p$.  Let $\mathcal{O}$ be the valuation ring of $F$, and let $k$ be its residue field.  Let $\rho: \mathrm{Gal}(\overline{F} / F) \to \mathrm{GL}_2(\overline{k})$ be a semisimple mod $p$ Galois representation.  Breuil and {\Paskunas} constructed infinite families of diagrams associated to each such $\rho$.  Diagrams are objects consisting of a finite-dimensional representation $D_0(\rho)$ of the finite group $\mathrm{GL}_2(k)$ over $\overline{k}$, together with some additional structure, such as a representation $D_1(\rho)$ of the normalizer of the pro-$p$-Iwahori subgroup $I(1)$ of $\mathrm{GL}_2(\mathcal{O})$.  In the diagrams of~\cite{BP/12}, one has $D_1(\rho) = D_0(\rho)^{I(1)}$.  By the theory developed earlier in~\cite{Paskunas/04}, each of these diagrams gives rise to infinite (see~\cite{Hu/10}) families of $\overline{k}$-representations of $\mathrm{GL}_2(F)$, which are irreducible and supersingular if $\rho$ is irreducible.

Not all smooth irreducible mod $p$ representations of $\mathrm{GL}_2(F)$ arise in this way, and probably not even all whose $\mathrm{GL}_2(\mathcal{O})$-socle is compatible with the weight part of Serre's modularity conjecture.  However, a considerable body of evidence has accumulated over the past fifteen years suggesting that the mod $p$ local Langlands conjecture is realized within the Breuil-{\Paskunas} families when $F / \Q_p$ is unramified and $\rho$ is generic in the sense of~\cite[Definition~11.7]{BP/12}.  Given $\rho$, one can fix a totally real field $L$ and a place $v | p$ of $L$ such that $L_v \simeq F$ and attempt to find a $\mathrm{GL}_2(L_v)$-module $\pi(\rho)$ in the completed cohomology of a tower of Shimura curves over $L$ that realizes the mod $p$ local Langlands correspondence.  It is not yet known in a single case, for $F \neq \Q_p$, that such a representation $\pi(\rho)$ is independent of the many global choices made in the course of its construction.  However, a series of papers by various authors, including~\cite{Breuil/14, EGS/15, HW/18, Le/19a, LMS/22} and culminating in~\cite{DL/21}, show that for any collection of global choices, the injection $\pi(\rho)^{I(1)} \hookrightarrow \pi(\rho)^{K(1)}$ is the same and arises from one of the Breuil-{\Paskunas} diagrams.  Here $K(1)$ is the first congruence subgroup $\ker (\mathrm{GL}_2(\mathcal{O}) \twoheadrightarrow \mathrm{GL}_2(k))$.  In particular, $\pi(\rho)^{K(1)} \simeq D_0(\rho)$.

When the extension $F / \Q_p$ is ramified, very little is known towards the mod $p$ local Langlands correspondence.  A substantial obstacle is that it is not enough to consider $\mathrm{GL}_2 (k)$-modules, and one must work with representations of the larger finite group $\mathrm{GL}_2(\mathcal{O} / \mathfrak{m}^n)$ for $n \geq 2$; here $\mathfrak{m} \triangleleft \mathcal{O}$ is the maximal ideal.  Observe that even the representation theory over $\mathbb{C}$ of such groups is understood more poorly than that of general linear groups over finite fields; cf.~\cite{AOPS/10, Singla/10, Stasinski/09, CMO/21, CMO/24} for examples of recent advances.  While all irreducible mod $p$ representations of $\mathrm{GL}_2(\mathcal{O} / \mathfrak{m}^n)$ factor through $\mathrm{GL}_2(k)$ (see, for instance,~\cite[Lemma~3]{BL/94}), the structure of reducible representations can be very different.

In particular, we would like to understand the structure of the principal series $$I_n(\chi) = \mathrm{Ind}_{B(\mathcal{O} / \mathfrak{m}^n)}^{\mathrm{GL}_2(\mathcal{O}/\mathfrak{m}^n)} \chi.$$
The present paper begins to fill this lacuna.

\subsection{Main results}
This section summarizes the main results of the article; in order to avoid excessive details at this stage, some notations of the introduction differ from those used in the body of the paper.
We view $I_n(\chi)$ as a representation over an extension field $k \subseteq E$; our results will be independent of $E$.  Let the cardinality of the residue field $k$ be $q = p^f$; assume that $p$ is odd.

\subsubsection{Jordan-H\"older constituents}
We start by determining the Jordan-H\"older constituents of $I_n(\chi)$.  
Definition~\ref{dfn:general.basis} fixes an $E$-basis $\mathcal{B}$ of $I_n(\chi)$.  Consider the set $S = \{ 0, 1, \dots, p - 1 \}^f$ endowed with a partial order $\preceq$ defined by $(r_0, \dots, r_{f-1}) \preceq (s_0, \dots, s_{f-1})$ if $r_i \leq s_i$ for all $0 \leq i \leq f - 1$.  Later it will be useful to consider a monoid structure on $S$; this is the monoid $\Nt$, defined in \S\ref{sec:notation} that is ubiquitous throughout the paper.  For each $\alpha \in S$ we fix (Definition~\ref{def:filtration}) a subset $\mathcal{B}_\alpha \subseteq \mathcal{B}$ such that $\mathcal{B}_\alpha \subseteq \mathcal{B}_\beta$ if $\alpha \preceq \beta$ and consider the subspace $W_\alpha$ of $I_n(\chi)$ spanned by $\mathcal{B}_\alpha$.  Since $W_{(p-1, \dots, p -1)} = I_n(\chi)$, this construction produces an exhaustive filtration of $I_n(\chi)$.  Note that any character $\chi$ of $B(\mathcal{O}/\mathfrak{m}^n)$ factors through a character of $B(k)$, and let $\eta : B(k) \to k^\times \subseteq E^\times$ be the character
$ \eta: \left( \begin{array}{cc} a & b \\ 0 & d \end{array} \right) \mapsto a d^{-1}$.

In Corollary~\ref{cor:jh.constituents.general} we find a recursive description of the multisets of Jordan-H\"older constituents of $I_n(\chi)$, starting with the sets of Jordan-H\"older constituents of the tame principal series $I(\chi) = I_1(\chi)$, which were determined by Bardoe and Sin.  The proof produces the following more precise result, which is
Proposition~\ref{pro:filtration}.

\begin{thm} \label{thm:intro.jh}
For each $\beta = (\beta_0, \dots, \beta_{f-1}) \in S$, the subspace $W_\beta$ is a $\mathrm{GL}_2(\mathcal{O} / \mathfrak{m}^n)$-submodule of $I_n(\chi)$.  Moreover, there is an isomorphism of $\mathrm{GL}_2(\mathcal{O} / \mathfrak{m}^n)$-modules
$$ W_\beta / \sum_{\alpha \in S \atop \alpha \prec \beta} W_\alpha \simeq  I_{n-1} (\chi \cdot \eta^{\sum_{i = 0}^{f-1} \beta_i p^i}),$$
where $\mathrm{GL}_2(\mathcal{O} / \mathfrak{m}^n)$ acts on the right-hand side via its natural projection to $\mathrm{GL}_2(\mathcal{O} / \mathfrak{m}^{n-1})$.
\end{thm}

\subsubsection{A family of submodules} \label{sec:intro.family}
The remainder of the paper restricts itself to the case $n = 2$ in order to obtain more precise results.  Then to every element of the basis $\mathcal{B}$ we associate a type, namely a pair $(I, \gamma)$, where $I \subseteq \Z / f\Z$ and $\gamma \in S$.  The set $I$ is determined by the columns where a carry must be performed when a certain two elements of $S$, viewed as $f$-digit numbers written in base $p$, are added.  Definition~\ref{def:covering.relations.ramified} introduces a partial order $\leq_\chi$, depending on the character $\chi$, on the set of types, in terms of generating relations.  An equivalent explicit closed-form definition of $\leq_\chi$ is obtained in Proposition~\ref{pro:equality.of.orders} by means of long but completely elementary manipulations involving the properties of carry sets.  For every type $(I, \gamma)$, define $V_{(I, \gamma)}$ to be the subspace of $I_2(\chi)$ spanned by the elements of $\mathcal{B}$ whose type is less than or equal to $(I, \gamma)$, with respect to the partial order $\leq_\chi$.  The following claim consists of Proposition~\ref{pro:coordinate.submodules.stable} and Theorem~\ref{thm:coordinate.submodules}:

\begin{thm} \label{thm:intro.submodules}
Suppose that the extension $F / \Q_p$ is ramified.  For every type $(I, \gamma)$, the subspace $V_{(I, \gamma)}$ of $I_2(\chi)$ is stable under the action of $\mathrm{GL}_2(\mathcal{O} / \mathfrak{m}^2)$.  Moreover, $V_{(I, \gamma)}$ is generated, as a $\mathrm{GL}_2(\mathcal{O} / \mathfrak{m}^2)$-submodule of $I_2(\chi)$, by any element of $\mathcal{B}$ of type $(I, \gamma)$.
\end{thm}
The proof consists of explicit calculations that fail when $F / \Q_p$ is unramified because of some complications introduced by summation of Witt vectors, and indeed the claim is false when $F / \Q_p$ is unramified.

\subsubsection{Submodule structure in the totally ramified case}
An essential feature of the tame principal series $I(\chi)$ is that they are multiplicity-free: each Jordan-H\"older constituent appears only once as a quotient of any composition series.  As a consequence, $I(\chi)$ has only finitely many $\mathrm{GL}_2(k)$-submodules.  For any Jordan-H\"older constituent $\sigma$ of $I(\chi)$, there is a unique submodule with cosocle isomorphic to $\sigma$, and this submodule can be specified by writing down its set of Jordan-H\"older constituents.  Any submodule of $I(\chi)$ is a sum of submodules with irreducible cosocle.  A description of the submodule structure of $I(\chi)$ of this form is very useful for applications; it provides a method for determining when two submodules are equal, or when one is contained in another.  See~\cite[Proposition~6.2.2]{BHHMS1/23} for another statement of this form, in a different setting.  When the extension $F / \Q_p$ is totally ramified, the principal series $I_2(\chi)$ is not multiplicity-free (Propositions~\ref{pro:totally.ramified.odd} and~\ref{pro:totally.ramified.even}) but still admits a similar complete description of its submodule structure.  We say that $\chi$ is odd if $\chi \eta^{i}$ does not factor through the determinant for any $i \in \Z / (q-1)\Z$, and that $\chi$ is even otherwise.  The following is Theorem~\ref{cor:finite.submodule.structure}, partially restricted to the case of odd $\chi$.  The statement for even $\chi$ is very similar but slightly more complicated, as it must treat an exceptional case, so we suppress it for the purposes of this introduction.

\begin{thm} \label{thm:intro.totram}
Suppose $F / \Q_p$ is totally ramified and non-trivial.  The principal series $I_2(\chi)$ has $2p$ Jordan{\-}-H\"older constituents, parametrized by types $(I, \gamma)$, where $I \subseteq \{ 0 \}$ and $\gamma \in \{ 0, 1, \dots, p - 1 \}$.  Moreover, $I_2(\chi)$ has finitely many submodules.  

If $\chi$ is odd, then the submodules with irreducible cosocle are exactly the submodules $V_{(I, \gamma)}$.  The multiset of Jordan-H\"older constituents of $V_{(I, \gamma)}$ is $\left\{ L(I^\prime, \gamma^\prime) : (I^\prime, \gamma^\prime) \leq_\chi (I, \gamma) \right\}$, where $L(I, \gamma)$ is the constituent parametrized by the type $(I, \gamma)$.
\end{thm}
A depiction of the partial order $\leq_\chi$ on the set of types, when $\chi$ is odd and $F / \Q_p$ is totally ramified, may be found in Figure~\ref{fig:jh.graph}.

\subsubsection{Partial submodule structure in the general ramified case}
If $F / \Q_p$ is ramified but not totally ramified, then $I_2(\chi)$ has an infinite lattice of submodules if the base field $E$ is infinite (Proposition~\ref{pro:infinite.submodule.structure}) and we cannot hope for a complete explicit description of the submodules of $I_2(\chi)$ as in Theorem~\ref{thm:intro.totram}.  Most submodules are not sums of the $V_{(I, \gamma)}$.    However, for the submodules in this family much of Theorem~\ref{thm:intro.totram} still holds.  The following is  Lemma~\ref{pro:head} and Proposition~\ref{pro:adjacent.types}.

\begin{thm} \label{thm:intro.ramified}
Suppose $F / \Q_p$ is ramified.  The Jordan-H\"older constituents of $I_2(\chi)$ are parametrized by a (possibly proper) subset of the set of pairs $(I, \gamma)$, where $I \subseteq \Z / f\Z$ and $\gamma \in S$.  Moreover, $V_{(I^\prime, \gamma^\prime)} \subseteq V_{(I, \gamma)}$ if and only if $(I^\prime, \gamma^\prime) \leq_\chi (I, \gamma)$, and the multiset of Jordan-H\"older constituents of $V_{(I,\gamma)}$ is $\left\{ L(I^\prime, \gamma^\prime) : (I^\prime, \gamma^\prime) \leq_\chi (I, \gamma) \right\}$. 

Moreover, excluding an exceptional case, the submodule $V_{(I, \gamma)}$ has irreducible cosocle $L(I, \gamma)$.
\end{thm}

Although Theorem~\ref{thm:intro.ramified} appears to describe only a very small part of the submodule lattice of $I_2(\chi)$, it describes an important part.  Breuil~\cite{Breuil/15} has defined a functor from the category of mod $p$ representations of $\mathrm{GL}_2(F)$ to that of $(\varphi, \Gamma)$-modules over $\Q_p$, and 
subquotients of $I_2(\chi)$ generated by elements of the basis $\mathcal{B}$ turn out to be relevant for the study of the image of $\pi(\rho)$ under Breuil's functor, with the aim of producing a description of this image analogous to that of~\cite[\S4]{BHHMS/21}.

\subsubsection{The unramified case}
All these results, except for Theorem~\ref{thm:intro.jh} fail when $F / \Q_p$ is unramified.  Indeed, we are unaware of a basis of $I_2(\chi)$ with the property that the $\mathrm{GL}_2(\mathcal{O} / \mathfrak{m}^2)$-submodule of $I_2(\chi)$ generated by any element of the basis is the linear span of a subset of the basis.  However, we can still obtain interesting partial results.  Assume for simplicity that $\chi$ does not factor through the determinant, and let $\sigma_\chi$ be the unique irreducible representation of $\mathrm{GL}_2(k)$, and hence of $\mathrm{GL}_2(\mathcal{O} / \mathfrak{m}^2)$, with lowest weight character $\chi$.  Then $I_2(\chi)$ has irreducible socle isomorphic to $\sigma_\chi$.  If $F / \Q_p$ is unramified, then we define a cyclic submodule $M(\chi) \subseteq I_2(\chi)$ that is analogous to $V_{(\varnothing, (1,1, \dots, 1))}$.  The submodule $M(\chi)$ is multiplicity-free, and we can describe its submodule structure completely by means of explicit calculations that are less pleasant than the ones in the rest of the paper.  In~\cite[\S17]{BP/12}, Breuil and {\Paskunas} associated a $\mathrm{GL}_2(\mathcal{O})$-module $R(\sigma)$ to every generic irreducible representation $\sigma$ of $\mathrm{GL}_2(k)$; these modules played a technical but essential role in the proof of the main results of~\cite{BP/12}.  Under the genericity assumption of Definition~\ref{def:strong.genericity} on $\chi$, we recover $R(\sigma_\chi)$ as a submodule of $I_2(\chi)$ and describe its submodule structure completely, thereby strengthening an essential property of $R(\sigma_\chi)$ established in~\cite[\S18]{BP/12} while simplifying its proof.  
We note parenthetically that our study of $M(\chi)$ may also be developed when $\chi$ is not generic.  An application, which is the aim of work in progress, is to generalize the definition of Breuil-{\Paskunas} diagrams to non-generic cases. 

Indeed, let $\widetilde{\Theta}_1 = \{ (J,I) : \, I, J \subseteq \Z / f\Z, \, J \cap \{ i - 1: i \in I \} = \varnothing \}$.  To each $\theta \in \widetilde{\Theta}_1$ we associate an explicit irreducible representation $\sigma_\theta$ of $\mathrm{GL}_2(k)$ by means of a recipe that depends on $\chi$.  Define a partial order on $\widetilde{\Theta}_1$ by setting $(J^\prime, I^\prime) \sqsubseteq (J,I)$ if the two conditions $I^\prime \subseteq I$ and $J^\prime \subseteq J \cup \{ i - 1: i \in I \setminus I^\prime \}$ hold.  The following is Lemma~\ref{lem:r.sigma.characterization} and Theorem~\ref{thm:unramified.submodule.structure}.

\begin{thm} \label{thm:intro.unramified}
Suppose that $F / \Q_p$ is unramified and that $\chi$ is a generic character.  Then $M(\chi) \simeq R(\sigma_\chi)$ as $\mathrm{GL}_2(\mathcal{O})$-modules.  The module $M(\chi)$ is multiplicity-free with Jordan-H\"older constituents $\{ \sigma_\theta: \, \theta \in \widetilde{\Theta}_1 \}$.  For every $\theta \in \widetilde{\Theta}_1$, the unique submodule with cosocle $\sigma_\theta$ has Jordan-H\"older constituents $\{ \sigma_{\theta^\prime}: \, \theta^\prime \in \widetilde{\Theta}_1, \, \theta^\prime \sqsubseteq \theta \}$.
\end{thm}

 \subsection{Overview of the paper} 
 In \S\ref{sec:prelim} we set notation, recall the work of Bardoe and Sin~\cite{BS/00} about the structure of the tame principal series $I(\chi)$, and establish properties of the carry sets mentioned in~\S\ref{sec:intro.family} above.  In~\S\ref{sec:jh.filtration}, Theorem~\ref{thm:intro.jh} is proved, and some general explicit computations are presented.  In~\S\ref{sec:main} we specialize to the case $n = 2$, and soon after to the case of ramified extensions $F / \Q_p$, and prove the main results of the paper.   In~\S\ref{sec:R.sigma} the computations are modified to treat the case of unramified $F / \Q_p$, and Theorem~\ref{thm:intro.unramified} is proved.  The final section~\S\ref{sec:equality.of.orders}, which is an appendix to the main body of the paper and essentially independent of it, contains the combinatorial proof of Proposition~\ref{pro:equality.of.orders} and thereby provides a closed-form definition of the partial order $\leq_\chi$ appearing in the statements of Theorems~\ref{thm:intro.totram} and~\ref{thm:intro.ramified}.

\begin{acknowledgements}
We are grateful to Shalini Bhattacharya, Stefano Morra, and Ariel Weiss for illuminating discussions during the course of this work and its prehistory.
\end{acknowledgements}

\section{Preliminaries} \label{sec:prelim}

\subsection{Notation} \label{sec:notation}
Let $F / \Q_p$ be a finite extension.  Let $\mathcal{O}$ be the valuation ring of $F$, let $k$ be its residue field, and let $q = p^f$ be the cardinality of $k$.  Fix a uniformizer $\varpi \in \mathcal{O}$, namely a generator of the maximal ideal $\mathfrak{m}$.  Let $e$ be the ramification index of $F / \Q_p$, so that $p \mathcal{O} = \mathfrak{m}^e$.

\subsubsection{The monoid $\Nt$}
For every $n \in \N$ we set $[n] = \{ 1, 2, \dots, n \}$ and $[n]_0 = [n] \cup \{ 0 \}$.  The symbol $\subset$ denotes strict inclusion.  Set $\N_0 = \N \cup \{ 0 \}$, and consider the equivalence relation on $\N_0$ defined by $a \sim b$ if the functions $x \mapsto x^a$ and $x \mapsto x^b$ from $k$ to itself are equal.  In other words, $a \sim b$ if and only if either $a = b = 0$ or if $a, b > 0$ and $a \equiv b \, \mathrm{mod} \; q - 1$.  Observe that $\Nt = \N_0 / \! \sim$ is a monoid under addition.  

For every class $\alpha \in \Nt$ there exists a unique $f$-tuple $(\alpha_0, \dots, \alpha_{f-1}) \in ([p-1]_0)^f$ satisfying $\sum_{i = 0}^{f-1} \alpha_i p^i \in \alpha$.  Define a partial order $\preceq$ on $\Nt$ as follows: $\alpha \preceq \beta$ if and only if $\alpha_i \leq \beta_i$ for all $i \in [f-1]_0$.  Equivalently, by Lucas's theorem, $\alpha \preceq \beta$ if and only if the binomial coefficient $\binom{\sum_{i = 0}^{f-1} \beta_i p^i}{\sum_{i = 0}^{f-1} \alpha_i p^i}$ does not vanish in $k$.  

We will use two notions of subtraction on $\Nt$, each of which is only partially defined.  Observe that the subsemigroup $\Nt \setminus \{ 0 \} \subset \Nt$, which is not a submonoid, has the structure of a group and is naturally identified with $\Z / (q-1) \Z$.  
\begin{enumerate}
\item
Given $a, b \in \Nt$ such that $a \neq 0$, define the difference $a - b$ to be the unique $c \in \Nt \setminus \{ 0 \}$ satisfying $b + c = a$. 
\item
If $\alpha, \beta \in \Nt$ satisfy $0 \neq \alpha \preceq \beta$, then we define $\beta \dotdiv \alpha$ to be the unique $\gamma \prec \beta$ satisfying $\alpha + \gamma = \beta$.  
\end{enumerate}
Thus, for example, we have $(q - 1) - (q - 1) = q - 1$ but $(q - 1) \dotdiv (q - 1) = 0$.

If $I \subseteq [f-1]_0$, we use the shorthand $p^I = \sum_{i \in I} p^i$.  If $b \in \Z / f\Z$, then, identifying $[f-1]_0$ with $\Z / f\Z$ in the natural way, we set $I - b = \{ i - b : i \in I \}$.  If $\alpha \in \Nt$, then the support of $\alpha$ is the set $\mathrm{supp}(\alpha) = \{ i \in [f-1]_0 : \alpha_i \neq 0 \}$.

\subsubsection{Induced representations}
Let $G$ be any group and $H \leq G$ a subgroup.  Let $E$ be a field, and let $\rho: H \to \mathrm{Aut}_E(V)$ be a representation of $H$ with underlying $E$-vector space $V$.  If $g \in G$ and $v \in v$, then set $g \tensor v \in \mathrm{Ind}_H^G \rho$ to be the function $f : G \to V$ supported on the right coset $H g^{-1}$ and satisfying $f(hg^{-1}) = \rho(h)v$ for all $h \in H$.  Observe that $gh \tensor v = g \tensor \rho(h)v$ for all $h \in H$ and that $x(g \tensor v) = xg \tensor v$ for all $x \in G$.

\subsubsection{Witt vectors}
If $x \in k$, let $[x] \in \mathcal{O}$ be the Teichm\"uller lift of $x$.  For a class $\lambda \in \mathcal{O} / \mathfrak{m}^n$, we set $\lambda_0, \lambda_1, \dots, \lambda_{n-1}$ to be the unique sequence of elements of the residue field $k$ satisfying $\sum_{i = 0}^{n-1} [\lambda_i] \varpi^i \in \lambda$.  Abusing notation, we will write $\lambda = \sum_{i = 0}^{n-1} [\lambda_i] \varpi^i$.  Similarly, any element $\lambda \in \mathcal{O}$ can be written uniquely as a series $\lambda = \sum_{i = 0}^\infty [\lambda_i] \varpi^i$.

Let $u = \frac{p}{\varpi^e} \in \mathcal{O}^\times$, and note that $u_0 \in k^\times$.  Consider the polynomial
$S(x,y) = \frac{x^p + y^p - (x + y)^p}{p} \in \Z[x,y]$.
The following identity will be crucial to our computations.
\begin{lem} \label{lem:witt.vector.summation}
Let $a, b \in \mathcal{O}$.  Write $a = \sum_{i = 0}^\infty [a_i] \varpi^i$ and $b = \sum_{i = 0}^\infty [b_i] \varpi^i$, for $a_i, b_i \in k$.  Then
$$ a + b \equiv \sum_{i = 0}^{e-1} [a_i + b_i] \pi^i + [ a_e + b_e + u_0 S(a_0,b_0)^{p^{f-1}}] \varpi^e \, \mathrm{mod} \, \varpi^{e+1}.$$
\end{lem}
\begin{proof}
Let $F_0$ be the maximal unramified subextension of $F / \Q_p$.  Observe that $[\lambda] \in F_0$ for all $\lambda \in k$ by Hensel's lemma.  We view elements of the Witt vector ring $W(k) = \mathcal{O}_{F_0}$ as sequences of elements of $k$ as in~\cite[\S{II.6}]{Serre/62} and use the notation $\mathrm{Fr}$ and $\mathrm{V}$ for the Frobenius and Verschiebung operators.  Then $[\lambda] = (\lambda, 0, 0, \dots)$, as the right-hand side is easily seen to be a $(p-1)$-th root of unity congruent to $\lambda$ modulo $p$.  Similarly, $p [ \lambda ] = \mathrm{Fr}(\mathrm{V}([\lambda])) = (0, \lambda^p, 0, \dots)$.  Hence,
$$ [\lambda] + [\mu] = (\lambda + \mu, S(\lambda, \mu), 0, \dots) \equiv [\lambda + \mu] + p [S(\lambda, \mu)^{p^{f-1}}] \, \mathrm{mod} \, p^2$$
for all $\lambda, \mu \in k$.  Since $p \equiv [u_0] \varpi^e \, \mathrm{mod} \, \varpi^{e+1}$, the claim follows.
\end{proof}

If $m \leq n$, then $\psi_m^n$ denotes the natural projection $\mathcal{O}_n \to \mathcal{O}_m$ of rings, and also the natural projection $G_n \to G_m$ of groups that it induces.
By $\mathbf{1}$, we always mean the trivial representation of the relevant group.

\begin{dfn} \label{def:subgroups}
We set notation for the following subgroups of $G_n$:
\begin{itemize}
\item
$T = \left\{ \left( \begin{array}{cc} [a] & 0 \\ 0 & [d] \end{array} \right): a,d \in k^\times \right\}$ ;
\item
$D_n = \left\{ \left( \begin{array}{cc} 1 & 0 \\ 0 & 1 + d \varpi \end{array} \right): d \in \mathcal{O}_{n-1} \right\} $;
\item
$U_n = \left\{ \left( \begin{array}{cc} 1  & b \\ 0 & 1 \end{array} \right): b \in \mathcal{O}_n \right\}$;
\item
$\overline{U}_n = \left\{ \left( \begin{array}{cc} 1  & 0 \\ c \varpi & 1 \end{array} \right): c \in \mathcal{O}_{n-1} \right\}$.
\end{itemize}
\end{dfn}

Let $B \leq \mathrm{GL}_2$ be the algebraic subgroup of upper triangular matrices.  For $n \in \N$, write $\mathcal{O}_n = \mathcal{O} / \mathfrak{m}^n$ and let $G_n = \mathrm{GL}_2(\mathcal{O}_n)$, and $B_n = B(\mathcal{O}_n)$.  Let $w = \left( \begin{array}{cc} 0 & 1 \\ 1 & 0 \end{array} \right) \in G_n$.  If $r \in \Z / (q-1) \Z$, then we define $\chi_r : B_n \to k^\times$ to be the character
\begin{equation*}
\chi_r : \left( \begin{array}{cc} a & b \\ 0 & d \end{array} \right) \mapsto d_0^r.
\end{equation*}

\subsection{Principal series representations}
The primary aim of this paper is to investigate the following representations of $G_n$:
\begin{dfn}
Let $n \in \N$ and let $r \in \Z / (q-1) \Z$.  We set $V_{n,r} = \mathrm{Ind}^{G_n}_{B_n} \chi_r$.
\end{dfn}

We view $V_{n,r}$ as a representation of dimension $(q+1)q^{n-1}$ over some sufficiently large extension field $k \subseteq E$.  
For explicit computations below, it is convenient to define an explicit $E$-basis of $V_{n,r}$.  

\begin{dfn} \label{dfn:general.basis}
If $\mathbf{j} = (j_0, \dots, j_{n-1}) \in \Nt^n$, set
$$ f_{\mathbf{j}} = \sum_{\lambda \in \mathcal{O}_n} \left( \begin{array}{cc} \lambda & 1 \\ 1 & 0 \end{array} \right) \tensor \lambda_0^{j_0} \cdots \lambda_{n-1}^{j_{n-1}}.$$
If $\mathbf{j}^\prime = (j_1, \dots, j_{n-1}) \in \Nt^{n-1}$, then set
$$ f_{(\infty, \mathbf{j}^\prime)} = w f_{(0,\mathbf{j}^\prime)} = \sum_{\lambda \in \mathcal{O}_n^\times} \left( \begin{array}{cc} \lambda & 1 \\ 1 & 0 \end{array} \right) \tensor (-\lambda_0)^r \prod_{i = 1}^{n-1} P_i(\lambda)^{j_i} + \sum_{\lambda \in \mathcal{O}_{n-1}} \left( \begin{array}{cc} 1 & 0 \\ \lambda \varpi & 1 \end{array} \right) \tensor \prod_{i = 1}^{n-1} \lambda_i^{j_i}.$$
\end{dfn}
Here $P_i : \mathcal{O}_n^\times \to k$ is defined by $\lambda^{-1} = \sum_{i = 0}^{n-1} [P_i(\lambda)] \varpi^i$; see Lemma~\ref{lem:inversion.power.series} below for an explicit formula.
We write $\lambda^\mathbf{j}$ for $\lambda_0^{j_0}\lambda_1^{j_1}\dots\lambda_{n-1}^{j_{n-1}}$ when this is unlikely to cause confusion.  Note also that $r \in \Z / (q - 1)\Z = \Nt \setminus \{ 0 \}$ by definition, so that the first sum in the previous displayed formula receives no contribution from $\lambda \in \mathcal{O}_n$ such that $\lambda_0 = 0$ and thus is actually a sum over $\mathcal{O}_n^\times$, where the functions $P_i$ are defined.

\begin{lem} \label{lem:bnr.basis}
    The set $\mathcal{B}_{n,r} = \{f_{\mathbf{j}}\mid \mathbf{j}\in\Nt^n\}\cup\{f_{(\infty,\mathbf{j}^\prime)}\mid \mathbf{j}^\prime\in\Nt^{n-1}\}$ is an $E$-basis of $V_{n,r}$. 
\end{lem}
\begin{proof}
A set of representatives of left cosets of $B_n$ in $G_n$ is given by
$$\Xi = \left\{ \left( \begin{array}{cc} \lambda & 1 \\ 1 & 0 \end{array} \right) : \lambda \in \mathcal{O}_n \right\} \cup \left\{ \left( \begin{array}{cc} 1 & 0 \\ \lambda \pi & 1 \end{array} \right) : \lambda \in \mathcal{O}_{n-1} \right\}, $$
from which it is clear that $\{ \xi \tensor 1 : \xi \in \Xi \}$ is a basis of $V_{n,r}$.  Iterating a standard Vandermonde argument (cf.~\cite[Lemma~2.10(ii)]{Morra/11}) we find that the set $\{ f_{\mathbf{j}} : \mathbf{j} \in \Nt^n \}$ is linearly independent, and that it spans the subspace $V^\prime_{n,r}$ spanned by $\left\{ \left( \begin{array}{cc} \lambda & 1 \\ 1 & 0 \end{array} \right) \tensor 1: \lambda \in \mathcal{O}_n \right\}$.  Since $V_{n,r}^\prime \subset \mathrm{span}(\mathcal{B}_{n,r})$, it is evident that $\mathrm{span}(\mathcal{B}_{n,r})$ contains the linear span of the set $\left\{ \sum_{\lambda \in \mathcal{O}_{n-1}} \left( \begin{array}{cc} 1 & 0 \\ \lambda \varpi & 1 \end{array} \right) \tensor \lambda^{\mathbf{j}} : \mathbf{j} \in \Nt^{n-1} \right\}$, which is the same as the linear span of $\left\{ \left( \begin{array}{cc} 1 & 0 \\ \lambda \varpi & 1 \end{array} \right) \tensor 1 : \lambda \in \mathcal{O}_{n-1} \right\}$ by the Vandermonde argument.  Hence $\mathrm{span}(\mathcal{B}_{n,r}) = V_{n,r}$.  Furthermore, $| \mathcal{B}_{n,r} | = (q + 1)q^{n-1} = \dim_E V_{n,r}$, so $\mathcal{B}_{n,r}$ is indeed a basis of $V_{n,r}$.
\end{proof}

The elements of Definition~\ref{dfn:general.basis} have the advantage of being eigenvectors for the action of the torus $T = \left\{ \left( \begin{array}{cc} [a] & 0 \\ 0 & [d] \end{array} \right) : a,d \in k^\times \right\}$ and of the center $Z(G_n)$.

\begin{lem} \label{rem:eigenvectors}  The following equalities hold:
\begin{enumerate}
\item
  If $t \in T$, then
\begin{eqnarray*}
t f_{\mathbf{j}} & = & a^r (a^{-1}d)^{j_0 + \cdots + j_{n-1}} f_{\mathbf{j}} \\
t f_{(\infty, \mathbf{j}^\prime)} & = & d^r (a d^{-1})^{j_1 + \cdots + j_{n-1}} f_{(\infty, \mathbf{j}^\prime)}.
\end{eqnarray*}
\item
If $\alpha \in \mathcal{O}_n^\times$ and $z = \mathrm{diag}(\alpha,\alpha) \in Z(G_n)$, then $z \varphi = \alpha_0^r \varphi$ for any $\varphi \in \mathcal{B}_{n,r}$.
\end{enumerate}
\end{lem}
\begin{proof}
This is a simple calculation.
\end{proof}

\subsection{Carry sets} \label{sec:carry.sets}
If $n = 1$, then the principal series representation $V_{1,r} = \mathrm{Ind}^{\mathrm{GL}_2(k)}_{B(k)} \chi_r$ was studied by Bardoe and Sin~\cite{BS/00}.  We will give a mild reformulation of their results, in terms of the definition below.

\begin{lemma} \label{lem:carry.set.def}
    Let $\alpha, \beta \in \Nt$.   As above, let $\alpha_i, \beta_i \in [p-1]_0$ be such that $\sum_{i = 0}^{f-1} \alpha_i p^i \in \alpha$ and $\sum_{i = 0}^{f-1} \beta_i p^i \in \beta$.  View the indices $i$ as elements of $\Z / f\Z$ via the natural identification of the sets $[f-1]_0$ and $\Z / f\Z$.  There exists a unique subset $I(\alpha, \beta) \subseteq \Z / f\Z$ satisfying the conditions:
\begin{itemize}
\item
\itemsep0em
    $\alpha_i + \beta_i \leq p -1$ if $i-1 \not\in I(\alpha, \beta)$ and $i \not\in I(\alpha, \beta)$,
\item
    $\alpha_i + \beta_i \leq p - 2$ if $i-1 \in I(\alpha, \beta)$ and $i \not\in I(\alpha, \beta)$,
\item
    $\alpha_i + \beta_i \geq p$ if $i - 1 \not\in I(\alpha, \beta)$ and $i \in I(\alpha, \beta)$,
\item
    $\alpha_i + \beta_i \geq p - 1$ if $i - 1 \in I(\alpha, \beta)$ and $i \in I(\alpha, \beta)$,
\item
    $I(\alpha, \beta) = \varnothing$ if $\alpha_i + \beta_i = p - 1$ for all $i \in [f-1]_0$.
\end{itemize}
\end{lemma}
\begin{proof}
Suppose that there is some index $i$ for which $\alpha_i + \beta_i \neq p - 1$.  If $\alpha_i + \beta_i < p - 1$, then the first two conditions imply that if such a set $I(\alpha, \beta)$ exists, then necessarily $i \not\in I(\alpha, \beta)$.  Similarly, if $\alpha_i + \beta_i > p - 1$ then $i \in I(\alpha, \beta)$ by the third and fourth conditions.  In either case, the conditions now determine whether or not $i + 1 \in I(\alpha, \beta)$.  Continuing in this way, we find that a set satisfying the first four conditions exists and is unique.

The only unresolved case is when $\alpha_i + \beta_i = p - 1$ for all $i$.  In this case, the two sets $\varnothing$ and $\Z / f\Z$ satisfy the first four conditions, but the fifth one determines $I(\alpha, \beta) = \varnothing$.
\end{proof}
\begin{rem}
Informally, $I(\alpha, \beta)$ is the set of columns where a carry is performed when computing the sum of $\sum_{i = 0}^{f-1} \alpha_i p^i$ and $\sum_{i = 0}^{f-1} \beta_i p^i$ in base $p$.  Observe that carries can only occur when $\alpha$ and $\beta$ are both non-zero; indeed, $I(\alpha, \beta) = \varnothing$ if $\alpha = 0$ or $\beta = 0$.  Otherwise, since we are working in $\Nt \setminus \{ 0 \} = \Z / (q-1) \Z$, any excess from the $p^{f-1}$ column is carried to the $p^0$ column.
\end{rem}

\begin{dfn} \label{def:admissible.set}
Let $\gamma \in \Nt$.  We say that a set $J \subseteq \Z / f\Z$ is $\gamma$-{\emph{admissible}} if there exist $\alpha, \beta \in \Nt$ such that $\alpha + \beta = \gamma$ and $I(\alpha, \beta) = J$.
\end{dfn}

The $\gamma$-admissible sets are easy to describe directly.  We leave the proof of the following observation to the reader, noting that $I(q-1, q-1) = \Z / f\Z$.
\begin{pro} \label{pro:characterize.admissible}
Let $\gamma \in \Nt$ and $J \subseteq \Z / f\Z$.
\begin{enumerate}[label=(\alph*)]
\item
If $\gamma \in \{ 0, q - 1 \}$, then $J$ is $\gamma$-admissible if and only if $J \in \{ \varnothing, \Z / f\Z \}$.
\item
If $\gamma \not\in \{0, q-1 \}$, then $J$ is $\gamma$-admissible if and only if the following two conditions are satisfied for every $i \in \Z / f\Z$:
\begin{itemize}
\item
If $\gamma_i = 0$ and $i - 1 \in J$, then $i \in J$.
\item
If $\gamma_i = p-1$ and $i \in J$, then $i - 1 \in J$.
\end{itemize}
\end{enumerate}
\end{pro}

The following lemma and its corollary state a property of carry sets that will be essential for some of our arguments below.  It has surely been well-known for centuries, but we were unable to find a proof in the literature.  The following argument was suggested by Steven Landsburg in a Math Overflow answer to a related question.

\begin{lem} \label{cor:carries.triple.sum}
Let $\alpha, \beta, \gamma \in \Nt$.  Then $I(\alpha, \gamma) \cup I(\alpha + \gamma, \beta) = I(\alpha, \beta) \cup I(\alpha + \beta, \gamma)$.  In particular, $I(\alpha + \gamma, \beta) \subseteq I(\alpha, \beta) \cup I(\alpha + \beta, \gamma)$.
\end{lem}
\begin{proof}
Consider the natural short exact sequence
\begin{equation} \label{equ:carry.sets.ses}
0 \to \Z / p\Z \stackrel{\iota}{\to} \Z / p^{f+1}\Z \stackrel{\pi}{\to} \Z / p^f \Z \to 0,
\end{equation}
where we view elements of $\Z / p^m \Z$ as $m$-digit numbers written in base $p$ and the maps are given by $\iota(a) = a0\cdots0$ and $\pi(a_{f} a_{f-1} \cdots a_1 a_0) = a_{f-1} \cdots a_1a_0$.  Here the $a_i$ are digits.  Consider the section $\eta$ of $\pi$ defined by $\eta(a_{f-1} \cdots a_0) = 0 a_{f-1} \cdots a_0$.  Then the class in $H^2(\Z / p^f \Z, \Z / p\Z)$ corresponding to the extension~\eqref{equ:carry.sets.ses} is represented by the $2$-cocycle $\psi: \Z / p^f \Z \times \Z / p^f \Z \to \Z / p\Z$ given by $\psi(\alpha,\beta) = \iota^{-1}(\eta(\alpha) + \eta(\beta) - \eta(\alpha+\beta))$.  It is clear that $\psi(\alpha, \beta) = 1$ if $f - 1 \in I(\alpha, \beta)$ and $\psi(\alpha, \beta) = 0$ otherwise.   Here we use the natural correspondence of sets between $\Z / p^f \Z$ and $\Nt$; of course, this correspondence does not respect the addition on each side.  Since the groups in~\eqref{equ:carry.sets.ses} are abelian, the induced $\Z / p^f \Z$-module structure on $\Z / p \Z$ is trivial.  Hence the cocycle condition on $\psi$ amounts to
$$ \psi(\alpha, \beta) + \psi(\alpha + \beta, \gamma) = \psi(\alpha, \gamma) + \psi(\alpha + \gamma, \beta).$$
Thus $f-1$ is contained in the left-hand side of the claimed equality of sets if and only it is contained in the right-hand side.      By permuting the digits, the same statement can be obtained for all $i \in \Z / f\Z$.

Observe finally that if we treat both sides of the claimed equality as multisets, then equality still holds.
\end{proof}

Now let $m \geq 2$ and let $\varepsilon_1, \dots, \varepsilon_m \in \Nt$.  Let $\mathcal{T}$ be a full rooted binary tree with $m$ leaf nodes $v_1, \dots, v_m$.  Let $v_0$ be the root of $\mathcal{T}$.  Let $V^0(\mathcal{T})$ be the set of leaf nodes and $V^2(\mathcal{T})$ the set of non-leaf nodes.  We now associate an element $\tau_v \in \Nt$ to each node $v$ of $\mathcal{T}$ and a subset $I_v \subseteq \Z / f\Z$ to each $v \in V^2(\mathcal{T})$.  If $v_i \in V^0(\mathcal{T})$, then set $\tau_{v_i} = \varepsilon_i$.  If $v \in V^2(\mathcal{T})$ is a non-leaf node with children $v^\prime$ and $v^\dpr$, then we define recursively $\tau_v = \tau_{v^\prime} + \tau_{v^\dpr}$ and $I_v = I(\tau_{v^\prime} , \tau_{v^\dpr})$.  

Observe that $\tau_{v_0} = \sum_{i = 1}^m \varepsilon_i$ is independent of the choice of tree $\mathcal{T}$.  The full binary trees $\mathcal{T}$ correspond to all possible ways of computing this sum by adding two elements of $\Nt$ at a time.  The following corollary states the number of times each column is carried while computing the sum $\tau_{v_0}$ does not depend on the way in which it is computed.

\begin{cor} \label{lem:carries.urlemma}
Let $m \geq 3$ and let $\mathcal{T}$ be a full rooted binary tree as above.  The multiset $\mathcal{I}_{\mathcal{T}} = \coprod_{v \in V^2(\mathcal{T})} I_v$ is independent of the choice of $\mathcal{T}$.
\end{cor}
\begin{proof}
If $m = 3$, then the tree $\mathcal{T}$ is unique up to permutation of the indices of the leaves, and our claim is immediate from Lemma~\ref{cor:carries.triple.sum}.  If $m > 3$, label each node $v$ of $\mathcal{T}$ by the subset of $[m]$ enumerating the leaves lying below $v$.  If $A, B, C \subset [m]$ are disjoint subsets such that the configuration on the left of the figure below appears as a subgraph of $\mathcal{T}$,
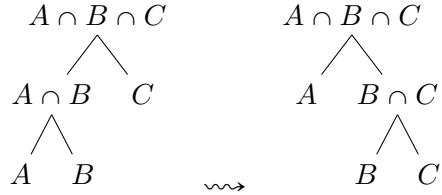
\begin{center}
\begin{tabular}{ccc}
\begin{forest}
for tree={parent anchor=south}
[{$A \cap B \cap C$}
[{$A \cap B$}
[{$A$}][{$B$}]]
[{$C$}]]
\end{forest}
&
$\rightsquigarrow$
&
\begin{forest}
for tree={parent anchor=south}
[{$A \cap B \cap C$}
[{$A$}]
[{$B \cap C$}
[{$B$}][{$C$}]]]]
\end{forest}
\end{tabular}
\end{center}
then by Lemma~\ref{cor:carries.triple.sum} we may replace it by the configuration on the right
to obtain a tree $\mathcal{T}^\prime$ satisfying $\mathcal{I}_{\mathcal{T}} = \mathcal{I}_{\mathcal{T}^\prime}$.  The claim follows, since any tree may be reached from any other by a sequence of such moves.
\end{proof}

\subsection{Submodule structure of principal series for $\mathrm{GL}_2(k)$}  \label{sec:bardoe.sin}
Every irreducible $\overline{k}$-representation of $\mathrm{GL}_2(\mathcal{O}_n)$ factors through $\mathrm{GL}_2(k)$ and is defined over $k$; such a representation is called a Serre weight.  See, for instance,~\cite[\S3.3]{Herzig/09} for a survey of the basics of Serre weights and, in particular, for the following notation.  If $\alpha \in \Z / (q-1)\Z$ and $a_2 \in \Z$, then by $F(\alpha, a_2)$ we mean the Serre weight $F(a_1, a_2)$, where $a_1 \in \Z$ is the unique integer satisfying $a_1 \in \alpha$ and $0 \leq a_1 - a_2 < q - 1$.  

\begin{dfn} \label{def:serre.weight.sjr}
Let $r \in \Z / (q-1)\Z$ and $J \subseteq \Z / f\Z$.  Define $s_{J}(r) = \sum_{i = 0}^{f-1} s_{J,i}(r)$, where
$$ s_{J,i}(r) = \begin{cases}
0 &: i \not\in J \\
r_i + 1 &: i-1 \not\in J, \, i \in J \\
r_i &: i-1, i \in J.
\end{cases}
$$
Let $\sigma_{J}(r)$ be the Serre weight $F(r - s_J(r), s_J(r))$. 
\end{dfn}
\begin{rem} \label{rem:serre.weight.data}
We leave it as a combinatorial exercise for the reader to determine the digits of $r - s_J(r)$ in base $p$ and thus establish that $\sigma_{J}(r) = \det^{s_J(r)} \tensor \bigotimes_{i = 0}^{f-1} (\mathrm{Sym}^{t_{J,i}(r)} k^2)^{(i)}$, where $\tau^{(i)}$ is the twist of the $k$-representation $\tau$ by the $i$-th power of the Frobenius automorphism of $k$ and
$$ t_{J,i}(r) = \begin{cases}
r_i &: i -1 , i \not\in J \\
r_i - 1 &: i - 1 \in J, \, i \not\in J \\
p - 2 - r_i &: i - 1 \not\in J, \, i \in J \\
p - 1 - r_i &: i - 1, i \in J.
\end{cases}
$$
\end{rem}

Recall that our basis $\mathcal{B}_{1,r}$ of $V_{1,r} = \mathrm{Ind}_{B(k)}^{\mathrm{GL}_2(k)} \chi_r$ consists of the following $q+1$ elements:
\begin{eqnarray*}
f_{j_0} & = & \sum_{\lambda \in k} \left( \begin{array}{cc} \lambda & 1 \\ 1 & 0 \end{array} \right) \tensor \lambda^{j_0}, \, j_0 \in \Nt \\
f_\infty & = & w f_0.
\end{eqnarray*}
This coincides with the basis defined just after~\cite[Theorem~2.4]{BP/12}.  

\begin{dfn} \label{def:submodules.of.tame.principal.series}
For each $I \subseteq \Z / f\Z$, define $V_I \subseteq V_{1,r}$ to be the linear span of $\{ f_\infty \} \cup \{ f_{j_0} : j_0 \in \Nt, \, I(j_0, r - j_0) \subseteq I \}$.  In particular, $V_{I} \subseteq V_J$ if and only if $I \subseteq J$.  
\end{dfn}

The work of Bardoe and Sin~\cite{BS/00} completely describes the submodule structure of $V_{1,r} = \mathrm{Ind}_{B(k)}^{\mathrm{GL}_2(k)} \chi_r$.  Recall that $r$ is defined modulo $q - 1$, so we view $r$ as an element of $\Nt \setminus \{ 0 \}$.  To provide a dictionary between the notions of~\cite{BS/00} and those of the present paper, consider the ring
$$ A = E[X,Y] / (X^q - X, Y^q - Y, (X^{q-1} - 1)(Y^{q-1} - 1)),$$
and let $A[r]$ be the subgroup of the underlying abelian group of $A$ generated by monomials whose total degree is congruent to $r$ modulo $q-1$; note that this is well-defined.  Let $\mathrm{GL}_2(k)$ act on $A[r]$ by $(gP)(X,Y) = P(aX + cY, bX + dY)$ for $g = \left( \begin{array}{cc} a & b \\ c & d \end{array} \right) \in \mathrm{GL}_2(k)$ and $P \in A[r]$.  Observing that $(1 - X^{q-1})Y^r \in A[r]$ is an eigenvector for the action of $B(k)$, with character $\chi_r$, by Frobenius reciprocity we obtain a map $\iota: V_{1,r} \to A[r]$.  It is easy to check that this map is an isomorphism and is given explicitly by
$$ \iota(f_j) = \begin{cases}
X^r &: j = 0 \\
Y^r &: j = \infty \\
- X^r Y^{q-1} &: j = q - 1 \\
\binom{q-1}{j} X^{r-j} Y^j &: j \not\in \{ 0, q - 1, \infty \}. \end{cases}
$$
If $\alpha, \beta \in \Nt$, then $X^\alpha Y^\beta \in A[r]$ is a basis monomial in the sense of~\cite{BS/00}.  The type of $X^\alpha Y^\beta$ is defined in~\cite[\S3.1]{BS/00} when $r = q-1$ and in~\cite[\S9.1]{BS/00} when $r \neq q - 1$; this is an $f$-tuple $(s_0, \dots, s_{f-1}) \in \N^{f}$.  Unwinding the definition, one verifies that $s_i = 1$ if $i - 1 \in I(\alpha, \beta)$ and $s_{i} = 0$ otherwise.  We can now translate some results of Bardoe and Sin into the following theorem.  These results are also stated in Theorems~2.4 and~2.7 of~\cite{BP/12}, with an alternative proof given in~\cite[\S4]{BP/12}; observe, again by unwinding definitions, that the subset $J \subseteq \Z / f\Z$ associated to an irreducible subquotient $\sigma$ of $V_{1,r}$ just after~\cite[Lemma~2.2]{BP/12} is the one satisfying $\sigma = \sigma_J(r)$.

\begin{thm} \label{thm:bardoe.sin.generic}
Suppose that $r \neq q - 1$.  The following statements hold.
\begin{enumerate}[label=(\alph*)]
\item
Let $j_0 \in \Nt$.  The $\mathrm{GL}_2(k)$-submodule of $V_{1,r}$ generated by $f_{j_0}$ is $V_{I(j_0, r - j_0)}$.
\item
The $\mathrm{GL}_2(k)$-submodule of $V_{1,r}$ generated by $f_{\infty}$ is $V_{\varnothing}$.
\item
If $I \subseteq \Z / f\Z$ is $r$-admissible, then the submodule $V_I$ has irreducible cosocle isomorphic to $\sigma_I(r)$.
\item
If $M \subseteq V_{1,r}$ is a submodule with irreducible cosocle, then $M = V_I$ for some $r$-admissible $I \subseteq \Z / f\Z$.  In particular, if $M \subseteq V_{1,r}$ is any submodule, then $M = \sum_{I \subseteq \Z / f \Z \atop V_I \subseteq M} V_I$.
\end{enumerate}
\end{thm}
\begin{proof}
The first two claims follow from~\cite[Theorem~5.1]{BS/00} and the third and fourth claims are Corollary~6.1 and Theorem~C of~\cite{BS/00}.  The results of \S5-6 of~\cite{BS/00} are stated in the case $r = q-1$, but it is observed in~\cite[\S9.2]{BS/00} that the proofs work in general.
\end{proof}

It remains to consider the exceptional case $r = q - 1$, i.e.~$\chi_r = \mathbf{1}$.  Observe that if $j_0 \in \Nt$, there are only two possibilities for the carry set $I(j_0, q - 1 - j_0)$: this set is $\Z / f\Z$ if $j_0 = q - 1$ and $\varnothing$ otherwise.  The following claim is obtained as above, except that we apply Theorem~A of~\cite{BS/00} rather than Theorem~C.

\begin{thm} \label{thm:bardoe.sin.exceptional}
The $\mathrm{GL}_2(k)$-module $V_{1,q-1} = \mathrm{Ind}_{B(k)}^{\mathrm{GL}_2(k)} \mathbf{1}$ is a direct sum of two Serre weights:
$$ V_{1,q-1} \simeq F(q-1,0) \oplus F(0,0).$$
The $q$-dimensional irreducible submodule $F(q-1,0)$ is $V_\varnothing$, namely the linear span of the set $\{ f_\infty \} \cup \{ f_{j_0} : j_0 \in \Nt \setminus \{ q - 1 \} \}$.  The one-dimensional submodule $F(0,0)$ is the span of $f_0 + f_\infty - f_{q-1}$.
\end{thm}

\section{The Jordan-H\"older constituents of $V_{n,r}$} \label{sec:jh.filtration}
\subsection{Jordan-H\"older constituents} \label{subsec:jh}
Using the work of Bardoe and Sin, it is simple to produce a recursive expression for the multiset $\mathrm{JH}(n, \chi)$ of Jordan-H\"older characters of $\mathrm{Ind}_{B_n}^{G_n} \chi$, for a character $\chi: B_n \to \overline{k}^\ast$; note that $\chi$ necessarily factors through $B(k)$.  Recall the character $\eta: B(k) \to k^\times$ of the introduction.  For each $n \geq 2$, define the subgroup
\begin{equation} \label{equ:btilde.general.def}
\widetilde{B}_n = (\psi_{n-1}^n)^{-1}(B(\mathcal{O}_{n-1})) = \left\{ \left( \begin{array}{cc} a &b \\ c & d \end{array} \right) \in G_n : \, c \in \mathfrak{m}^{n-1} / \mathfrak{m}^n \right\} \leq G_n.
\end{equation}
In the following, we view characters of $B(k)$ as characters of $B_n$ and $\widetilde{B}_n$ for all $n$ by inflation and omit composition with $\psi_1^n$ from the notation.

\begin{lem} \label{lem:intermediate.induction.filtration}
Let $\chi: B(k) \to k^\times$ be a character.  The multiset of Jordan-H\"older factors of $\mathrm{Ind}^{\widetilde{B}_n}_{B_n} \chi$ is $\{ \chi \eta^\alpha : \alpha \in \Nt \}$.
\end{lem}
\begin{proof}
The $q$-dimensional space $\mathrm{Ind}^{\widetilde{B}_n}_{B_n} \chi$ clearly has a basis consisting of the elements $\left( \begin{array}{cc} 1 & 0 \\ {[\lambda]} \varpi^{n-1} & 1 \end{array} \right) \tensor 1$ for $\lambda \in k$.  By the same Vandermonde argument as in the proof of Lemma~\ref{lem:bnr.basis}, an alternative basis is given by the elements
$$ m_\alpha = \sum_{\lambda \in k} \left( \begin{array}{cc} 1 & 0 \\ {[\lambda]} \varpi^{n-1} & 1 \end{array} \right) \tensor \lambda^\alpha$$
for $\alpha \in \Nt$.  Moreover, a straightforward computation shows that for $g \in \widetilde{B}_n$ as in~\eqref{equ:btilde.general.def} and $\beta \in \Nt$, we have
\begin{multline} \label{equ:action.of.g.on.mbeta}
gm_\beta = \sum_{\lambda \in k} \left( \begin{array}{cc} 1 & 0 \\ {[a_0^{-1}(d_0 \lambda + c_{n-1})]} \varpi^{n-1} & 1 \end{array} \right) \tensor \chi(g) \lambda^\beta = \\ \sum_{\lambda \in k} \left( \begin{array}{cc} 1 & 0 \\ {[\lambda]} \varpi^{n-1} & 1 \end{array} \right) \tensor \chi(g) (d_0^{-1})^\beta (a_0 \lambda - c_{n-1})^\beta,
\end{multline}
from which it is evident that the subspaces $U_\beta = \mathrm{span} \{ m_\alpha : \alpha \preceq \beta \}$ are $\widetilde{B}_n$-submodules of $\mathrm{Ind}^{\widetilde{B}_n}_{B_n} \chi$.  Moreover, we have
\begin{equation} \label{equ:u.filtration}
U_\beta / \sum_{\alpha \preceq \beta} U_\alpha \simeq \chi \eta^\beta,
\end{equation}
since the left-hand side is a one-dimensional space spanned by the image of $m_\beta$.  
\end{proof}

Now we obtain the promised recursive formula for $\mathrm{JH}(n,\chi)$.  For the basis of the recursion, recall that the sets $\mathrm{JH}(1,\chi)$ of Jordan-H\"older constituents of $\mathrm{Ind}_{B(k)}^{\mathrm{GL}_2(k)} \chi$ are specified (up to twist) in Theorems~\ref{thm:bardoe.sin.generic} and~\ref{thm:bardoe.sin.exceptional}.

\begin{cor} \label{cor:jh.constituents.general}
Suppose $n \geq 2$, and let $\chi: B(k) \to k^\times$ be a character.  Then $\mathrm{JH}(n, \chi)$ is equal to the disjoint union of multisets
$$ \mathrm{JH}(n, \chi) = \coprod_{\beta \in \Nt} \mathrm{JH}(n-1, \chi \eta^\beta).$$
\end{cor}
\begin{proof}
Set $K_n = \ker \psi_1^n$.  Then $K_n \unlhd G_n$ is a normal subgroup that acts trivially on any one-dimensional representation of $G_n$, and $K_n \leq \widetilde{B}_n$.
By Lemma~\ref{lem:intermediate.induction.filtration}, the induced module $\mathrm{Ind}^{G_n}_{B_n} \chi = \mathrm{Ind}^{G_n}_{\widetilde{B}_n} ( \mathrm{Ind}_{B_n}^{\widetilde{B}_n} \chi)$ has an exhaustive filtration with quotients $\mathrm{Ind}^{G_n}_{\widetilde{B}_n} \chi \eta^\beta \simeq \mathrm{Ind}^{G_n / K_n}_{\widetilde{B}_n / K_n} \chi \eta^\beta \simeq
 \mathrm{Ind}_{B_{n-1}}^{G_{n-1}} \chi \eta^\beta$.
\end{proof}
\begin{rem}
 A recursive description of Jordan-H\"older constituents in a similar spirit, although proved by computations of Brauer characters, is given in~\cite[Theorem~2.3]{Schein/09} for reductions modulo $p$ of cuspidal representations of $\mathrm{GL}_2(\mathcal{O}_e)$.
\end{rem}

\subsection{A filtration}
In this section we assume $n \geq 2$ and further investigate the filtration of $V_{n,r}$, indexed by $\Nt$, that arose in the proof of Corollary~\ref{cor:jh.constituents.general}.  We also compute the action of certain elements of $G_n$ on the elements $f_{\mathbf{j}} \in V_{n,r}$ of Definition~\ref{dfn:general.basis}.  These computations will be crucial later in the paper when we study a much finer filtration of $V_{2,r}$.  

\begin{dfn} \label{def:filtration}
Given $\beta \in \Nt$, define the subspace $W_\beta$ to be the linear span of 
$$B_\beta = \left\{ f_{(j_0, \dots, j_{n-1})}, f_{(\infty, j_1, \dots, j_{n-1})} : j_0, \dots, j_{n-2} \in \Nt, j_{n-1} \preceq \beta \right\}.$$
\end{dfn}

\begin{pro} \label{pro:filtration}
Let $\beta \in \Nt$.  
\begin{enumerate}[label=(\alph*)]
\item
The subspace $W_\beta$ is a $G_n$-submodule of $V_{n,r}$.
\item
There is an isomorphism of $G_n$-modules 
$$W_\beta / \sum_{\alpha \prec \beta} W_\alpha \simeq (\psi^n_1 \circ \det\nolimits^{\beta}) \tensor V_{n-1, r - 2 \beta},$$ where $G_n$ acts on the right-hand side via the projection $\psi^n_{n-1}: G_n \to G_{n-1}$.
\end{enumerate}
\end{pro}
\begin{proof}
Examining the isomorphism $\mathrm{Ind}^{G_n}_{\widetilde{B}_n} ( \mathrm{Ind}_{B_n}^{\widetilde{B}_n} \chi_r) \to \mathrm{Ind}^{G_n}_{B_n} \chi_r$ of $G_n$-modules, it is easy to see that 
that $W_\beta$ is the image of $\mathrm{Ind}^{G_n}_{\widetilde{B}_n} U_\beta$, where $U_\beta \subseteq \mathrm{Ind}_{B_n}^{\widetilde{B}_n} \chi_r$ are the subspaces defined in the proof of Lemma~\ref{lem:intermediate.induction.filtration}.  This implies the first claim, and the second follows from~\eqref{equ:u.filtration}. 
\end{proof}

The following claim, which is a corollary of the proof of Proposition~\ref{pro:filtration} rather than of its statement, will be an ingredient in some of the structural results of \S\ref{sec:main}.
\begin{cor} \label{cor:unique.maximal.submodule}
If $\beta \in \Nt$ satisfies $r - 2\beta \neq q - 1$, then $W_\beta$ has a unique maximal $G_n$-submodule.
\end{cor}
\begin{proof}
We saw in the course of proving Proposition~\ref{pro:filtration} that $W_\beta \simeq \mathrm{Ind}_{\widetilde{B}_n}^{G_n} U_\beta$.  Let $\sigma$ be an irreducible $G_n$-module.  By Frobenius reciprocity, $\mathrm{Hom}_{G_n}(W_\beta, \sigma) \simeq \mathrm{Hom}_{\widetilde{B}_n}(U_\beta, \sigma)$.  The image of any non-zero $\varphi \in \mathrm{Hom}_{\widetilde{B}_n}(U_\beta, \sigma)$ is a submodule of $\sigma$, hence is invariant under the action of $\left( \begin{array}{cc} 1 & 0 \\ {[c_{n-1}]} \varpi^{n-1} & 1 \end{array} \right)$ for all $c_{n-1} \in k$.  It is evident from~\eqref{equ:action.of.g.on.mbeta} that 
$$ \left( \begin{array}{cc} 1 & 0 \\ {[c_{n-1}]} \varpi^{n-1} & 1 \end{array} \right) m_\beta - m_\beta = \sum_{\alpha \prec \beta} (-1)^{\beta - \alpha} \binom{\beta}{\alpha} c_{n-1}^{\beta \dotdiv \alpha} m_\alpha$$
vanishes in any quotient of $U_\beta$ with this property.  By the Vandermonde argument, the only such
non-zero quotient is the one-dimensional cosocle $\chi_r \eta^\beta$.  By assumption, there is only one irreducible $G_n$-module $\sigma$ admitting $\chi_r \eta^\beta$ as a $\widetilde{B}_n$-submodule.  Only this $\sigma$ arises as an irreducible quotient of $W_\beta$.  It satisfies $\dim \mathrm{Hom}_{G_n} (W_\beta, \sigma) = 1$, implying the claim.
\end{proof}

\subsection{Some computations}
Proposition~\ref{pro:filtration} can be proved more laboriously by direct computation, applying elements of a set of generators of $G_n$ to the basis $B_\beta$ of $W_\beta$ and verifying that the image is contained in $W_\beta$.  We now state the results of some computations of this sort, as we will rely on their specializations later in the paper.

Given $\ell \in \N$, let $\mathrm{Part}(\ell)$ be the set of partitions of $\ell$, namely non-increasing sequences $\underline{m} = (m_1, \dots, m_r)$ such that $\sum_{i = 1}^r m_i = \ell$.  For a partition $\underline{m} \in \mathrm{Part}(\ell)$, let $|\underline{m}| = r$ denote the length of $\underline{m}$.  For every $j  \in \N$, let $\underline{m}(j) = | \{i \in [r] : m_i = j \} |$ be the number of times $j$ appears as a part of $\underline{m}$. 
We start with two auxiliary lemmas.

\begin{lem} \label{lem:inversion.power.series}
Let the functions $P_\ell(\lambda) : \mathcal{O}^\times \to k$ be defined by $\lambda^{-1} = \sum_{\ell = 0}^\infty [P_\ell(\lambda)] \varpi^\ell$.
If $\ell \leq e$, then 
$$ P_\ell(\lambda) = \sum_{\beta \in \mathrm{Part}_\ell} (-1)^{| \beta |} | \beta |! \, \lambda_0^{- (|\beta | + 1)}  \left( \prod_{i = 1}^\ell \frac{\lambda_i^{\beta(i)}}{\beta(i)!} \right).$$
\end{lem}
\begin{proof}
Let $E$ be a field, and let $\sum_{i = 0}^\infty a_i X^i \in E[[X]]^\times$ be an invertible formal series.  Then its inverse is $\sum_{\ell = 0}^\infty b_\ell X^\ell$, where
\begin{equation} \label{equ:formal.power.series}
b_\ell =
 \sum_{\beta \in \mathrm{Part}_\ell} (-1)^{| \beta |} | \beta |! \, a_0^{- (|\beta | + 1)}  \left( \prod_{i = 1}^\ell \frac{a_i^{\beta(i)}}{\beta(i)!} \right). 
 \end{equation}
Indeed, the sequence $\{ b_i \}$ is determined by $b_0 = a_0^{-1}$ and the recursion $\sum_{i = 0}^\ell a_{\ell - i} b_i = 0$ for all $\ell \in \N$; it is easy to verify that the sequence of~\eqref{equ:formal.power.series} satisfies these conditions.  The claim follows from this and Lemma~\ref{lem:witt.vector.summation}.
\end{proof}

By definition, $w f_{(0, j_1, \dots, j_{n-1})} = f_{(\infty, j_1, \dots, j_{n-1})}$ and $w f_{(\infty, j_1, \dots, j_{n-1})} = f_{(0, j_1, \dots, j_{n-1})}$.  So consider $\mathbf{j} = (j_0, \dots, j_{n-1}) \in \Nt^n$ with $j_0 \neq 0$.
Observing that $\lambda \in \mathcal{O}_n^\times$ if and only if $\lambda_0 \neq 0$, we have
$$ f_{\mathbf{j}} = \sum_{\lambda \in \mathcal{O}_n^\times} \left( \begin{array}{cc} \lambda & 1 \\ 1 & 0 \end{array} \right) \tensor \lambda^{\mathbf{j}} \in V_{n,r},$$
and thus
\begin{multline} \label{equ:w.action}
w f_{\mathbf{j}} = \sum_{\lambda \in \mathcal{O}_n^\times} \left( \begin{array}{cc} 1 & 0 \\ \lambda & 1 \end{array} \right) \tensor \lambda^{\mathbf{j}} = 
 \sum_{\lambda \in \mathcal{O}_n^\times} \left( \begin{array}{cc} \lambda^{-1} & 1 \\ 1 & 0 \end{array} \right)
 \left( \begin{array}{cc} \lambda & 1 \\ 0 & - \lambda^{-1} \end{array} \right) \tensor \lambda^{\mathbf{j}} = \\
 \sum_{\lambda \in \mathcal{O}_n^\times} \left( \begin{array}{cc} \lambda & 1 \\ 1 & 0 \end{array} \right) \tensor (-\lambda_0)^r \prod_{i = 0}^{n-1} P_i(\lambda)^{j_i},
\end{multline}
with $P_i(\lambda)$ as in Lemma~\ref{lem:inversion.power.series}.
Similarly, if $d \in \mathcal{O}_{n-1}$ then
\begin{equation} \label{equ:diagonal.basis}
\left( \begin{array}{cc} 1 & 0 \\ 0 & 1+d\varpi \end{array} \right) f_{\mathbf{j}} = \sum_{\lambda \in \mathcal{O}_n} \left( \begin{array}{cc} \lambda(1 + d \varpi)^{-1} & 1 \\ 1 & 0 \end{array} \right) \tensor \lambda^{\mathbf{j}} =  \sum_{\lambda \in \mathcal{O}_n} \left( \begin{array}{cc} \lambda & 1 \\ 1 & 0 \end{array} \right) \tensor \prod_{i = 0}^{n-1} Q_i(\lambda)^{j_i},
\end{equation}
where the $Q_i(\lambda)$ are defined by $\lambda(1 + d \varpi) = \sum_{i = 0}^{n-1} [Q_i(\lambda)] \varpi^i$. 

\begin{lem} \label{lem:lambda.minus.b}
Suppose that $\lambda, b \in \mathcal{O}$, and let the series of functions $R_i : \mathcal{O} \to k$ be defined by $\lambda - b = \sum_{i = 0}^\infty [R_i(\lambda)] \varpi^i$.  Then
\begin{enumerate}[label=(\alph*)]
\item
$R_i(\lambda) = \lambda_i - b_i$ for all $i \in [e-1]_0$.
\item
$R_e(\lambda) = \lambda_e - b_e + u_0 S(\lambda_0^{p^{f-1}}, - b_0^{p^{f-1}})$.
\item
For all $i \geq e$ we have $R_i(\lambda) = \lambda_i - b_i + R_i^\prime(\lambda)$, where $R^\prime_i(\lambda)$ depends only on $\lambda_0, \dots, \lambda_{i-e}$.
\end{enumerate}
\end{lem}
\begin{proof}
All the claims are immediate consequences of Lemma~\ref{lem:witt.vector.summation}.
\end{proof}

Consider $\mathbf{j} = (j_0, \dots, j_{n-1}) \in \Nt^n$ and observe that 
\begin{equation} \label{equ:upper.triangular.basis.element}
\left( \begin{array}{cc} 1  & b \\ 0 & 1 \end{array} \right) f_{\mathbf{j}} = \sum_{\lambda \in \mathcal{O}_n} \left( \begin{array}{cc} \lambda + b & 1 \\ 1 & 0 \end{array} \right) \tensor \lambda^{\mathbf{j}} = \sum_{\lambda \in \mathcal{O}_n} \left( \begin{array}{cc} \lambda & 1 \\ 1 & 0 \end{array} \right) \tensor \prod_{i = 0}^{n-1} R_i(\lambda)^{j_i},
\end{equation}
where the $R_i(\lambda)$ are as in Lemma~\ref{lem:lambda.minus.b}.  Finally, a direct computation shows that if $\mathbf{j} \in \Nt^{n-1}$ and $\lambda^{\mathbf{j}} = \prod_{i = 1}^{n-1} \lambda_i^{j_i}$, then
\begin{multline} \label{equ:b.on.infinity}
\left( \begin{array}{cc} 1 & b \\ 0 & 1 \end{array} \right) f_{(\infty, \mathbf{j})} = \sum_{\lambda \in \mathcal{O}_n^\times} 
\left( \begin{array}{cc} {\lambda^{-1}
(1 + \lambda b)} & 1 \\ 1 & 0 \end{array} \right) 
\left( \begin{array}{cc} \lambda & 1 \\ 0 & {- \lambda^{-1}} \end{array} \right) 
\tensor \lambda^{\mathbf{j}} + \\
\sum_{\lambda \in \mathfrak{m} / \mathfrak{m}^n} \left( \begin{array}{cc} 1 & 0 \\ \lambda(1 + \lambda b)^{-1} & 1 \end{array} \right) \left( \begin{array}{cc} 1 + \lambda b & b \\ 0 & (1 + \lambda b)^{-1} \end{array} \right) \tensor \lambda^{\mathbf{j}} = \\
\sum_{\nu \in b + \mathcal{O}_n^\times} \left( \begin{array}{cc} \nu & 1 \\ 1 & 0 \end{array} \right) \tensor (b_0 - \nu_0)^r \prod_{i = 1}^{n-1} P_i (R_0(\nu), \dots, R_{n-1}(\nu))^{j_i} + 
\sum_{\nu \in \mathfrak{m} / \mathfrak{m}^n} \left( \begin{array}{cc} 1 & 0 \\ \nu & 1 \end{array} \right) \tensor \prod_{i = 1}^{n-1} \tilde{R}_i(\nu)^{j_i},
\end{multline}
where $\tilde{R}_i(\nu)$ is defined by $\nu(1 - \nu b)^{-1} = \sum_{i = 1}^{n-1} [\tilde{R}_i(\nu)] \varpi^i$ for $\nu \in \mathfrak{m} / \mathfrak{m}^n$.

\section{Principal series for $\mathrm{GL}_2(\mathcal{O}/\mathfrak{m}^2)$} \label{sec:main}
We assume for the remainder of the paper that $n=2$.  If additionally $e \geq 2$, then $\mcO_2$ is isomorphic to $\Fq[\pi]/(\pi^2)$ and we can compute easily. The results also apply in the case of function fields as $\mcO$ is the power series ring $\Fq((\pi))$.  

\subsection{Preliminary computations} \label{sec:prelim.ramified}
In this section we prepare some computational lemmas that will be main ingredients in the proofs of our main results.  The first indication that life is particularly simple when $n = 2$ is that $w$ permutes the elements of our basis, up to sign.
        \begin{lem} \label{lem:w.action.on.basis}
        Let $(j_0, j_1) \in (\Nt \cup \{ \infty \}) \times \Nt$.  Then
        $$ w f_{(j_0, j_1)} = \begin{cases}
        f_{(\infty, j_1)} &: j_0 = 0 \\
        f_{(0, j_1)} &: j_0 = \infty \\
        (-1)^{r + j_1} f_{(r - 2j_1 - j_0, j_1)} &: j_0 \not\in \{ 0, \infty \}.
        \end{cases}
        $$
        \end{lem}
       \begin{proof}
       The first two lines are immediate from the definition of $f_{(\infty, j_1)}$.  To obtain the third line, observe that if $\lambda_0 \neq 0$, then the identity
       $$ ([\lambda_0] + [\lambda_1] \varpi)^{-1} \equiv [\lambda_0^{-1}] - [\lambda_0^{-2} \lambda_1] \varpi \, \mathrm{mod} \, \varpi^2$$
       holds in $\mathcal{O}$ with no restriction on the ramification of $F / \Q_p$.  Hence $P_0(\lambda) = \lambda_0^{-1}$ and $P_1(\lambda) = - \lambda_0^{-2} \lambda_1$, in the notation of the proof of Proposition~\ref{pro:filtration}, and our claim is immediate from~\eqref{equ:w.action}.
       \end{proof} 
        
        \begin{lem} \label{lem:diagonal.action.on.basis}
        Let $(j_0, j_1) \in (\Nt \cup \{ \infty \}) \times \Nt$.  The $D_2$-submodule of $V_{2,r}$ generated by $f_{(j_0, j_1)}$ is the linear span of 
        $$\begin{cases}
         \left\{ f_{(j_0 + j_1^\prime, j_1 \dotdiv j_1^\prime)} : j_1^\prime \preceq j_1 \right\} &: j_0 \in \Nt \\
         \{ f_{(\infty, j_1)} \} \cup \left\{ f_{(r-2j_1 + j_1^\prime, j_1 \dotdiv j_1^\prime)} : 0 \neq j_1^\prime \preceq j_1 \right\} &: j_0 = \infty.
         \end{cases}
         $$ 
        \end{lem}
        \begin{proof}
        Observe that if $n \leq e + 1$, then for all $\lambda, d \in \mathcal{O}$ we have
        $$ (1 + d \varpi)\lambda = \sum_{i = 0}^{n-1} \left( \lambda_i + \sum_{\ell = 0}^{i - 1} \lambda_\ell d_{i - \ell - 1} \right) \varpi^i \, \mathrm{mod} \, \varpi^n;$$
        indeed, the Witt vector summation of Lemma~\ref{lem:witt.vector.summation} does not appear.  In particular, for all $e \geq 1$ we have $Q_0(\lambda) = \lambda_0$ and $Q_1(\lambda) = \lambda_1 + \lambda_0 d_0$, in the notation of the proof of Proposition~\ref{pro:filtration}.  It follows from~\eqref{equ:diagonal.basis} that if $j_0, j_1 \in \Nt$, then in $V_{2,r}$ we have 
        $$ \left( \begin{array}{cc} 1 & 0 \\ 0 & 1 + d \varpi \end{array} \right) f_{(j_0, j_1)} = \sum_{\lambda \in \mathcal{O}_2} \left( \begin{array}{cc} \lambda & 1 \\ 1 & 0 \end{array} \right) \tensor \lambda_0^{j_0} (\lambda_1 + \lambda_0 d_0)^{j_1},$$
        and hence if $j_1^\prime \preceq j_1$ then
        \begin{equation} \label{equ:diagonal.explicit}
         - \sum_{d_0 \in k} d_0^{(q - 1) \dotdiv j_1^\prime} \left( \begin{array}{cc} 1 & 0 \\ 0 & 1 + [d_0] \varpi \end{array} \right) f_{(j_0, j_1)} = \binom{j_1}{j_1^\prime} f_{(j_0 + j_1^\prime, j_1 \dotdiv j_1^\prime)}.
         \end{equation}
         The binomial coefficients are non-zero in $E$, and by the usual Vandermonde argument the elements of the form~\eqref{equ:diagonal.explicit} linearly span the $D_2$-module generated by $f_{(j_0, j_1)}$.
         
         The remaining case $j_0 = \infty$ follows from Lemma~\ref{rem:eigenvectors}, Lemma~\ref{lem:w.action.on.basis} and the observation
         $$ \left( \begin{array}{cc} 1 & 0 \\ 0 & 1 + d \varpi \end{array} \right) f_{(\infty, j_1)} = w \left( \begin{array}{cc} 1 & 0 \\ 0 & (1 + d \varpi)^{-1} \end{array} \right) \left( \begin{array}{cc} 1 + d \varpi & 0 \\ 0 & 1 + d \varpi \end{array} \right) f_{(0, j_1)},$$
         so that $\langle D_2 \cdot f_{(\infty, j_1)} \rangle = w \langle D_2 \cdot f_{(0, j_1)} \rangle$.
        \end{proof}
        
        \begin{rem}
        The previous two lemmas already show that the analogous computations become much more involved when $n \geq 3$.  Indeed, by Lemma~\ref{lem:inversion.power.series} the expressions $P_\ell(\lambda)$ appearing in the computation of $wf_{\mathbf{j}}$ are sums of monomials indexed by partitions of $\ell$.  The clean statement of Lemma~\ref{lem:w.action.on.basis}, that $w$ permutes the elements of the basis $\mathcal{B}_{2,r}$ up to sign, arises because all elements of $[n-1]_0$ have only one partition when $n = 2$.  Of course, this is false for larger $n$, and we obtain far messier formulas.  These difficulties cannot be resolved by computing with respect to a different basis.  Moreover, the previous observation is independent of the complications arising when $e < n$ and the computations involve Witt vector polynomials as in Lemma~\ref{lem:witt.vector.summation}.  The latter features appear already when $n = 2$, in the unramified case $e = 1$; see \S\ref{sec:R.sigma} below.
        \end{rem}
        
        \begin{lem} \label{lem:lower.triangular.unramified}
        Let $(j_0, j_1) \in \Nt^2$.  The $\overline{U}_2$-submodule of $V_{2,r}$ generated by $f_{(j_0, j_1)}$ is the linear span of $\left\{ f_{(j_0 + 2j_1^\prime, j_1 \dotdiv j_1^\prime)} : j_1^\prime \preceq j_1 \right\}$.
        \end{lem}
        \begin{proof}
        A straightforward computation shows that for all $e \geq 1$ we have
        $$ \left( \begin{array}{cc} 1 & 0 \\ c \varpi & 1 \end{array} \right) f_{(j_0, j_1)} = \sum_{\lambda \in \mathcal{O}_2} \left( \begin{array}{cc} \lambda & 1 \\ 1 & 0 \end{array} \right) \tensor \lambda_0^{j_0} (\lambda_1 + \lambda_0^2 c_0)^{j_1},$$
        whence the claim follows by the usual Vandermonde argument.
        \end{proof}
        
        The next lemmas study the action of the subgroup $U_2 \leq G_2$ of upper unitriangular matrices.  Here we already see a divergence in behavior depending on whether $F / \Q_p$ is unramified ($e = 1$) or ramified ($e \geq 2$); the analogue of the following statement in the unramified case appears in Lemma~\ref{lem:submodule.generated.unramified} below.
                
        \begin{lem} \label{lem:upper.triangular.action.on.basis}
        Suppose that $e \geq 2$.  If $(j_0, j_1) \in \Nt^2$, then the $U_2$-submodule of $V_{2,r}$ generated by $f_{(j_0, j_1)}$ is the linear span of 
        $$ \left\{ f_{(j_0^\prime, j_1^\prime)} : j_0^\prime \preceq j_0, \, j_1^\prime \preceq j_1 \right\}.$$
        \end{lem}
        \begin{proof}
        Since we have assumed $e \geq 2$, it follows from Lemma~\ref{lem:lambda.minus.b} and~\eqref{equ:upper.triangular.basis.element} that
        \begin{equation} \label{equ:upper.triangular.action.general}
        \left( \begin{array}{cc} 1 & b \\ 0 & 1 \end{array} \right) f_{(j_0, j_1)} = \sum_{\lambda \in \mathcal{O}_2} \left( \begin{array}{cc} \lambda & 1 \\ 1 & 0 \end{array} \right) \tensor (\lambda_0 - b_0)^{j_0} (\lambda_1 - b_1)^{j_1}.
        \end{equation}
        If $j_0^\prime \preceq j_0$ and $j_1^\prime \preceq j_1$, then
        \begin{equation} \label{equ:upper.triangular.applied}
        \sum_{b_0, b_1 \in k} b_0^{(q-1) \dotdiv (j_0 \dotdiv j_0^\prime)} b_1^{(q-1) \dotdiv (j_1 \dotdiv j_1^\prime)} \left( \begin{array}{cc} 1 & [b_0] + [b_1] \varpi \\ 0 & 1 \end{array} \right) f_{(j_0, j_1)} = \binom{j_0}{j_0^\prime} \binom{j_1}{j_1^\prime} f_{(j_0^\prime, j_1^\prime)}.
        \end{equation}
        Elements of the form~\eqref{equ:upper.triangular.applied} span $\langle U_2 \cdot f_{(j_0, j_1)} \rangle$ by a standard Vandermonde argument.
        \end{proof}
       
        \begin{lem} \label{lem:upper.triangular.action.on.infinity}
        Suppose that $e \geq 2$.  Let $j_1 \in \Nt$.  The $U_2$-submodule of $V_{2,r}$ generated by $f_{(\infty, j_1)}$ is the linear span of 
        $$ \left\{ f_{(\infty, j_1)} \right\} \cup \left\{ f_{(j_0^\prime, j_1^\prime)} : j_0^\prime \preceq r - 2j_1, \, j_1^\prime \preceq j_1, (r-2j_1, j_1) \neq (j_0^\prime, j_1^\prime) \in \Nt^2 \right\}.$$
        \end{lem}
        \begin{proof}
        Specializing~\eqref{equ:b.on.infinity} to the case $n = 2$, we obtain
        \begin{multline} \label{equ:triangle.on.infinity.specific}
         \left( \begin{array}{cc} 1 & b \\ 0 & 1 \end{array} \right) f_{(\infty, j_1)} = \\ \sum_{\nu \in \mathcal{O}_2} \left( \begin{array}{cc} \nu & 1 \\ 1 & 0 \end{array} \right) \tensor (-1)^{r + j_1} (\nu_0 - b_0)^{r - 2j_1} (\nu_1 - b_1)^{j_1} + \sum_{\nu_1 \in k} \left( \begin{array}{cc} 1 & 0 \\ {[\nu_1] \varpi} & 1 \end{array} \right) \tensor \nu_1^{j_1},
         \end{multline}
where we observe that $r - 2j_1 \in \Nt \setminus \{ 0 \}$ by definition, and therefore the first sum only has non-zero contributions from $\nu \in b + \mathcal{O}_2^\times$ as in~\eqref{equ:b.on.infinity}.  Now consider weighted sums analogously to~\eqref{equ:upper.triangular.applied}.
        \end{proof}
 
\subsection{Types and submodules}
Consider the set $\Theta = (\Nt \cup \{ \infty \}) \times \Nt$.  Recall that a preorder is a reflexive and transitive relation.  We will now define a preorder $\preceq_{r}$ on $\Theta$.
Its equivalence classes will correspond to the Jordan-H\"older constituents of $V_{2,r}$.

\begin{dfn} \label{def:covering.relations.ramified}
    The preorder $\leq_{r}$ on the set $\Theta$ is defined to be the preorder generated by the following relations:
    \begin{enumerate}
        \item If $(j_0,j_1) \in \Nt^2$ and $m \in [f-1]_0$ satisfies $p^m \preceq j_0$, then $(j_0 \dotdiv p^m, j_1) \leq_r (j_0, j_1)$.
        \item If $(j_0,j_1) \in \Nt^2$ and $m \in [f-1]_0$ satisfies $p^m \preceq j_1$, then $(j_0, j_1 \dotdiv p^m) \leq_r (j_0, j_1)$.
        \item If $(j_0, j_1) \in \Nt^2$ and $m \in [f-1]_0$ satisfies $p^m \preceq j_1$, then $(j_0 + p^m, j_1 \dotdiv p^m) \leq_r (j_0, j_1)$.
        \item If $j_0 \not\in \{0, \infty \}$ and $j_1 \in \Nt$, then $(r - 2j_1 - j_0, j_1) \leq_r (j_0, j_1)$.
        \item If $j_1 \in \Nt$ and $m \in [f-1]_0$ satisfies $p^m \preceq j_1$, then $(r-2j_1, j_1 \dotdiv p^m) \leq_r (\infty, j_1)$.
         \item If $j_1 \in \Nt$ and $m \in [f-1]_0$ satisfies $p^m \preceq r - 2j_1$, then $((r-2j_1) \dotdiv p^m, j_1) \leq_r (\infty, j_1)$.
        \item For all $j_1 \in \Nt$, the elements $(0, j_1)$ and $(\infty, j_1)$ are equivalent, i.e. $(0, j_1) \leq_r (\infty, j_1)$ and $(\infty, j_1) \leq_r (0,j_1)$.
        \end{enumerate}
        \end{dfn}
        
        We say that $(j_0, j_1), (j_0^\prime, j_1^\prime) \in \Theta$ are equivalent if both $(j_0, j_1) \leq_r (j_0^\prime, j_1^\prime)$ and $(j_0^\prime, j_1^\prime) \leq_r (j_0, j_1)$.  Then $\leq_r$ induces a partial order, also denoted $\leq_r$, on the set $\widetilde{\Theta}$ of equivalence classes.
        
        \begin{rem} \label{rem:equivalence}
        We record some implications of the relations, for use below.
        \begin{enumerate}
        \item \label{item:w}
        Applying the fourth generating relation to $(r - 2j_1 - j_0, j_1)$, we observe that if $j_0 \not\in \{0, \infty \}$, then $(j_0, j_1)$ and $(r - 2j_1 - j_0, j_1)$ are equivalent.  
        \item
        Applying the sixth and seventh relations, we observe that if $p^m \preceq r - 2j_1$, then $(0, j_1) \leq_r ((r-2j_1) \dotdiv p^m, j_1) \leq_r (\infty, j_1) \leq_r (0,j_1)$, so all of these are equivalent.
        \item  \label{item:d}
        If $j_1 \in \Nt$ and $p^m \preceq j_1$ for some $m \in [f-1]_0$, then $(r-2j_1 + p^m, j_1 \dotdiv p^m)$ is equivalent to $(p^m, j_1 \dotdiv p^m)$ by the first observation of this remark.  Hence $(r-2j_1 + p^m, j_1 \dotdiv p^m) \preceq (\infty, j_1)$ by the third and seventh relations.
        \item \label{item:lower.triangular.no.problem}
        If $p^m \preceq j_1$, then applying the fourth, second, and fourth relations in sequence produces $(j_0 + 2p^m, j_1 \dotdiv p^m) \leq_r (r - 2j_1 - j_0, j_1 \dotdiv p^m) \leq_r (r - 2j_1 - j_0, j_1) \leq_r (j_0, j_1)$.
        \end{enumerate}
        \end{rem}

        The elements of $\widetilde{\Theta}$ will be called {\emph{types}}.  If $(j_0, j_1) \in \Theta$, then the equivalence class in $\widetilde{\Theta}$ to which it belongs is denoted by $[(j_0, j_1)]$.  
                  
       \begin{dfn} \label{def:v.theta}
       Let $\theta \in \widetilde{\Theta}$ be a type.  Set $V_\theta$ to be the subspace of $V_{2,r}$ with basis $\{ f_{(j_0, j_1)} : (j_0, j_1) \in \Theta, \, [(j_0, j_1)] \leq_r \theta \}$.
       \end{dfn}
       
       \begin{pro} \label{pro:coordinate.submodules.stable}
       Let $\theta \in \widetilde{\Theta}$.  If $e \geq 2$, then the subspace $V_\theta$ is stable under the action of $G_2 = \mathrm{GL}_2(\mathcal{O}/ \mathfrak{m}^2)$.
       \end{pro}
       \begin{proof}
       By the Bruhat decomposition, for every $n \in \N$ the group $G_n$ is generated by $B_n$ and $w$, whereas $B_n$ is, in turn, generated by the center $Z(G_n)$ and the subgroups $T$, $D_n$, and $U_n$ of Definition~\ref{def:subgroups}.
       Hence it suffices to verify that $V_\theta$ is stable under the action of the subgroups $Z(G_2), T, D_2, U_2$ and the element $w$.  Since each of the basis elements in Definition~\ref{def:v.theta} is an eigenvector for the actions of $Z(G_2)$ and of $T$ by Lemma~\ref{rem:eigenvectors}, it is clear that these subgroups preserve $V_\theta$.  By Lemma~\ref{lem:w.action.on.basis}, together with Remark~\ref{rem:equivalence}\eqref{item:w} and the last generating relation, the action of $w$ sends every basis element $f_{(j_0, j_1)}$ to a scalar multiple of a basis element of the same type.  Similarly, it is immediate from Lemma~\ref{lem:diagonal.action.on.basis}, Remark~\ref{rem:equivalence}\eqref{item:d}, and the third generating relation that the subgroup $D_2$ preserves the subspace $V_\theta$.
       
       It remains to verify that the subgroup $U_2$ preserves $V_\theta$.  If $j_0 \in \Nt$, then by Lemma~\ref{lem:upper.triangular.action.on.basis} and the first two relations, the $U_2$-module generated by $f_{(j_0, j_1)}$ is spanned by basis elements of type equal to or smaller than $[(j_0, j_1)]$.  The analogous claim for the $U_2$-module generated by $f_{(\infty, j_1)}$ follows from Lemma~\ref{lem:upper.triangular.action.on.infinity} and the fifth and sixth generating relations.   
       
       While this follows from the above, we can also observe directly from Remark~\ref{rem:equivalence}\eqref{item:lower.triangular.no.problem} that $V_\theta$ is stable under the action of $\overline{U}_2$ described in Lemma~\ref{lem:lower.triangular.unramified}.
       \end{proof}

\begin{thm} \label{thm:coordinate.submodules}
Let $(j_0, j_1) \in \Theta$.  If $e \geq 2$, then the $\mathrm{GL}_2(\mathcal{O}/ \mathfrak{m}^2)$-submodule of $V_{2,r}$ generated by $f_{(j_0, j_1)}$ is $V_{[(j_0, j_1)]}$.
\end{thm}
\begin{proof}
Since $f_{(j_0, j_1)} \in V_{[j_0, j_1]}$ by definition, it suffices to prove that $V_{[j_0, j_1]} \subseteq \langle G_2 \cdot f_{(j_0, j_1)} \rangle$.  We argue by induction on the type $[(j_0, j_1)]$, with respect to the partial order $\leq_r$ on the set $\widetilde{\Theta}$ of types.  If $j_1 = 0$, then we are working inside the submodule $W_0$ of Definition~\ref{def:filtration}, which is isomorphic to $\mathrm{Ind}_{B(k)}^{\mathrm{GL}_2(k)} \chi_r$ by Proposition~\ref{pro:filtration}.  In this case our claim was proved by Bardoe and Sin; see Theorems~\ref{thm:bardoe.sin.generic} and~\ref{thm:bardoe.sin.exceptional}.

In general, by induction it suffices to show that if $(j_0^\prime, j_1^\prime) \leq_r (j_0, j_1)$ by one of the generating relations for the preorder $\leq_r$ on $\Theta$, then $f_{(j_0^\prime, j_1^\prime)} \in \langle G_2 \cdot f_{(j_0, j_1)} \rangle$.  For the first and second relations, this follows from Lemma~\ref{lem:upper.triangular.action.on.basis}.  The claim for the third relation is implied by Lemma~\ref{lem:diagonal.action.on.basis} and for the fourth and seventh by Lemma~\ref{lem:w.action.on.basis}.  For the fifth and sixth covering relations, this is Lemma~\ref{lem:upper.triangular.action.on.infinity}.
\end{proof}
          
        \subsection{An alternative view of types}  
       The carry sets defined in Lemma~\ref{lem:carry.set.def} provide an alternative way to describe types that is often more convenient to work with than the equivalence classes in terms of which they were defined.  Recall the notion of a $\gamma$-admissible set $J \subseteq \Z / f\Z$ of Definition~\ref{def:admissible.set}; an explicit characterization of such sets is given in Proposition~\ref{pro:characterize.admissible}.
        
        \begin{pro} \label{cor:types.as.pairs}
        There is a well-defined injection $\Upsilon: \widetilde{\Theta} \to \mathcal{P}(\Z / f \Z) \times \Nt$ given by
        $$ [(j_0, j_1)] \mapsto (I(j_0, r - 2j_1 - j_0), j_1).$$
        The image of $\Upsilon$ consists of the pairs $(J, j_1)$ such that $J$ is $(r - 2j_1)$-admissible.
        \end{pro}
        \begin{proof}
        It is clear from the generating relations that if $(j_0^\prime, j_1^\prime) \leq_r (j_0, j_1)$, then $j_1^\prime \preceq j_1$.  Thus the second component is constant for all elements $\theta = (j_0, j_1)$ of a type.   Observe from their definition that the submodules $V_\theta$ are distinct for distinct types $\theta \in \widetilde{\Theta}$.  By Theorem~\ref{thm:coordinate.submodules}, the elements $(j_0, j_1), (j_0^\prime, j_1) \in \Theta$ belong to the same type if and only if  $f_{(j_0,j_1)}$ and $f_{(j_0^\prime, j_1)}$ generate the same submodule of $V_{2,r}$.  In particular they generate the same submodule of $W_{j_1} / \sum_{\ell \prec j_1} W_{\ell}$, which is isomorphic to $\det^{j_1} \tensor V_{1, r - 2j_1}$ by Proposition~\ref{pro:filtration}.  The structure of $V_{1,r-2j_1}$ is described by Theorems~\ref{thm:bardoe.sin.generic} and~\ref{thm:bardoe.sin.exceptional}, from which we conclude that the carry sets $I(j_0, r - 2j_1 - j_0)$ and $I(j_0^\prime, r - 2j_1 - j_0^\prime)$ must coincide.  Hence the claimed map is well-defined.
        
        To show the map is injective, it suffices to show that $V_{[(j_0, j_1)]}$ is determined by its image in the quotient $W_{j_1} / \sum_{\ell \preceq j_1} W_\ell$.  Indeed, suppose that $V_{[(j_0^\prime, j_1^\prime)]}$ has the same image.  Then $V_{[(j_0^\prime, j_1^\prime)]}$ contains an element of the form $f_{(j_0, j_1)} + h$, where $h \in \sum_{\ell \preceq j_1} W_\ell$.  Since $V_{[(j_0^\prime, j_1^\prime)]}$ is the linear span of a subset of $\mathcal{B}_{2,r}$, it must contain $f_{(j_0, j_1)}$ and hence, by Theorem~\ref{thm:coordinate.submodules}, we have $V_{[(j_0, j_1)]} \subseteq V_{[(j_0^\prime, j_1^\prime)]}$.  Reversing the roles of the two types, we obtain the opposite inclusion.  Hence $V_{[(j_0, j_1)]} = V_{[(j_0^\prime, j_1^\prime)]}$ as claimed.
        
        The claim regarding the image of the map $\Upsilon$ is immediate from the definition of $(r-2j_1)$-admissibility.
        \end{proof}

In light of the previous claim, we will often write types in the form $\theta = (I, \gamma)$, where $I \subseteq \Z / f\Z$ and $\gamma \in \Nt$.  Note that the map of Proposition~\ref{cor:types.as.pairs} is in general not surjective.

\begin{rem} \label{rem:order.is.inclusion}
We collect several observations that follow readily from Proposition~\ref{cor:types.as.pairs} and its proof.
\begin{enumerate}
\item
If $(I,\gamma)$ and $(I^\prime, \gamma)$ are types with the same second component, then $(I^\prime, \gamma) \leq_r (I,\gamma)$ if and only if $I^\prime \subseteq I$.
\item \label{item:all.and.nothing}
The types $(\varnothing, \gamma)$ and $(\Z / f\Z, \gamma)$ lie in the image of $\Upsilon$ for any $\gamma \in \Nt$.
\item
If $r - 2j_1 \equiv 0 \, \mathrm{mod} \, q - 1$, then $\Upsilon([(q-1, j_1)]) = (\Z / f\Z , j_1)$ and $\Upsilon([(j_0, j_1)]) = (\varnothing, j_1)$ for $j_0 \neq q - 1$.
\end{enumerate}
\end{rem}

We can strengthen the previous remark by producing a closed-form definition of the partial order on types that is convenient to use in practice.  Our proof of the following statement is completely elementary but long and rather tedious and independent of the rest of the paper, apart from the definition of carry sets and their properties in~\S\ref{sec:carry.sets} above.  For these reasons we relegate it to \S\ref{sec:equality.of.orders} below.

\begin{pro} \label{pro:equality.of.orders}
Let $(I, \gamma)$ and $(I^\prime, \gamma^\prime)$ be two types.  Then $(I^\prime, \gamma^\prime) \leq_r (I,\gamma)$ if and only if the following two conditions are satisfied:
\begin{itemize}
\item
$\gamma^\prime \preceq \gamma$;
\item
$I^\prime \subseteq I \cup I(r - 2\gamma, 2(\gamma \dotdiv \gamma^\prime)) \cup I(\gamma \dotdiv \gamma^\prime, \gamma \dotdiv \gamma^\prime)$.
\end{itemize}
\end{pro}

\subsection{Extensions of Serre weights as subquotients of principal series}
Throughout this section, and for the remainder of \S\ref{sec:main}, we assume $e \geq 2$.
As demonstrated in Proposition~\ref{pro:infinite.submodule.structure}, the principal series $V_{2,r}$ may, in general, have infinitely many submodules.  However, we can elucidate the structure of the family of submodules $\{ V_\theta: \theta \in \widetilde{\Theta} \}$.  Recall the Serre weight $\sigma_J(r)$ of Definition~\ref{def:serre.weight.sjr}.

\begin{dfn}
Given a type $(I, \gamma)$, set $L(I,\gamma) = V_{(I,\gamma)} / \sum_{(I^\prime, \gamma^\prime) <_r (I,\gamma)} V_{(I^\prime, \gamma^\prime)}$.  
\end{dfn}

\begin{lem} \label{pro:head}
Let $(I, \gamma)$ be a type.  If $e \geq 2$, then the $\mathrm{GL}_2(\mathcal{O}_2)$-module $L(I,\gamma)$ is irreducible and isomorphic to $\det^\gamma \tensor \sigma_I (r - 2\gamma)$.
\end{lem}
\begin{proof}
By definition, $L(I,\gamma)$ is a subquotient of $W_\gamma / \sum_{\alpha \preceq \gamma} W_\alpha$, which is isomorphic to $\det^\gamma \tensor V_{1, r - 2 \gamma}$ by Proposition~\ref{pro:filtration}.  Moreover, the image of $V_{(I,\gamma)}$ in this quotient is $V_I \subseteq \det^\gamma \tensor V_{1, r - 2 \gamma}$.  By Remark~\ref{rem:order.is.inclusion}, inside this quotient we can compute $L(I,\gamma) = V_I / \sum_{I^\prime \subset I} V_{I^\prime}$.  If $r - 2\gamma \neq q -1 \in \Nt \setminus \{ 0 \}$, then $L(I,\gamma) = \det^\gamma \tensor \sigma_I(r - 2 \gamma)$ by Theorem~\ref{thm:bardoe.sin.generic}.  In the exceptional case $r - 2 \gamma = q - 1$, we find $V_{\varnothing} = F(\gamma + q - 1, \gamma)$ by Theorem~\ref{thm:bardoe.sin.exceptional}, whereas $V_{\Z / f\Z}$ is the entire space $\det^\gamma \tensor V_{1, r - 2 \gamma}$.  Thus $V_{\Z / f\Z} / V_{\varnothing} \simeq F(\gamma, \gamma)$.  Since $\sigma_{\Z / f \Z}(r - 2 \gamma) = F((q-1) - (q-1), q - 1) = F(q-1, q-1) = F(0,0)$ by Definition~\ref{def:serre.weight.sjr}, the claim holds in this case also.
\end{proof}

Now, assume that the types $(I^\prime, \gamma^\prime) <_r (I,\gamma)$ are adjacent, i.e. there is no other type between them.  Consider the extension
$$ \mathcal{E} = V_{(I,\gamma)} / \sum_{(I^{\prime \prime}, \gamma^{\prime \prime}) \leq_r (I, \gamma) \atop (I^{\prime \prime}, \gamma^{\prime \prime}) \neq (I^\prime, \gamma^\prime)} V_{(I^{\prime \prime}, \gamma^{\prime \prime})},$$
so that there is a short exact sequence
    \begin{equation} \label{equ:extension.ses}
         0\to L(I',\gamma')\to \mathcal{E}\to L(I,\gamma)\to 0
    \end{equation}   
    As $L(I^\prime, \gamma^\prime)$ and $L(I,\gamma)$ are irreducible by Lemma~\ref{pro:head}, the extension $\mathcal{E}$ is of length two.
    
\begin{pro} \label{pro:adjacent.types}
Suppose $e \geq 2$, and consider adjacent types $(I^\prime, \gamma^\prime) <_r (I,\gamma)$.  The corresponding short exact sequence~\eqref{equ:extension.ses} splits if and only if $\gamma^\prime = \gamma$ and $r - 2 \gamma = q - 1$.
\end{pro}
\begin{proof}
By assumption, there exist pairs $(j_0^\prime, j_1^\prime), (j_0, j_1) \in \Theta$ representing the classes $(I^\prime, \gamma^\prime), (I, \gamma)$, respectively, such that $(j_0^\prime, j_1^\prime) <_r (j_0, j_1)$ by one of the generating relations of Definition~\ref{def:covering.relations.ramified}.  We need not concern ourselves with the fourth, sixth (see Remark~\ref{rem:equivalence}), or seventh covering relations, as these are equivalences.  We consider the remaining relations in sequence.

If $(j_0^\prime, j_1^\prime), (j_0, j_1) \in \Theta$ by the first generating relation, then $j_1^\prime = j_1$ and $\mathcal{E}$ is a subquotient of the tame principal series $W_{j_1} / \sum_{\ell \prec j_1} W_\ell \simeq \det^{j_1} \tensor V_{1, r - 2j_1}$.  In this case the sequence~\eqref{equ:extension.ses} splits by Theorem~\ref{thm:bardoe.sin.exceptional} when $r - 2j_1 = q - 1$ and is non-split otherwise by~\cite[Theorem~6.1]{BS/00}.

In the case of the second relation, we have $(j_0^\prime, j_1^\prime) = (j_0, j_1 \dotdiv p^m)$ for $p^m \preceq j_1$.  Then
we see from~\eqref{equ:upper.triangular.action.general} that 
\begin{equation} \label{equ:moving.vertically.down}
\sum_{b_1 \in k} b_1^{(q-1) \dotdiv p^m} \left( \begin{array}{cc} 1 & [b_1] \varpi \\ 0 & 1 \end{array} \right) f_{(j_0, j_1)} = (j_1)_m f_{(j_0^\prime, j_1^\prime)} \neq 0.
\end{equation}
If~\eqref{equ:extension.ses} were to split, then $\mathcal{E}$ would be isomorphic to a direct sum of Serre weights, and the action of $\mathrm{GL}_2(\mathcal{O} / \mathfrak{m}^2)$ on $\mathcal{E}$ would factor through $\mathrm{GL}_2(k)$.  In that case, the left-hand side of the previous displayed formula would be $\sum_{b_1 \in k} b_1^{(q-1) \dotdiv p^m} f_{(j_0, j_1)} = 0$.  This contradiction proves that~\eqref{equ:extension.ses} does not split.

The arguments for the third and fifth relations are similar.  Observe by~\eqref{equ:diagonal.explicit} and~\eqref{equ:triangle.on.infinity.specific}, respectively, that 
\begin{align*}
 - \sum_{d_0 \in k} d_0^{(q-1) \dotdiv p^m} \left( \begin{array}{cc} 1 & 0 \\ 0 & 1 + [d_0] \varpi \end{array} \right) f_{(j_0, j_1)} & = & (j_1)_m f_{(j_0^\prime, j_1^\prime)} & \neq 0 \\
 (-1)^{r + j_1} \sum_{b_1 \in k} b_1^{(q-1) \dotdiv p^m} \left( \begin{array}{cc} 1 & [b_1] \varpi \\ 0 & 1 \end{array} \right) f_{(\infty, j_1)} & = & (j_1)_m f_{(r - 2j_1, j_1 \dotdiv p^m)} & \neq 0
 \end{align*}
and again the left-hand side of each expression would vanish if~\eqref{equ:extension.ses} split and the action of $\mathrm{GL}_2(\mathcal{O} / \mathfrak{m}^2)$ on $\mathcal{E}$ factored through $\mathrm{GL}_2(k)$.
\end{proof}

The following claim is a modification of the previous one, accounting for the fact that $V_{(\Z / f\Z, \gamma)}$ does not have irreducible cosocle if $r - 2\gamma = q - 1$.

\begin{lem} \label{lem:nonsplit.extension.exceptional}
Suppose that $e \geq 2$, that $\gamma \in \Nt$ satisfies $r - 2\gamma = q-1$, and that $m \in [f-1]_0$ is such that $p^m \preceq \gamma$.  Let $M \subset V_{2,r}$ be a $\mathrm{GL}_2(\mathcal{O} / \mathfrak{m}^2)$-submodule with one-dimensional image in the quotient $V_{2,r} / \sum_{\beta \prec \gamma} W_\beta$.  
\begin{enumerate}[label=(\alph*)]
\item \label{item:containment.exceptional}
The containment $W_{\gamma \dotdiv p^m} \subset M$ holds.
\item
Let $N = \sum_{\beta \prec \gamma, \, J \subseteq \Z / f\Z \atop (J, \beta) \neq (\Z / f\Z, \gamma \dotdiv p^m)} V_{(J, \beta)}$.  The subquotient $\mathcal{E} = M/N$
admits a short exact sequence $0 \to L(\Z / f\Z, \gamma \dotdiv p^m) \to \mathcal{E} \to L(\Z / f\Z, \gamma) \to 0$ that does not split.
\end{enumerate}
\end{lem}
\begin{proof}
Combining Theorem~\ref{thm:bardoe.sin.exceptional} with Proposition~\ref{pro:filtration}, note that $M$ contains an element of the form $h = f_{(0,\gamma)} + f_{(\infty, \gamma)} - f_{(q-1, \gamma)} + z$, for some $z \in \sum_{\beta \prec \gamma} W_\beta$.  Observe by~\eqref{equ:upper.triangular.applied} that
\begin{equation} \label{equ:extension.exceptional}
 \sum_{b_1 \in k} b_1^{(q-1) \dotdiv p^m} \left( \begin{array}{cc} 1 & [b_1] \varpi \\ 0 & 1 \end{array} \right) h  =  \gamma_m (f_{(q-1, \gamma \dotdiv p^m)} - f_{(0, \gamma \dotdiv p^m)} - f_{(\infty, \gamma \dotdiv p^m)}) + z^\prime \in M,
\end{equation}
where $z^\prime \in \sum_{\beta \prec \gamma \dotdiv p^m} W_\beta$ and $\gamma_m \neq 0$.  By Theorem~\ref{thm:coordinate.submodules}, the element $f_{(q-1,\gamma \dotdiv p^m)}$ generates $V_{[(q-1, \gamma \dotdiv p^m)]} = W_{\gamma \dotdiv p^m}$, whereas $\gamma_m(f_{(0,\gamma \dotdiv p^m)} + f_{(\infty, \gamma \dotdiv p^m)}) - z^\prime$ is contained in a proper submodule of $W_{\gamma \dotdiv p^m}$.  Since $W_{\gamma \dotdiv p^m}$ has a unique maximal submodule by Corollary~\ref{cor:unique.maximal.submodule}, the right-hand side of~\eqref{equ:extension.exceptional} generates $W_{\gamma \dotdiv p^m}$.  This proves the first claim.  

Now consider the restriction to $\mathcal{E}$ of the projection $W_\gamma / N \to W_\gamma / \sum_{\beta \prec \gamma} W_\beta$.  The image is the one-dimensional Serre weight $L(\Z / f\Z, \gamma)$ by assumption, and the kernel is $(\sum_{\beta \prec \gamma} W_\beta)/N \simeq L(\Z / f\Z, \gamma \dotdiv p^m)$ by part~\ref{item:containment.exceptional}.  Hence $\mathcal{E}$ admits a short exact sequence as claimed.  By~\eqref{equ:extension.exceptional}, the extension $\mathcal{E}$ is not semisimple, so it cannot split.  
\end{proof}

\subsection{Socle and radical filtrations}
Let $\Gamma(\leq_r)$ be the graph associated to the partial order $\leq_r$: the set of its vertices is the set $\widetilde{\Theta}$ of types (identified, as usual, with the image of $\Upsilon$), and there is a directed edge from $\theta$ to $\theta^\prime$ if the types $\theta^\prime <_r \theta$ are adjacent.  Let $\Gamma_r$ be the graph obtained from $\Gamma(\leq_r)$ by the following procedure: for all $\gamma \in \Nt$ such that $r - 2 \gamma = q - 1$, we remove the edge $(\Z / f\Z, \gamma) \to (\varnothing, \gamma)$ and add the edge $(\Z / f\Z, \gamma) \to (\Z / f\Z, \gamma \dotdiv p^m)$ for all $m \in [f-1]_0$ such that $p^j \preceq \gamma$.

\begin{lem} \label{lem:paths}
Let $\theta \in \Upsilon(\widetilde{\Theta})$.  Then $\Gamma_r$ contains a path from the maximal vertex $(\Z / f\Z, q - 1)$ to $\theta$, and a path from $\theta$ to the minimal vertex $(\varnothing, 0)$.
\end{lem}
\begin{proof}
If $(I^\dpr, \delta) <_r (I^\prime, \delta)$ are vertices of $\Gamma_r$ and $r - 2\delta \neq q - 1$, then by~\cite[Corollary~4.1]{BS/00} (see also Lemma~\ref{lem:bardoe.sin.chains} below) there is a path 
$$ (I^\prime, \delta) = (I_0, \delta) \to (I_1, \delta) \to \cdots \to (I_s, \delta) = (I^\dpr, \delta),$$
where for each $i \in [s]$ the set $I_i$ is obtained from $I_{i-1}$ by removing exactly one element.

Let $\theta = (I, \gamma) \in \Upsilon(\widetilde{\Theta})$.  Fix a sequence 
$$0 = \gamma_0 \prec \gamma_1 \prec \cdots \prec \gamma_{\sum_{i = 0}^{f-1} \gamma_i} = \gamma \prec \dots \prec \gamma_{f(p-1)} = q - 1,$$ 
where for each $i \in [f(p-1)]$ there exists $j_i \in \Z / f\Z$ such that $\gamma_i \dotdiv \gamma_{i-1} = p^{j_i}$.  Then
\begin{equation} \label{equ:coarse.sequence}
(\varnothing, \gamma_0) \leq_r (\varnothing, \gamma_1) \leq_r \cdots \leq_r (\varnothing, \gamma) \leq_r (I, \gamma) \leq_r (\Z / f\Z, \gamma) \leq_r \cdots \leq_r (\Z / f\Z, \gamma_{f(p-1)})
\end{equation}
is a sequence of vertices of $\Gamma_r$ by Remark~\ref{rem:order.is.inclusion}\eqref{item:all.and.nothing}.  If $\delta \in \Nt$ satisfies $r - 2\delta = q - 1$ and $p^m \preceq \delta$, then $\Gamma_r$ contains an edge $(\Z / f\Z, \delta) \to (\Z / f\Z, \delta \dotdiv p^m)$ by definition and an edge $(\varnothing, \delta) \to (\Z / f\Z, \delta \dotdiv p^m)$ by Proposition~\ref{pro:equality.of.orders}.  By this observation and the first paragraph of this proof, the sequence~\eqref{equ:coarse.sequence}
may be refined to a path from $(\Z / f\Z, q - 1)$ to $(\varnothing, 0)$ passing through $\theta$.
\end{proof}

For every vertex $\theta$ of $\Gamma_r$, set $\ell^\uparrow(\theta)$ to be the length of the longest path from $(\Z / f\Z, q - 1)$ to $\theta$, and $\ell_\downarrow (\theta)$ to be the length of the longest path from $\theta$ to $(\varnothing,0)$.  Let $L = \ell^\uparrow(\varnothing, 0) = \ell_\downarrow ( \Z / f \Z, q - 1)$.  These definitions make sense by Lemma~\ref{lem:paths}.  Then the socle and radical filtrations of $V_{2,r}$ can be determined from Proposition~\ref{pro:adjacent.types} and Lemma~\ref{lem:nonsplit.extension.exceptional}.  Recall that if $\theta = (I, \gamma)$, then $L(\theta)$ is the Serre weight $\det^\gamma \tensor \sigma_I ( r - 2 \gamma)$.

To fix definitions, recall that the socle filtration of $V_{2,r}$ is defined recursively as follows: $\mathrm{soc}_{-1}(V_{2,r}) = 0$, whereas $\mathrm{soc}_i(V_{2,r})$ is the pre-image in $V_{2,r}$ of $\mathrm{soc} (V_{2,r} / \mathrm{soc}_{i-1} (V_{2,r}))$ for all $i \geq 0$.  Similarly, the radical filtration is determined by $\mathrm{rad}_{-1} (V_{2,r}) = V_{2,r}$ and the recursion $\mathrm{rad}_i (V_{2,r}) = \mathrm{rad} (\mathrm{rad}_{i-1}(V_{2,r}))$ for $i \geq 0$.  

\begin{cor}
Suppose $e \geq 2$.
The socle and radical filtrations of $V_{2,r}$ have length $L$, and their graded pieces are given by
$$ \begin{array}{lcr}
 \mathrm{soc}_i (V_{2,r}) / \mathrm{soc}_{i-1}(V_{2,r}) \simeq \bigoplus_{\theta \in \Upsilon(\widetilde{\Theta}) \atop \ell_\downarrow(\theta) = i} L(\theta), & &
 \mathrm{rad}_{i-1} (V_{2,r}) / \mathrm{rad}_{i} (V_{2,r}) \simeq \bigoplus_{\theta \in \Upsilon(\widetilde{\Theta}) \atop \ell^\uparrow(\theta) = i} L(\theta).
 \end{array}
 $$
\end{cor}

\begin{rem}
For the tame principal series $V_{1,r}$, the elements of the radical filtration are just the elements of the socle filtration in reverse~\cite[Corollary~7.1]{BS/00}.  This is false for $V_{2,r}$.  For instance, suppose that $\gamma \in \Nt$ satisfies $r - 2 \gamma = q - 1$.  Let $\theta^\prime = (I^\prime, \gamma^\prime)$ be any type with $\gamma^\prime \prec \gamma$.  Since $I(r - 2\gamma, 2(\gamma \dotdiv \gamma^\prime)) = \Z / f\Z$, we see by Proposition~\ref{pro:equality.of.orders} that $\theta^\prime <_r (\varnothing, \gamma)$.  Thus a vertex of $\Gamma_r$ lies below $(\varnothing, \gamma)$ if and only if it lies below $(\Z / f\Z, \gamma)$.  Hence $\ell_\downarrow (\varnothing, \gamma) = \ell_\downarrow(\Z / f\Z, \gamma)$, and these two Serre weights appear in the same graded piece of the socle filtration.  

This need not be true for the radical filtration.  For instance, if $r = q - 3$, then the only vertex above $(\Z / f\Z, q - 2)$ is $(\Z / f\Z, q - 1)$, so $\ell^\uparrow(\Z / f\Z, q - 2) = 1$.  By contrast, $(\varnothing, q - 2) <_r (\varnothing, q - 1)$, so $\ell^\uparrow(\varnothing, q - 2) \geq 2$.  Analogous examples are readily produced for other values of $r$.
\end{rem}

\subsection{Infinite submodule structure}
In this section we will show that if $E$ is an infinite field, then $V_{2,r}$ will, in general, contain infinitely many submodules.  An important exception occurs when $F/\Q_p$ is totally ramified; see Theorem~\ref{cor:finite.submodule.structure} below.  We will make use of the following classical criterion.

\begin{lem} \label{lem:jans.camillo}
Let $E$ be an infinite field, let $R$ be an $E$-algebra, and let $M$ be a left $R$-module.  Then $M$ has finitely many $R$-submodules if and only if the socle of $M/N$ is multiplicity-free for every $R$-submodule $N \subseteq M$.
\end{lem}
\begin{proof}
Combine~\cite[Lemma~1.2]{Jans/57} with~\cite[Theorem~1]{Camillo/75}.
\end{proof} 
 
The genericity hypothesis in the next claim could be removed with a suitable adjustment of the proof that would make it more involved and less transparent.
 
\begin{pro} \label{pro:infinite.submodule.structure}
Suppose that $e \geq 2$ and $f \geq 2$, i.e.~the extension $F / \Q_p$ is ramified but not totally ramified.  Let $r \in \Z / (q-1)\Z$ such that $1 \leq r_i \leq p - 2$ for all $i \in \Z / f\Z$.  If $E$ is an infinite field, then $V_{2,r}$ has infinitely many $\mathrm{GL}_2(\mathcal{O} / \mathfrak{m}^2)$-submodules.
\end{pro}
\begin{proof}
By Lemma~\ref{lem:jans.camillo} it suffices to exhibit a quotient of $V_{2,r}$ whose socle fails to be multiplicity-free.
Let $\varnothing \neq J \subseteq \Z / f \Z$ and set
$ \gamma(J) = \sum_{i = 0}^{f-1} \gamma_i(J) p^i \in \Z / (q - 1) \Z = \Nt \setminus \{ 0 \}$, where
$$
\gamma_i (J) = \begin{cases}
0 &: i -1 , i \not\in J \\
- 1 &: i - 1 \in J, \, i \not\in J \\
r_i + 1 &: i - 1 \not\in J, \, i \in J \\
r_i &: i - 1, i \in J.
\end{cases}
$$
Under our assumptions on $r$, a simple calculation and substitution in Definition~\ref{def:serre.weight.sjr} verify that
$L(J, \gamma(J)) = \det^{\gamma(J)} \tensor \sigma_J(r - 2 \gamma(J)) = F(r,0) = L(\varnothing, 0)$, where the first equality is Lemma~\ref{pro:head}.
Consider the singleton sets $J = \{ j \}$ for $j \in \Z / f\Z$.  We find that 
$$\gamma(\{ j \}) = \sum_{i \in [f-1]_0 \setminus \{ j \} } (p-1) p^i + r_j p^j.$$
Let $j, j^\prime$ be distinct elements of $\Z / f\Z$, which exist since we assumed $f \geq 2$.  Then $\gamma(\{ j \})$ and $\gamma(\{ j^\prime \})$ are incomparable with respect to the partial order $\preceq$ on $\Nt$.  Now put 
$$M = \sum_{(I, \gamma) <_r (\{ j \}, \gamma(\{ j \})} V_{(I, \gamma)} + \sum_{(I, \gamma) <_r (\{ j^\prime \}, \gamma(\{ j^\prime \})} V_{(I, \gamma)}.$$
It is clear from Lemma~\ref{pro:head} that $L(\{ j \}, \gamma(\{ j \}) \oplus L(\{ j^\prime \}, \gamma(\{ j^\prime \}) \simeq F(r,0) \oplus F(r,0)$ injects into $V_{2,r} / M$, and this completes the proof.
\end{proof}

\subsection{An example: the totally ramified case}
To illustrate our results, in this section we discuss in detail the submodule structure of $V_{2,r}$ when $F / \Q_p$ is totally ramified.  In contrast to Proposition~\ref{pro:infinite.submodule.structure}, we will find that in this case $V_{2,r}$ always has finite submodule structure, and we will be able to give an explicit description of all submodules with irreducible cosocle; see Theorem~\ref{cor:finite.submodule.structure} below.

For the rest of this section, $k = \F_p$ and $e \geq 2$.  We assume $p \geq 5$ to avoid degenerate cases.  In this case the distinct Serre weights are precisely
$$ F(r+s,s) = \det\nolimits^s \tensor \Sym^r E^2$$
for $s \in \Z / (p-1)\Z$ and $r \in [p-1]_0$.  We usually drop $E^2$ from the notation for brevity.

There is a natural correspondence between elements of $\Nt$ and integers in the interval $[p-1]_0$.  Write $(\alpha) \in [p-1]_0$ for the integer corresponding to $\alpha \in \Nt$.  In this case, the order $\preceq$ on the monoid $\Nt$ is a total linear order: $0 \prec 1 \prec \cdots \prec p - 1$.  The filtration of \S\ref{sec:jh.filtration} is therefore a flag of subspaces $0 \subset W_1 \subset \cdots \subset W_{p-1} = V_{2,r}$, and its quotients are the tame principal series $W_i / W_{i-1} \simeq \det^i \tensor V_{1,r - 2i}$.  If $r - 2i \neq q - 1$, then this is
the unique non-split extension
$$ 0 \to \det\nolimits^i \tensor \mathrm{Sym}^{( r-2i ) } E^2 \to W_i / W_{i-1} \to \det\nolimits^{r-i} \tensor \mathrm{Sym}^{ ( p-1-r+2i )} E^2 \to 0.$$
The short exact sequence above is a special case of Theorem~\ref{thm:bardoe.sin.generic}; the uniqueness of the extension is~\cite[Corollary~5.6]{BP/12}.
If $r - 2i = p - 1 \in \Nt$, then $W_i / W_{i-1} \simeq \det\nolimits^i \oplus \det\nolimits^i \mathrm{Sym}^{p-1} E^2$.  

We conclude that all the pairs $(I, \gamma)$ with $I \subseteq \{ 0 \}$ and $\gamma \in \Nt$ lie in the image of the correspondence of Proposition~\ref{cor:types.as.pairs}.  This observation readily identities the Jordan-H\"older constituents of $V_{2,r}$ and their multiplicities; in particular, $V_{2,r}$ has exactly $2p$ Jordan-H\"older constituents.    
Moreover, the partial order $\leq_r$ is easily made explicit using the criterion of Proposition~\ref{pro:equality.of.orders}; alternatively, it can readily be worked out directly from Definition~\ref{def:covering.relations.ramified}.  Given $r \in \Z / (p-1)\Z = \Nt \setminus \{ 0 \}$, define $\lfloor \frac{r}{2} \rfloor$ to be the least element $\alpha \in \Nt$, with respect to $\preceq$, satisfying $2 \alpha \prec r$ but not $2 (\alpha + 1) \prec r$.  Note that the parity of $r$ is well-defined, since $p > 2$, and that $\lfloor \frac{r}{2} \rfloor$ is $\frac{r-1}{2}$ if $r$ is odd and $\frac{r}{2} - 1$ if $r$ is even, so this is a non-standard use of the floor notation.  However, it is convenient for our purposes.

\begin{lem}
Let $F/\Q_p$ be a non-trivial totally ramified extension, and let $(I^\prime, \gamma)$ and $(I, \gamma)$ be two types.  Then $(I^\prime, \gamma^\prime) \leq_r (I,\gamma)$ if and only if $\gamma^\prime \preceq \gamma$ and (at least) one of the following conditions holds:
\begin{itemize}
\item $I^\prime \subseteq I$;
\item There exists $\alpha \in \left\{ \lfloor \frac{r}{2} \rfloor, \lfloor \frac{r}{2} \rfloor + \frac{p-1}{2} \right\}$ such that $\gamma^\prime \preceq \alpha \prec \gamma$.
\end{itemize}
\end{lem}
\begin{proof}
Since $I^\prime$ has at most one element, it is immediate from Proposition~\ref{pro:equality.of.orders} that $(I^\prime, \gamma^\prime) \leq_r (I,\gamma)$ if and only if $\gamma^\prime \preceq \gamma$ and one of the following holds:
\begin{itemize} 
\item $I^\prime \subseteq I$;
\item $I^\prime \subseteq I(r-2\gamma, 2(\gamma \dotdiv \gamma^\prime)) \cup I(\gamma \dotdiv \gamma^\prime, \gamma \dotdiv \gamma^\prime)$.
\end{itemize}  
The second option holds precisely when $I(r-2\gamma, 2(\gamma \dotdiv \gamma^\prime)) \cup I(\gamma \dotdiv \gamma^\prime, \gamma \dotdiv \gamma^\prime) \neq \varnothing$.  

Clearly, $I(\gamma \dotdiv \gamma^\prime, \gamma \dotdiv \gamma^\prime) \neq \varnothing$ if and only if $\gamma^\prime + \frac{p-1}{2} \prec \gamma$.  On the other hand, we see that $I(r-2\gamma, 2(\gamma \dotdiv \gamma^\prime)) \neq \varnothing$ if and only if $(r - 2\gamma) + (2(\gamma \dotdiv \gamma^\prime)) > p - 1$ as integers.  Thus $I(\gamma \dotdiv \gamma^\prime, \gamma \dotdiv \gamma^\prime) \neq \varnothing$ implies $I(r-2\gamma, 2(\gamma \dotdiv \gamma^\prime)) \neq \varnothing$, whereas the condition $I(r-2\gamma, 2(\gamma \dotdiv \gamma^\prime)) \neq \varnothing$ is easily seen to be equivalent to the second of the two alternatives in our claim.
\end{proof}

The structure of $V_{2,r}$ is more uniform if $r$ is odd, since the exceptional case $r - 2\gamma = p - 1$ does not arise.  For this reason, we treat the odd case first.

\begin{pro} \label{pro:totally.ramified.odd}
Let $F / \Q_p$ be a non-trivial totally ramified extension, and suppose that $r \in \Z / (p-1)\Z$ is odd.  Then
$V_{2,r}$ has $2p$ Jordan-H\"older constituents.  Moreover,
\begin{enumerate}[label=(\alph*)]
\item 
The Serre weight $\mathrm{Sym}^r E^2$ appears with multiplicity three, as $L(\varnothing, 0)$, as $L(\{ 0 \}, r)$, and as $L(\varnothing, p-1)$.
\item
The Serre weight $\det^r \tensor \mathrm{Sym}^{p-1-r} E^2$ appears with multiplicity three, as $L(\{ 0 \}, 0)$, as $L(\varnothing, r)$, and as $L(\{ 0 \}, p - 1)$.
\item
For each of the $p-3$ elements $i \in \Nt \setminus \{ 0, r, p - 1 \}$, the Serre weight $\det^i \tensor \mathrm{Sym}^{(r - 2i)} E^2$ appears with multiplicity two, as $L(\varnothing, i)$ and as $L(\{ 0 \}, r - i)$.  
\end{enumerate}
\end{pro}
\begin{proof}
This is immediate from Lemma~\ref{pro:head} and the observation above that the correspondence of Proposition~\ref{cor:types.as.pairs} is surjective in our case, noting that $r - 2i \neq p - 1$ for all $i \in \Nt$ as $r$ is odd.
\end{proof}

The set of types in the case of odd $r$, together with the partial order $\leq_r$ and the corresponding Serre weights $L(I,\gamma)$, are illustrated in Figure~\ref{fig:jh.graph}.  An arrow has been drawn from $L(I,\gamma)$ to $L(I^\prime, \gamma^\prime)$ when $(I^\prime, \gamma^\prime) \leq_r (I, \gamma)$ are adjacent types; by Proposition~\ref{pro:adjacent.types}, a non-split extension of these Serre weights appears as a subquotient of $V_{2,r}$.

    \begin{figure} 
    \begin{tikzcd}
    L(\varnothing, p-1) = \Sym^r \arrow[d] & \arrow[l] \det\nolimits^r \tensor \Sym^{p-1-r} = L(\{ 0 \}, p - 1) \arrow[d] \\
    L(\varnothing, p - 2) = \det\nolimits^{p-2} \tensor \Sym^{r+2}  \arrow[d] & \arrow[l] \det\nolimits^{r+1} \tensor \Sym^{p-3-r} = L(\{ 0 \}, p - 2) \arrow[d] \\
    \vdots & \vdots \\
    L(\varnothing, \frac{p+r}{2} + 1) = \det\nolimits^{\frac{p+r}{2}+1} \tensor \Sym^{p-4} \arrow[d] & \arrow[l] \det\nolimits^{\frac{p+r}{2} - 2} \tensor \Sym^3 = L(\{ 0 \}, \frac{p+r}{2} + 1) \arrow[d] \\
    L(\varnothing, \frac{p+r}{2}) = \det\nolimits^{\frac{p+r}{2}} \tensor \Sym^{p-2} \arrow[rd] & \arrow[l] \det\nolimits^{\frac{p+r}{2} - 1} \tensor \Sym^1 = L(\{ 0 \}, \frac{p+r}{2}) \\
    L(\varnothing, \frac{p+r}{2} - 1) = \det\nolimits^{\frac{p+r}{2} - 1} \tensor \Sym^1 \arrow[d] & \arrow[l] \det\nolimits^{\frac{p+r}{2}} \tensor \Sym^{p-2} = L(\{ 0 \}, \frac{p+r}{2} - 1) \arrow[d] \\
    \vdots & \vdots \\
    L(\varnothing, \frac{r+1}{2}) = \det\nolimits^{\frac{r+1}{2}} \tensor \Sym^{p-2} \arrow[rd] & \arrow[l] \det\nolimits^{\frac{r-1}{2}} \tensor \Sym^1 = L(\{ 0 \}, \frac{r+1}{2}) \\
    L(\varnothing, \frac{r-1}{2}) = \det\nolimits^{\frac{r-1}{2}} \tensor \Sym^1 \arrow[d] & \arrow[l] \det\nolimits^{\frac{r+1}{2}} \tensor \Sym^{p-2} = L(\{ 0 \}, \frac{r-1}{2}) \arrow[d] \\
    L(\varnothing, \frac{r-3}{2}) = \det\nolimits^{\frac{r-3}{2}} \tensor \Sym^3 \arrow[d] & \arrow[l] \det\nolimits^{\frac{r+3}{2}} \tensor \Sym^{p-4} = L(\{ 0 \}, \frac{r-3}{2}) \arrow[d] \\
    \vdots & \vdots \\
    L(\varnothing, 1) = \det \tensor \Sym^{r-2} \arrow[d] & \arrow[l] \det\nolimits^{r-1} \tensor \Sym^{p+1-r} = L(\{ 0 \}, 1) \arrow[d] \\
    L(\varnothing, 0) = \Sym^r & \arrow[l] \det\nolimits^r \tensor \Sym^{p-1-r} = L(\{ 0 \}, 0)
    \end{tikzcd}
    \caption{Submodule structure of $V_{2,r}$ for $F / \Q_p$ totally ramified, $r$ odd}
    \label{fig:jh.graph}
    \end{figure}
    
 The analogous statement for even $r$ is very similar, but the numerology must be adjusted to make allowance for the exceptional case $r - 2\gamma = 0$, in which case $L(\varnothing, \gamma)$ is always $p$-dimensional and $L(\{ 0 \}, \gamma)$ is one-dimensional.  As in the case of odd $r$ above, Lemma~\ref{pro:head} implies the following listing of Jordan-H\"older constituents and their multiplicities.
    
    \begin{pro} \label{pro:totally.ramified.even}
    Let $F / \Q_p$ be a non-trivial totally ramified extension, and suppose that $r \in \Z / (p-1)\Z$ is even.  Then
$V_{2,r}$ has $2p$ Jordan-H\"older constituents. 
\begin{enumerate}[label=(\alph*)]
\item
If $r \neq p -1$, then
\begin{enumerate}[label={(\roman*)}]
\item
The Serre weight $\mathrm{Sym}^r$ appears with multiplicity three, as $L(\varnothing, 0)$, as $L(\{ 0 \}, r)$, and as $L(\varnothing, p - 1)$.
\item
The Serre weight $\det^r \tensor \Sym^{p-r-1}$ appears with multiplicity three, as $L(\{ 0 \}, 0)$, as $L(\varnothing, r)$, and as $L(\{ 0 \}, p - 1)$.
\item
There are exactly two elements $\gamma \in \Nt$ satisfying $r - 2\gamma = p - 1$.  For each such $\gamma$, the Serre weight $\det^\gamma \tensor \Sym^{p-1}$ appears with multiplicity one as $L(\varnothing, \gamma)$, and the Serre weight $\det^\gamma = \det^{r - \gamma}$ appears with multiplicity one as $L(\{ 0 \}, \gamma)$.
\item
For each of the $p - 5$ elements $i \in \Nt \setminus \{ 0, r, p-1 \}$ such that $r - 2i \neq p - 1$, the Serre weight $\det^i \tensor \Sym^{(r-2i)}$ appears with multiplicity two, as $L(\varnothing, i)$ and as $L(\{ 0 \}, r - i)$.
\end{enumerate}
\item
If $r = p - 1$, then
 \begin{enumerate}[label={(\roman*)}]
 \item
 The Serre weight $\Sym^{p-1}$ appears with multiplicity two, as $L(\varnothing, 0)$ and $L(\varnothing, p - 1)$.
 \item
 The trivial Serre weight $\mathbf{1}$ appears with multiplicity two, as $L(\{ 0 \}, 0)$ and $L(\{ 0 \}, p - 1)$.
 \item
 The Serre weights $\det^{\frac{p-1}{2}} \tensor \Sym^{p-1}$ and $\det^{\frac{p-1}{2}}$ each appear with multiplicity one, as $L(\varnothing, \frac{p-1}{2})$ and $L(\{ 0 \}, \frac{p-1}{2})$, respectively.
 \item
 For each of the $p - 3$ elements $i \in \Nt \setminus \left\{ 0, \frac{p-1}{2}, p - 1 \right\}$, the Serre weight $\det^i \tensor \Sym^{(r-2i)}$ appears with multiplicity two, as $L(\varnothing, i)$ and as $L(\{ 0 \}, r - i)$.
 \end{enumerate}
\end{enumerate}
    \end{pro}

    \begin{lem} \label{pro:totally.ramified.either.or}
Let $F / \Q_p$ be a non-trivial totally ramified extension, suppose we are given $r \in \Z / (p-1)\Z$, and let $\alpha \in \left\{ \lfloor \frac{r}{2} \rfloor, \lfloor \frac{r}{2} \rfloor + \frac{p-1}{2} \right\}$.  Let $M \subseteq V_{2,r}$ be a $\mathrm{GL}_2(\mathcal{O} / \mathfrak{m}^2)$-submodule.  Then either $M \subseteq W_\alpha$ or $W_\alpha \subset M$.
\end{lem}
\begin{proof}
First consider the case $\alpha = \lfloor \frac{r}{2} \rfloor + \frac{p-1}{2}$.  Assume $M$ is not contained in $W_\alpha$.  By the explicit Propositions~\ref{pro:totally.ramified.odd} and~\ref{pro:totally.ramified.even}, the quotient $V_{2,r} / W_\alpha$ is multiplicity free with the following socle:
$$
\mathrm{soc} \, (V_{2,r} / W_\alpha) = \begin{cases}
L(\varnothing, \alpha + 1) \simeq \det^{\alpha + 1} \tensor \Sym^{p-2} &: r \text{ odd}\\
L(\varnothing, \alpha +1) \oplus L(\{ 0 \}, \alpha + 1) \simeq \det\nolimits^{\alpha + 1} \tensor \Sym^{p-1} \oplus \det\nolimits^{\alpha + 1} &: r \text{ even}.
\end{cases}
$$
The image of $M$ in $V_{2,r} / W_\alpha$ must contain at least one component of the socle.  Suppose first either that $r$ is odd, or that $r$ is even and that the image of $M$ contains the first component of the socle.  Then there exists $w \in W_\alpha$ such that $z = f_{(p-2,\alpha + 1)} + w \in M$.
By~\eqref{equ:moving.vertically.down} we have
\begin{equation} \label{equ:finite.submodule.structure.totram}
\sum_{b_1 \in \F_p} b_1^{p-2} \left( \begin{array}{cc} 1 & [b_1] \varpi \\ 0 & 1 \end{array} \right) z = 
(\alpha + 1) f_{(p-2, \alpha)} + w^\prime \in M
\end{equation}
for some $w^\prime \in W_{\alpha - 1}$.  Since $p \geq 3$, we see that $[(p-2, \alpha)] = (\{ 0 \}, \alpha)$.  Since $W_\alpha$ has a unique maximal submodule by Corollary~\ref{cor:unique.maximal.submodule} (the hypotheses of the corollary are satisfied because of our non-standard definition of $\lfloor \frac{r}{2} \rfloor$) and $w^\prime$ is contained in it, the element in the right-hand side of the previous displayed formula generates $W_\alpha = \langle G_2 \cdot f_{(p-2,\alpha)} \rangle$.  Thus $W_\alpha \subset M$.

The remaining case is that $r$ is even and the image of $M$ in $V_{2,r} / W_\alpha$ contains only the one-dimensional component of the socle.  Then the claim follows from Lemma~\ref{lem:nonsplit.extension.exceptional}\ref{item:containment.exceptional}.

Now suppose $\alpha = \lfloor \frac{r}{2} \rfloor$.  If $M$ is not contained in $W_{\alpha + \frac{p-1}{2}}$, then
$W_\alpha \subset W_{\alpha + \frac{p-1}{2}} \subseteq M$ by the previous paragraph.  So we may assume $M \subseteq W_{\alpha + \frac{p-1}{2}}$.  Observe that the subquotient $W_{\alpha + \frac{p-1}{2}} / W_\alpha$ is multiplicity-free and repeat the argument of the previous paragraph. 
\end{proof}

We are now able to determine the submodule structure of $V_{2,r}$ completely, when $F / \Q_p$ is totally ramified.
 
 \begin{thm} \label{cor:finite.submodule.structure}
 Suppose the extension $F / \Q_p$ is totally ramified and non-trivial.  The following statements hold for all $r \in \Z / (p-1)\Z$:
 \begin{enumerate}[label=(\alph*)]
 \item \label{item:finite.submodule.structure}
 The $\mathrm{GL}_2(\mathcal{O} / \mathfrak{m}^2)$-module $V_{2,r}$ has finitely many submodules.
 \item \label{item:exhaustive.list.submodules}
 Let $M \subseteq V_{2,r}$ be a submodule.  Then $M$ has irreducible cosocle if and only if $M = V_{(I, \gamma)}$ for some type $(I, \gamma)$, except that if $r - 2\gamma = p-1$ then we consider the submodule $V^\prime_{(\{ 0 \}, \gamma)}$ generated by $ f_{0,\gamma} + f_{\infty, \gamma} - f_{p-1, \gamma}$ instead of $V_{(\{ 0 \}, \gamma)}$. 
 \item \label{item:jh.of.exhaustive.list}
 The multiset of Jordan-H\"older constituents of $V_{(I, \gamma)}$ is 
 $\left\{ L(I^\prime, \gamma^\prime): (I^\prime, \gamma^\prime) \leq_r (I, \gamma) \right\}$.  If $r - 2\gamma = p - 1$, then the multiset of Jordan-H\"older constituents of $V^\prime_{(\{ 0 \}, \gamma)}$ is
 $\left\{ L(I^\prime, \gamma^\prime): (I^\prime, \gamma^\prime) \leq_r (\{ 0 \}, \gamma) \right\} \setminus \{ L(\varnothing, \gamma) \}$.
 \end{enumerate}
 \end{thm}
 \begin{proof}
 To establish~\ref{item:finite.submodule.structure}, it suffices by Lemma~\ref{lem:jans.camillo} to prove that the quotient $V_{2,r}/M$ has multiplicity-free socle for every submodule $M \subseteq V_{2,r}$.  Suppose that $M \subseteq V_{2,r}$ is a counterexample, and let $M \subset M^\prime$ be such that $M^\prime / M = \mathrm{soc} V_{2,r}/M$.  
Set 
$$V_{2,r}^{(0)} = 0, \, V_{2,r}^{(1)} =  W_{\lfloor \frac{r}{2} \rfloor}, \, V_{2,r}^{(2)} = W_{\lfloor \frac{r}{2} \rfloor + \frac{p-1}{2}} \, V_{2,r}^{(3)} = V_{2,r}.$$
By Proposition~\ref{pro:totally.ramified.odd} for odd $r$ and Proposition~\ref{pro:totally.ramified.even} for even $r$, the quotients $V_{2,r}^{(i)} / V_{2,r}^{(i-1)}$ are multiplicity-free for all $i \in [3]$.  Thus by Lemma~\ref{pro:totally.ramified.either.or} there exists $\alpha \in \left\{ \lfloor \frac{r}{2} \rfloor, \lfloor \frac{r}{2} \rfloor + \frac{p-1}{2} \right\}$ such that $M \subset W_\alpha \subset M^\prime$.  As in the proof of Lemma~\ref{pro:totally.ramified.either.or} there exists $w \in W_\alpha$ such that $f_{(p-3, \alpha + 1)} + w \in M^\prime$.  Since the quotient $M^\prime / M$ is semisimple, it is invariant under the action of the first congruence subgroup $\ker (\mathrm{GL}_2(\mathcal{O} / \mathfrak{m}^2) \twoheadrightarrow \mathrm{GL}_2(k))$.  Hence the expression in~\eqref{equ:finite.submodule.structure.totram}, which is already known to generate $W_\alpha$, is contained in $M$.  However, this contradicts $M \subset W_\alpha$, completing the proof of~\ref{item:finite.submodule.structure}.
 
For every $i \in [3]$ and every Jordan-H\"older constituent $\sigma$ of $V_{2,r}^{(i)} / V_{2,r}^{(i-1)}$, there is a unique submodule of $V_{2,r}^{(i)} / V_{2,r}^{(i-1)}$ with cosocle $\sigma$ and hence a unique submodule $M \subset V_{2,r}$ such that $V_{2,r}^{(i-1)} \subseteq M \subseteq V_{2,r}^{(i)}$ and $M$ has cosocle $\sigma$.  These are exactly the submodules of the statement~\ref{item:exhaustive.list.submodules}.  By Proposition~\ref{pro:totally.ramified.either.or}, they exhaust the submodules of $V_{2,r}$ with irreducible cosocle, and~~\ref{item:exhaustive.list.submodules} is proved.  

Finally,~\ref{item:jh.of.exhaustive.list} follows from Lemma~\ref{pro:head} and Proposition~\ref{pro:adjacent.types}.
 \end{proof}
    
       \section{The unramified case and the representation $R(\sigma)$ of Breuil and \Paskunas} \label{sec:R.sigma}
       In this section we assume that $F / \Q_p$ is an unramified extension.  In this case, the analogue of Theorem~\ref{thm:coordinate.submodules}, the principal result of the previous section, fails.  When $e = 1$, it is not in general true that the submodule of $V_{2,r}$ generated by a basis vector $f_{(j_0, j_1)}$ is the linear span of a subset of the basis $\mathcal{B}_{2,r}$.  However, if we restrict to the submodule of $V_{2,r}$ generated by $f_{(0, 1 + p + \cdots + p^{f-1})}$, then we can obtain an elegant description of its submodule structure.  While it is possible to make a self-contained study of this submodule, roughly along the lines of \S\ref{sec:main}, we will take a shortcut.  It turns out that this submodule is isomorphic to the module $R(\sigma)$ considered by Breuil and {\Paskunas} in~\cite[\S17-18]{BP/12}, for the Serre weight $\sigma = \sigma_{\varnothing}(r)$, and we make use of their results on our way to providing a complete description of the submodule structure.

       \subsection{Preliminary computations}
       The results of the previous section, for the ramified case $e \geq 2$, rely on the explicit computations of \S\ref{sec:prelim.ramified}.  While Lemmas~\ref{lem:w.action.on.basis} and~\ref{lem:diagonal.action.on.basis} make no assumption on the ramification of $F / \Q_p$, the action of upper unitriangular matrices on $V_{2,r}$ involves summation of Witt vectors in the unramified case and produces more complicated formulae than those of Lemma~\ref{lem:upper.triangular.action.on.basis}.  In particular, the $U_2$-submodule of $V_{2,r}$ generated by a basis element $f_{(j_0, j_1)}$ will not, in general, be a linear span of basis elements.  However, the $U_2$-submodule generated by $f_{(j_0, j_1)}$ does have this property in some special cases.  
        
 \begin{lem} \label{lem:submodule.generated.unramified}
 Suppose that $e = 1$.  
 \begin{enumerate}
 \item
 If $(j_0, j_1) \in \Nt^2$, then the $U_2$-submodule of $V_{2,r}$ generated by $f_{(j_0, j_1)}$ is the linear span of the elements
 \begin{equation} \label{equ:upper.triangular.general.unramified}
  \sum_{\lambda \in \mathcal{O}_2} \left( \begin{array}{cc} \lambda & 1 \\ 1 & 0 \end{array} \right) \tensor (\lambda_0 - b_0)^{j_0} (\lambda_1 + S(\lambda_0^{p^{f-1}} , (-b_0)^{p^{f-1}}))^{j_1^\prime},
  \end{equation}
 where $b_0 \in k$ and $j_1^\prime \preceq j_1$.  
 \item \label{item:small.support}
 In particular, let $I \subseteq [f-1]_0$, and let $j_0 \in \Nt$ satisfy $\mathrm{supp}(j_0) \cap (I - 1) = \varnothing$.  Then a basis of the $U_2$-submodule of $V_{2,r}$ generated by $f_{(j_0, p^I)}$ is given by
 $$ \left\{ f_{\left( \sum_{i \in (I^\prime - 1) \cup \mathrm{supp}(j_0)} a_i p^i, p^J \right)} : J, I^\prime \subseteq I, J \cap I^\prime = \varnothing; \, a_i \in \begin{cases} [p-1] &: i \in I^\prime - 1 \\ [(j_0)_i]_0 &: i \in \mathrm{supp}(j_0) \end{cases} \right\}.$$
 \end{enumerate}
 \end{lem}
 \begin{proof}
 In the case $e = 1$, it is immediate from Lemma~\ref{lem:lambda.minus.b} and~\eqref{equ:upper.triangular.basis.element} that
\begin{equation} \label{equ:full.upper.triangular.unramified}
\left( \begin{array}{cc} 1 & b \\ 0 & 1 \end{array} \right) f_{(j_0, j_1)} = \sum_{\lambda \in \mathcal{O}_2} \left( \begin{array}{cc} \lambda & 1 \\ 1 & 0 \end{array} \right) \tensor (\lambda_0 - b_0)^{j_0} (\lambda_1 - b_1 + S(\lambda_0^{p^{f-1}}, (-b_0)^{p^{f-1}}))^{j_1},
\end{equation}
 and considering weighted sums of the form 
 $$\sum_{b_1 \in k} b_1^{(q-1) \dotdiv (j_1 \dotdiv j_1^\prime)} \left( \begin{array}{cc} 1 & [b_0] + [b_1] \varpi \\ 0 & 1 \end{array} \right) f_{(j_0, j_1)}$$
 for $b_0 \in k$ and $j_1^\prime \preceq j_1$ gives the first part of our claim.  In general, the expressions of the form~\eqref{equ:upper.triangular.general.unramified}, viewed as functions of three variables $\lambda_0, \lambda_1, b_0$,  contain more than one monomial in which $b_0$ appears with a given exponent.  Thus, taking weighted sums with respect to $b_0$ will produce linear combinations of the basis elements $f_{(j_0^\prime, j_1^\prime)}$.  
 
 The case $(j_0, j_1) = (j_0, p^I)$, where $\mathrm{supp}(j_0)$ and $I - 1$ are disjoint, behaves particularly agreeably.  Observe that $j_1^\prime \preceq p^I$ if and only if $j_1^\prime = p^J$ for some subset $J \subseteq I$.  Then the elements of~\eqref{equ:upper.triangular.general.unramified} have the form
 \begin{multline*}
 \sum_{\lambda \in \mathcal{O}_2} \left( \begin{array}{cc} \lambda & 1 \\ 1 & 0 \end{array} \right) \tensor (\lambda_0 - b_0)^{j_0} \prod_{i \in J} (\lambda_1 + S(\lambda_0^{p^{f-1}}, (-b_0)^{p^{f-1}}))^{p^i} = \\
 \sum_{\lambda \in \mathcal{O}_2} \left( \begin{array}{cc} \lambda & 1 \\ 1 & 0 \end{array} \right) \tensor (\lambda_0 - b_0)^{j_0} \prod_{i \in J} \left( \lambda_1^{p^i} + \sum_{a = 1}^{p-1} (-1)^{p-a} \frac{\binom{p}{a}}{p} \lambda_0^{a p^{i-1+f}} b_0^{(p-a) p^{i-1+f}} \right).
 \end{multline*}
Each monomial has a distinct exponent of $b_0$, since we have assumed that $\mathrm{supp}(j_0)$ and $J - 1$ are disjoint.  Taking weighted sums, with respect to $b_0$, of such expressions produces non-zero scalar multiples of all elements of the form
 $$ \sum_{\lambda \in \mathcal{O}_2} \left( \begin{array}{cc} \lambda & 1 \\ 1 & 0 \end{array} \right) \tensor \lambda_1^{p^{I^\prime}} \lambda_0^{j_0^\prime + \sum_{i \in ((J \setminus I^\prime) - 1)} a_i p^i},$$
 where $j_0^\prime \preceq j_0$.
 Running over all $J \subseteq I$, we obtain the claimed basis of $\langle U_2 \cdot f_{(j_0, p^J)} \rangle$.
 \end{proof}
       
       Recall from \S\ref{subsec:jh} that $\widetilde{B}_2 = (\psi_1^2)^{-1}(B(k)) \leq \mathrm{GL}_2(\mathcal{O} / \mathfrak{m}^2)$ is the image of the standard Iwahori subgroup of $\mathrm{GL}_2(\mathcal{O})$ under the projection $\mathrm{GL}_2(\mathcal{O}) \to \mathrm{GL}_2(\mathcal{O}/\mathfrak{m}^2)$.  Thus
       $$ \widetilde{B}_2 = \left\{ \left( \begin{array}{cc} a & b \\ c & d \end{array} \right) \in \mathrm{GL}_2(\mathcal{O} / \mathfrak{m}^2) : c \in \mathfrak{m} / \mathfrak{m}^2 \right\}.$$
       
       \begin{lem} \label{lem:iwahori.unramified.basis}
       Let $I \subseteq [f-1]_0$, and suppose that $j_0 \in \Nt$ satisfies $\mathrm{supp}(j_0) \cap (I - 1) = \varnothing$.  Given a subset $I^\prime \subseteq I$, define 
       $$\mathcal{L}(j_0, I^\prime) = \left\{ \sum_{i \in (I^\prime - 1) \cup \mathrm{supp}(j_0)} a_i p^i : a_i \in \begin{cases}
       [p-1] &: i \in I^\prime - 1 \\
       [(j_0)_i]_0 &: i \in \mathrm{supp}(j_0) \end{cases} \right\}.$$
       
       The $\widetilde{B}_2$-submodule of $V_{2,r}$ generated by $f_{(j_0, p^I)}$ is the linear span of the following basis:
       $$\mathcal{B}(j_0, I) = \left\{ f_{( \ell + p^{J \setminus J^\prime} + 2 p^{J^\prime \setminus J^{\dpr}}, p^{J^{\dpr}})} : 
       I^\prime, J \subseteq I, \, I^\prime \cap J = \varnothing, \, \ell \in \mathcal{L}(j_0, I^\prime), \, 
       J^{\dpr} \subseteq J^\prime \subseteq J
        \right\}.$$
         
       In particular, the $\widetilde{B}_2$-submodule of $V_{2,r}$ generated by $f_{(0, 1 + p + \cdots + p^{f-1})}$ is the linear span of the basis
      $
             \{ f_{(\ell + p^{I \setminus J} + 2p^{J \setminus L}, p^L)} : L \subseteq J \subseteq I, \, \ell \in \Nt, \, \mathrm{supp}(\ell) \cap (I - 1) = \varnothing \}.
      $
       \end{lem}
       \begin{proof}
       Any element $\xi \in \widetilde{B}_2$ may be expressed in the form $\xi = \kappa \delta \mu$, where $\mu \in U_2$ and $\kappa \in \overline{U}_2$ (see Definition~\ref{def:subgroups}), while $\delta$ is a diagonal matrix.  The claim follows from an application of Lemma~\ref{lem:submodule.generated.unramified}\eqref{item:small.support} to account for the action of $\mu$, then of Lemmas~\ref{rem:eigenvectors} and~\ref{lem:diagonal.action.on.basis} to describe the action of $\delta$, and finally of Lemma~\ref{lem:lower.triangular.unramified} to treat the action of $\kappa$.
       \end{proof}

       \subsection{A family of submodules}
       We now define the analogue, in the setting of unramified $F / \Q_p$, of (a subset of) the family of submodules $\{ V_\theta \}_{\theta \in \widetilde{\Theta}}$ considered in Definition~\ref{def:v.theta}.  Recall that we identify $\Z / f\Z$ with $[f-1]_0$.  Given $r \in \Z / (q-1)\Z$ and $J \subseteq \Z / f\Z$, recall the integer $s_J(r)$ from Definition~\ref{def:serre.weight.sjr}.
       
       \begin{dfn} \label{def:unramified.basic.notions}
       Let $\widetilde{\Theta}_1 = \{ (J, I) : I, J \subseteq \Z / f\Z \, , \, J \cap (I - 1) = \varnothing \}$.  
      
       If $\theta = (J, I) \in \widetilde{\Theta}_1$, set $f_\theta = f_{(s_J(r - 2p^I), p^I)}$, and let $M_\theta$ be the $\mathrm{GL}_2(\mathcal{O} / \mathfrak{m}^2)$-submodule of $V_{2,r}$ generated by $f_\theta$.
       Let $\sigma_\theta$ be the Serre weight $\sigma_\theta = \det^{p^I} \tensor \sigma_J(r - 2p^I)$.
       
       Define a partial order $\sqsubseteq$ on $\widetilde{\Theta}_1$ as follows: $(J^\prime, {I^\prime}) \sqsubseteq (J, I)$ if $I^\prime \subseteq I$ and if $J^\prime \subseteq J \cup ((I \setminus I^\prime) - 1)$.
       \end{dfn}
       
       In order to obtain uniform results in this section, we make a genericity assumption stronger than the one of~\cite[Definition~11.7]{BP/12}.
       
       \begin{dfn} \label{def:strong.genericity}
       We say that $r \in \Z / (q-1)\Z$ is generic if $3 \leq r_i \leq p - 3$ for all $i \in \Z / f\Z$.
              \end{dfn}
       
       \begin{lem} \label{lem:unramified.multiplicity.free}
       If $r \in \Z / (q-1)\Z$ is generic, then the $\mathrm{GL}_2(\mathcal{O} / \mathfrak{m}^2)$-submodule $W_{1 + p + \cdots + p^{f-1}} \subset V_{2,r}$ is multiplicity-free.
       \end{lem}
       \begin{proof}
       By genericity, for all $I \subseteq \Z / f\Z$ we have $r - 2p^I \neq q - 1$.  Hence by the filtration of Proposition~\ref{pro:filtration} and by Theorem~\ref{thm:bardoe.sin.generic}, which describes the graded pieces of this filtration, we see that the multiset of Jordan-H\"older constituents of $W_{1 + p + \cdots + p^{f-1}}$ is 
       $$ \left\{ \det\nolimits^{p^I} \tensor \sigma_J(r-2p^I) \simeq F(r - s_J(r-2p^I) + p^I, s_J(r - 2p^I) + p^I) : \, I,J \subseteq \Z / f\Z   \right\}.$$
Hence it suffices to show that the numbers $s_J(r - 2p^I) + p^I$ (recall Definition~\ref{def:serre.weight.sjr}), viewed as elements of $\Z / (q-1)\Z$, are distinct for all pairs $(I,J) \in \mathcal{P}(\Z / f\Z)^2$.  By genericity, given $\alpha = \sigma_J(r - 2p^I) + p^I$ we can read off $J = \{ i \in \Z / f\Z : \, \alpha_i > 1 \}$.  Knowing $J$, we can then determine $I = \{ i \in \Z / f\Z : \, \alpha_i \neq s_{J,i}(r) \}$.
       \end{proof}
       
     \begin{pro} \label{pro:unramified.irreducible.cosocle}
     Let $\theta = (J, I) \in \widetilde{\Theta}_1$, and suppose that $r$ is generic.  The $\mathrm{GL}_2(\mathcal{O} / \mathfrak{m}^2)$-module $M_\theta$ has irreducible cosocle isomorphic to the Serre weight $\sigma_\theta$.
     \end{pro}  
       \begin{proof}
       Let $N_\theta \subseteq W_{1 + p + \cdots + p^{f-1}}$ be the submodule with cosocle $\sigma_\theta$; this is well-defined by Lemma~\ref{lem:unramified.multiplicity.free}.  Recall from Proposition~\ref{pro:filtration} that $W_{p^I} / \sum_{I^\prime \subset I} W_{p^{I^\prime}} \simeq \det^{p^I} \tensor V_{1, r - 2p^I}$.  The image of $M_\theta$ in this quotient is generated by $f_{s_J(r-2p^I)}$.  This is the submodule of $\det^{p^I} \tensor V_{1, r - 2p^I}$ with cosocle $\sigma_\theta$, by~\cite[Theorem~2.7]{BP/12} or Theorem~\ref{thm:bardoe.sin.generic}.  Since $N_\theta$ is the minimal submodule of $W_{1 + p + \cdots + p^{f-1}}$ that contains $\sigma_\theta$ as a Jordan-H\"older constituent, the inclusion $N_\theta \subseteq M_\theta$ must hold.
       
       The module $N_\theta$ is cyclic, and by the previous paragraph it has a generator of the form $h = f_\theta + z$, for $z \in \sum_{I^\prime \subset I} W_{p^{I^\prime}}$.  Without loss of generality we may assume that $h$ is an eigenvector for the action of $T$.  By genericity of $r$, this implies that 
       $$ h = f_\theta + \sum_{I^\prime \subset I} \varepsilon_{I^\prime} f_{(s_J(r-2p^I) + p^{I \setminus I^\prime}, p^{I^\prime})},$$
       for scalars $\varepsilon_{I^\prime} \in E$.  Now by~\eqref{equ:diagonal.explicit} we have
       \begin{multline*}
        - \sum_{d \in k} d^{(q - 1) \dotdiv p^{I \setminus I^\prime}} \left( \begin{array}{cc} 1 & 0 \\ 0 & 1 + [d] \varpi \end{array} \right) h = \\ f_{(s_J(r-2p^I) + p^{I \setminus I^\prime}, p^{I^\prime})} + \sum_{I^\dpr \subset I^\prime} \varepsilon_{I^\dpr \cup (I \setminus I^\prime)} f_{(s_J(r-2p^I) + p^{I \setminus I^\dpr}, p^{I^\dpr})} \in N_\theta.
       \end{multline*}
       Taking a suitable linear combination of these elements, we find that $f_\theta \in N_\theta$ and hence that $M_\theta \subseteq N_\theta$, which proves the claim.
       \end{proof}
       
       \begin{pro} \label{pro:unramified.inclusion}
       Suppose that $r$ is generic.  Let $\theta, \theta^\prime \in \widetilde{\Theta}_1$.  If $\theta^\prime \sqsubseteq \theta$, then $M_{\theta^\prime} \subseteq M_\theta$.  
       \end{pro}
       \begin{proof}
       Let $\theta = (J, I)$.
       The partial order $\sqsubseteq$ is generated by the following two covering relations:
       \begin{itemize}
       \item
       $(J \setminus \{ j \}, I) \sqsubset (J, I)$ for all $j \in J$;
       \item
       $(J \cup \{ i - 1 \}, {I \setminus \{ i \}}) \sqsubset (J,I)$ for all $i \in I$.
       \end{itemize}
       
       If $\theta^\prime \sqsubset \theta$ is a cover of the first kind, then by Theorem~\ref{thm:bardoe.sin.generic} and the proof of Proposition~\ref{pro:unramified.irreducible.cosocle}, the image of $M_\theta$ in the quotient $W_{p^I} / \sum_{I^\prime \subset I} W_{p^{I^\prime}}$ admits the Serre weight $\sigma_{\theta^\prime} = \det^{p^I} \tensor \sigma_{J \setminus \{ j \} }(r - 2p^I)$ as a Jordan-H\"older constituent.  Hence so does $M_\theta$.  This implies $M_{\theta^\prime} \subseteq M_\theta$, since it is immediate from Proposition~\ref{pro:unramified.irreducible.cosocle} that $M_{\theta^\prime}$ is the minimal submodule admitting $\sigma_{\theta^\prime}$ as a Jordan-H\"older constituent.
       
       Let $\theta^\prime \sqsubset \theta$ now be a cover of the second kind, and let $i \in I$.  We see directly from Definition~\ref{def:serre.weight.sjr} that 
       $ s_{J \cup \{ i - 1 \} }(r - 2p^I + 2p^{i}) = s_J (r - 2p^I) + (r_{i-1} + \kappa)p^{i-1} + \delta_{i \in J} p^i$,
       where $\delta_{i \in J} = 1$ if $i \in J$ and $\delta_{i \in J} = 0$ otherwise, while 
       $$ \kappa = \begin{cases} 1 &: i - 2 \not\in J, \, i - 1 \not\in I \\
       0 &: i - 2 \in J, \, i - 1 \not\in I \\
       -1 &: i - 2 \not\in J, \, i - 1 \in I.  \end{cases}
       $$
     Observe that $s_J (r - 2p^I) + (r_{i-1} + \kappa)p^{i-1} \in \mathcal{L}(s_J (r - 2p^I), \{ i \})$ by the genericity of $r$.  Hence if $i \not\in J$, then $f_{\theta^\prime} \in \langle \widetilde{B}_2 \cdot f_\theta \rangle \subseteq M_\theta$ by Lemma~\ref{lem:iwahori.unramified.basis}, where $I^\prime$ in the notation of that lemma is $\{ i \}$, while $J = J^\prime = J^\dpr$ is $I \setminus \{ i \}$.  Thus $M_{\theta^\prime} \subseteq M_\theta$ by Proposition~\ref{pro:unramified.irreducible.cosocle}; in fact, the containment is proper since $M_{\theta^\prime}$ and $M_\theta$ have non-isomorphic cosocles.
     
     If $i \in J$, then we observe that $f_{(s_J(r - 2p^I) + p^i, p^I)}$ generates $M_\theta$ by the same argument as in the proof of Proposition~\ref{pro:unramified.irreducible.cosocle}, but with this element in place of $f_\theta$.  Applying Lemma~\ref{lem:iwahori.unramified.basis} to $f_{(s_J(r - 2p^I) + p^i, p^I)}$, we find as in the previous case that $f_{\theta^\prime} \in M_\theta$.  
       \end{proof}

             \subsection{Relation to work of Breuil and \Paskunas}
             Observe that $(\varnothing, \Z / f\Z)$ is the unique maximal element of $\widetilde{\Theta}_1$.
             Hence, by Proposition~\ref{pro:unramified.inclusion}, the submodule $M_{(\varnothing, {\Z / f\Z})}$ of $V_{2,r}$ generated by $f_{(0, 1 + p + \cdots + p^{f-1})}$ contains the modules $M_{\theta}$ for all $\theta \in \widetilde{\Theta}_1$.  In order to completely describe the submodule structure of $M_{(\varnothing, {\Z / f\Z})}$, it remains to show that these submodules exhaust the submodules of $M_{(\varnothing, {\Z / f\Z})}$ with irreducible cosocle and that the implication of Proposition~\ref{pro:unramified.inclusion} is, in fact, an equivalence.  For this we make use of the results of~\cite{BP/12}, where the module $M_{(\varnothing, {\Z / f\Z})}$ was studied in another guise.  We start by elucidating the dictionary between some concepts of~\cite{BP/12} and those of the present paper.

       Write $\sigma_r$ for the Serre weight $F(r,0) = \sigma_{\varnothing}(r)$, and recall that it can be modeled as the subspace of the space $A[r]$ of \S\ref{sec:bardoe.sin} spanned by the monomials $X^{r-i}Y^i$ for $i \preceq r$.  Set $\Pi = \left( \begin{array}{cc} 0 & 1 \\ \varpi & 0 \end{array} \right)$.  Given a subset $J \subseteq \{ i \in \Z / f\Z : \, r_i > 0 \}$,  Breuil and {\Paskunas} define $\mathrm{Fil}^J \widetilde{R}(\sigma_r)$ to be the $\mathrm{GL}_2(\mathcal{O})$-submodule of the compact induction $\text{c-ind}^{\mathrm{GL}_2(F)}_{F^\times \mathrm{GL}_2(\mathcal{O})} \sigma_r$ generated by the element $\Pi \tensor X^{r - p^J} Y^{p^J}$; note that the notation $[g,v]$ of~\cite{BP/12} corresponds to our $g \tensor v$.
       
     \begin{lemma} \label{lem:bp.filtration}
     Let $J \subseteq \{ i \in \Z / f\Z : \, r_i > 0 \}$.  There is an isomorphism of $\mathrm{GL}_2(\mathcal{O})$-modules $\mathrm{Fil}^J \widetilde{R}(\sigma_r) \simeq W_{p^J}$, where $W_{p^J}$ is viewed as a $\mathrm{GL}_2(\mathcal{O})$-module by inflation via the surjection $\mathrm{GL}_2(\mathcal{O}) \to \mathrm{GL}_2(\mathcal{O} / \mathfrak{m}^2)$.
     \end{lemma}
     \begin{proof}
     Recall the submodules $U_\beta \subseteq \mathrm{Ind}^{\widetilde{B}_2}_{B_2} \chi_r$ defined in the proof of Lemma~\ref{lem:intermediate.induction.filtration}, for all $\beta \in \Nt$.  For every $\beta \preceq r$ there is an embedding of $\widetilde{B}_2$-modules $U_{\beta} \to \text{c-ind}^{\mathrm{GL}_2(F)}_{F^\times \mathrm{GL}_2(\mathcal{O})} \sigma_r$ given by $m_\alpha \mapsto \Pi \tensor (-1)^\alpha X^{r- \alpha} Y^\alpha$ for all $\alpha \preceq \beta$; indeed, the action of $\widetilde{B}_2$ on the left-hand side is given by~\eqref{equ:action.of.g.on.mbeta}, while a simple calculation provides the action on the right-hand side.  Now set $\beta = p^J$; then the image of this map is the $\widetilde{B}_2$-submodule of $\text{c-ind}^{\mathrm{GL}_2(F)}_{F^\times \mathrm{GL}_2(\mathcal{O})} \sigma_r$ generated by $\Pi \tensor X^{r - p^J} Y^{p^J}$, and by Frobenius reciprocity we get a surjection of $\mathrm{GL}_2(\mathcal{O} / \mathfrak{m}^2)$-modules $\Phi: W_{p^J} = \mathrm{Ind}^{\mathrm{GL}_2(\mathcal{O} / \mathfrak{m}^2)}_{\widetilde{B}_2} U_{p^J} \to \mathrm{Fil}^J \widetilde{R}(\sigma_r)$.
          
          It remains to prove that $\Phi$ is injective.  We show this directly.  Consider the basis of $W_{p^J}$ given in Definition~\ref{def:filtration} and compute, for all $j_0 \in \Nt$ and $J^\prime \subseteq J$, that
          \begin{eqnarray*}
          \Phi ( f_{(j_0, p^{J^\prime})}) & = & \sum_{\lambda \in k} \left( \begin{array}{cc} \varpi & \lambda \\ 0 & 1 \end{array} \right) \tensor (-1)^{p^{J^\prime}} \lambda^{j_0} X^{r - p^{J^\prime}} Y^{p^{J^\prime}} \\
          \Phi (f_{(\infty, p^{J^\prime})}) & = & \sum_{\lambda \in k} \left( \begin{array}{cc} \varpi & \lambda \\ 0 & 1 \end{array} \right) \tensor \lambda^{r - 2p^{J^\prime}} X^{r - p^{J^\prime}} Y^{p^{J^\prime}} + \left( \begin{array}{cc} 1 & 0 \\ 0 & \varpi \end{array} \right) \tensor (-1)^{p^{J^\prime}} X^{p^{J^\prime}} Y^{r - p^{J^\prime}}.
          \end{eqnarray*} 
          The elements on the right-hand side are well-known to be linearly independent in $\text{c-ind}^{\mathrm{GL}_2(F)}_{F^\times \mathrm{GL}_2(\mathcal{O})} \sigma_r$; for instance, cf.~\cite[\S3]{Breuil/03}.
     \end{proof}
     
     \begin{rem}
     Lemma~\ref{lem:bp.filtration} allows us to recover~\cite[Lemma~17.1]{BP/12} as a special case of Proposition~\ref{pro:filtration}.
     \end{rem}
     
      We are now in a position to reinterpret the module $R(\sigma_r)$ of~\cite[Definition~17.9]{BP/12} in terms of the notions studied in the present paper.  
      
      \begin{lem} \label{lem:r.sigma.characterization}
      Suppose that $r_i > 0$ for all $i \in \Z / f\Z$.  The $\mathrm{GL}_2(\mathcal{O})$-modules $R(\sigma_r)$ and $M_{(\varnothing, {\Z / f\Z})}$ are isomorphic.
      \end{lem}
     \begin{proof}
     Let $I \subseteq \{ i \in \Z / f\Z : \, r_i > 0 \}$.  A Jordan-H\"older constituent $\det^{p^I} \tensor \sigma_J(r - 2p^I)$ of $W_{p^I} / \sum_{I^\prime \subset I} W_{p^{I^\prime}}$ is special in the sense of~\cite[Definition~17.2]{BP/12}, where we have implicitly applied the isomorphism of Proposition~\ref{pro:filtration} and the identification of Lemma~\ref{lem:bp.filtration}, if $t_{J,i}(r - 2p^I) \in \{ (r - 2p^I)_i, p - 2 - (r - 2p^I)_i \}$ for all $i \in I$.  By Remark~\ref{rem:serre.weight.data} this exactly means that $(I - 1) \cap J = \varnothing$.  Under our hypotheses on $r$, the subset $I$ can be any subset of $\Z / f\Z$, and thus $\det^{p^I} \tensor \sigma_J(r - 2p^I)$ is special if and only if $(J, I) \in \widetilde{\Theta}_1$.  By the definition of $R(\sigma_r)$ and Proposition~\ref{pro:unramified.irreducible.cosocle}, we have $R(\sigma_r) \simeq \sum_{\theta \in \widetilde{\Theta}_1} M_\theta$.  Since $(\varnothing, {\Z / f\Z})$ is the unique maximal element of $\widetilde{\Theta}_1$, by Proposition~\ref{pro:unramified.inclusion} the sum on the right-hand side is just $M_{(\varnothing, {\Z / f\Z})}$.
     \end{proof}
     
     \subsection{Submodule structure}
     Before giving a complete description of the submodule structure of $R(\sigma_r) \simeq M_{(\varnothing, {\Z / f\Z})}$, we state an auxiliary lemma.
     
     \begin{lem} \label{lem:unramified.auxiliary}
     Suppose that $r$ is generic and $M_{(\varnothing, {\Z / f\Z})}$ admits a non-split extension $\mathcal{E}$ of two Serre weights $\sigma_{\theta}$ and $\sigma_{\theta^\prime}$, for $\theta, \theta^\prime \in \widetilde{\Theta}_1$, as a subquotient:
     $ 0 \to \sigma_{\theta^\prime} \to \mathcal{E} \to \sigma_\theta \to 0.$  Then $\theta^\prime \sqsubset \theta$.
     \end{lem}
     \begin{proof}
     If $f = 1$, then it is simple to work out the submodule structure of $M_{(\varnothing, \Z / f\Z)}$ directly.  It is a uniserial module of length three:
     \begin{eqnarray*}
     \mathrm{soc} (M_{(\varnothing, \Z / f\Z)}) & \simeq & \sigma_{\varnothing}(r) = \mathrm{Sym}^r E^2 \\
     \mathrm{rad} (M_{(\varnothing, \Z / f\Z)}) & \simeq & W_0 \simeq \mathrm{Ind}^{\mathrm{GL}_2(\mathbb{F}_p)}_{B_2(\mathbb{F}_p)} \chi_r \\
     M_{(\varnothing, \Z / f\Z)} /  \mathrm{rad} (M_{(\varnothing, \Z / f\Z)}) & \simeq & \det \tensor \sigma_{\varnothing}(r - 2) = \det \tensor \mathrm{Sym}^{r-2} E^2.
     \end{eqnarray*}
     In particular, the claim holds.  So we assume $f > 1$.
     
     If $\mathrm{Ext}^1_{\mathrm{GL}_2(\mathcal{O})}(\sigma_{\theta}, \sigma_{\theta^\prime}) \neq 0$, then the Serre weights $\sigma_\theta$ and $\sigma_{\theta^\prime}$ must satisfy one of five conditions specified by~\cite[Corollary~5.6]{BP/12}.  The conditions (i)(a) and (i)(b), in the labeling of~\cite{BP/12}, amount to the union of the following possibilities, where $I, J \subseteq \Z / f\Z$ satisfy $J \cap (I - 1) = \varnothing$:
     \begin{itemize}
     \item
     $\{ \theta, \theta^\prime \} = \{ (J, I), (J \cup \{ j \}, I) \}$ for some $j \not\in J$;
     \item
     $\{ \theta, \theta^\prime \} = \{ (J, I), (J \cup \{ i - 1 \}, {I \setminus \{ i \}}) \}$ for some $i \in I$.
     \end{itemize}
     The conditions (ii)(a), (ii)(b), and (ii)(c) amount to the following possibilities:
     \begin{itemize}
     \item
     $\{ \theta, \theta^\prime \} = \{ (J, {I}), (J, {I \setminus \{ i \} }) \}$ for some $i \in I$;
     \item
     $\theta = \theta^\prime$.
     \end{itemize}
     The module $M_{(\varnothing, {\Z / f\Z})}$ is multiplicity-free, so $\theta = \theta^\prime$ is impossible; see also~\cite[Proposition~2.10]{Hu/10}.  For the other possibilities, either $\theta^\prime \sqsubset \theta$ or $\theta \sqsubset \theta^\prime$.  If $\theta \sqsubset \theta^\prime$, then $\mathcal{E}$ is a subquotient of $M_{\theta^\prime}$.  But Proposition~\ref{pro:unramified.irreducible.cosocle} implies that $M_{\theta^\prime}$ cannot admit a reducible subquotient containing $\sigma_{\theta^\prime}$ in its socle, giving rise to a contradiction.  Therefore, the relation $\theta^\prime \sqsubset \theta$ necessarily holds.
     \end{proof}

     \begin{thm} \label{thm:unramified.submodule.structure}
     Suppose that $r$ is generic.
     \begin{enumerate}
     \item
     The $\mathrm{GL}_2(\mathcal{O} / \mathfrak{m}^2)$-module $M_{(\varnothing, {\Z / f\Z})}$ is multiplicity-free, and the set of its Jordan-H\"older constituents is
     $$ \mathrm{JH}(M_{(\varnothing, {\Z / f\Z})}) = \{ \sigma_\theta: \, \theta \in \widetilde{\Theta}_1 \}.$$
     \item
     Let $\theta = (J, I) \in \widetilde{\Theta}_1$.  Then $M_\theta$ is the unique submodule of $M_{(\varnothing, {\Z / f\Z})}$ with cosocle $\sigma_\theta$, and its Jordan-H\"older constituents are
     $$ \mathrm{JH}(M_\theta) = \{ \sigma_{\theta^\prime} : \, \theta^\prime \in \widetilde{\Theta}_1, \, \theta^\prime \sqsubseteq \theta \}.$$
     \end{enumerate}
     \end{thm}
     \begin{proof}
     If $\theta = (J, I) \in \widetilde{\Theta}_1$, then $\sigma_\theta$ is a constituent of $M_\theta$ by Proposition~\ref{pro:unramified.irreducible.cosocle} and hence of $M_{(\varnothing, {\Z / f\Z})}$ by Proposition~\ref{pro:unramified.inclusion}.  To prove the first statement, we must show that $M_{(\varnothing, {\Z / f\Z})}$ has no other Jordan-H\"older constituents.   This is the content of~\cite[Lemma~17.8]{BP/12}; recall our characterization of special subquotients in the course of the proof of Lemma~\ref{lem:r.sigma.characterization}.
     
     Now consider the second part of the claim.  Again by Proposition~\ref{pro:unramified.inclusion} we know that all the Serre weights contained in the right-hand side are indeed Jordan-H\"older constituents of $M_\theta$.  If our claim were false, there would exist $\theta^\prime, \theta^\dpr \in \widetilde{\Theta}_1$ such that $\theta^\dpr \not\sqsubseteq \theta^\prime$ and yet $M_\theta$ admits as a subquotient a non-split extension $\mathcal{E}$ of the form $0 \to \sigma_{\theta^\dpr} \to \mathcal{E} \to \sigma_{\theta^\prime} \to 0$, contradicting Lemma~\ref{lem:unramified.auxiliary}.
     \end{proof}
     
     We immediately deduce an analogue of Proposition~\ref{pro:adjacent.types}.
     \begin{cor} \label{cor:unramified.extension}
     Suppose that $r$ is generic and $\theta^\prime \sqsubset \theta$ are adjacent elements of $\widetilde{\Theta}_1$, with respect to the partial order $\sqsubseteq$.  Then a non-split extension of $\sigma_\theta$ by $\sigma_{\theta^\prime}$ arises as a subquotient of $M_{(\varnothing, \Z / f\Z)}$.
     \end{cor}
     
      \begin{rem}
     The ``extension lemma''~\cite[Lemma~18.4]{BP/12} is a crucial ingredient in the main results of~\cite{BP/12}.  Corollary~\ref{cor:unramified.extension} should be viewed as a strengthening of this lemma, specifying which of the two non-split extensions occurs in $R(\sigma)$.  In fact, the proof, which is less intricate than the argument in~\cite{BP/12}, involves only Propositions~\ref{pro:unramified.irreducible.cosocle} and~\ref{pro:unramified.inclusion} and a translation of conditions on Serre weights into conditions on elements of $\widetilde{\Theta}_1$ as in the proof of Lemma~\ref{lem:unramified.auxiliary}.
          \end{rem}

       \section{Proof of Proposition~\ref{pro:equality.of.orders}} \label{sec:equality.of.orders}
       In this section we prove a closed-form description of the partial order $\leq_r$ on the set $\widetilde{\Theta}$ of types of the elements of the basis $\mathcal{B}_2$ of $V_{2,r}$, for ramified extensions $F / \Q_p$.  This partial order was originally introduced in Definition~\ref{def:covering.relations.ramified} in terms of generating relations.  In this section we refer to any pair $(I, \gamma) \in \mathcal{P}(\Z / f\Z) \times \Nt$ as a type, and the notion of $r$-admissible types will distinguish the pairs arising from equivalence classes in $\widetilde{\Theta}$ by the correspondence $\Upsilon$ of Proposition~\ref{cor:types.as.pairs}; see Definition~\ref{def:admissible.type} and Remark~\ref{rem:admissible.types} below.
       
       \begin{dfn}
       Let $(I^\prime, \gamma^\prime)$ and $(I, \gamma)$ be two types.  We say that $(I^\prime, \gamma^\prime) \preceq_r (I, \gamma)$ if the following conditions are satisfied:
       \begin{itemize}
\item
$\gamma^\prime \preceq \gamma$;
\item
$I^\prime \subseteq I \cup I(r - 2\gamma, 2(\gamma \dotdiv \gamma^\prime)) \cup I(\gamma \dotdiv \gamma^\prime, \gamma \dotdiv \gamma^\prime)$.
\end{itemize}
       \end{dfn}
       \begin{lem}
       The relation $\preceq_r$ is a partial order.
       \end{lem}
       \begin{proof}
       Observe that $(I^\prime, \gamma) \preceq_r (I, \gamma)$ if and only if $I^\prime \subseteq I$.  It remains to show that the relation $\preceq_r$ is transitive.  Suppose $(
I^{\prime\prime}, \gamma^{\prime \prime})\preceq_r (I^\prime, \gamma^\prime)\preceq_r (I, \gamma)$.  Then clearly $\gamma^{\prime\prime}\preceq\gamma$.  By our assumptions,
$$ I^{\prime \prime} \subseteq I \cup I(r-2\gamma^\prime,2(\gamma^\prime \dotdiv \gamma^{\prime\prime}))\cup I(\gamma^\prime \dotdiv \gamma^{\prime\prime},\gamma^\prime \dotdiv \gamma^{\prime\prime})\cup I(r-2\gamma,2(\gamma \dotdiv \gamma^\prime))\cup I(\gamma \dotdiv \gamma^\prime,\gamma \dotdiv \gamma^\prime),$$
so it suffices to prove that
\begin{multline*}
I(r-2\gamma^\prime,2(\gamma^\prime \dotdiv \gamma^{\prime\prime}))\cup I(\gamma^\prime \dotdiv \gamma^{\prime\prime},\gamma^\prime \dotdiv \gamma^{\prime\prime})\cup I(r-2\gamma,2(\gamma \dotdiv \gamma^\prime))\cup I(\gamma \dotdiv \gamma^\prime,\gamma \dotdiv \gamma^\prime) \subseteq \\
I(r-2\gamma,2(\gamma \dotdiv \gamma^{\prime\prime})) \cup I(\gamma \dotdiv \gamma^{\prime\prime},\gamma \dotdiv \gamma^{\prime\prime}).
\end{multline*}
In fact, this is an equality.  Set
$(\varepsilon_1, \dots, \varepsilon_5) = (r - 2 \gamma, \gamma \dotdiv \gamma^\prime, \gamma \dotdiv \gamma^\prime, \gamma^\prime \dotdiv \gamma^{\prime \prime}, \gamma^\prime \dotdiv \gamma^{\prime \prime})$.  Then we indeed obtain
\begin{multline*}
I(r-2\gamma^\prime,2(\gamma^\prime \dotdiv \gamma^{\prime\prime}))\cup I(\gamma^\prime \dotdiv \gamma^{\prime\prime},\gamma^\prime \dotdiv \gamma^{\prime\prime})\cup I(r-2\gamma,2(\gamma \dotdiv \gamma^\prime))\cup I(\gamma \dotdiv \gamma^\prime,\gamma \dotdiv \gamma^\prime) = \\ 
I(\varepsilon_2, \varepsilon_3) \cup I(\varepsilon_4, \varepsilon_5) \cup I(\varepsilon_1, \varepsilon_2 + \varepsilon_3) \cup I(\varepsilon_1 + \varepsilon_2 + \varepsilon_3, \varepsilon_4 + \varepsilon_5) = \\
I(\varepsilon_2, \varepsilon_4) \cup I(\varepsilon_3, \varepsilon_5) \cup I(\varepsilon_2 + \varepsilon_4, \varepsilon_3 + \varepsilon_5) \cup I(\varepsilon_1, \varepsilon_2 + \varepsilon_3 + \varepsilon_4 + \varepsilon_5) = \\
I(\varepsilon_2 + \varepsilon_4, \varepsilon_3 + \varepsilon_5) \cup I(\varepsilon_1, \varepsilon_2 + \varepsilon_3 + \varepsilon_4 + \varepsilon_5) = I(r-2\gamma,2(\gamma-\gamma^{\prime\prime})) \cup I(\gamma \dotdiv \gamma^{\prime\prime},\gamma \dotdiv \gamma^{\prime\prime}),
\end{multline*}
where the second equality is an application of Lemma~\ref{lem:carries.urlemma} and the third holds because $I(\varepsilon_2, \varepsilon_4) = I(\varepsilon_3, \varepsilon_5) = I(\gamma \dotdiv \gamma^{\prime}, \gamma^\prime \dotdiv \gamma^{\prime \prime}) = \varnothing$ since $\gamma^{\prime \prime} \preceq \gamma^\prime \preceq \gamma$.
\end{proof}

Our aim now is to show that each of the partial orders $\preceq_r$ and $\leq_r$ refines the other.

\begin{lem} \label{lem:one.refinement.partial.orders}
The partial order $\preceq_r$ is a refinement of $\leq_r$.
\end{lem}
\begin{proof}
We must show that if $(I^\prime, \gamma^\prime) \leq_r (I, \gamma)$ by one of the generating relations of Definition~\ref{def:covering.relations.ramified}, then $(I^\prime, \gamma^\prime) \preceq_r (I, \gamma)$.  Note that the relation $(I^\prime, \gamma) \leq_r (I, \gamma)$ is equivalent, by definition, to the inclusion $V_{(I^\prime, \gamma)} \subseteq V_{(I, \gamma)}$ of submodules of $V_{2,r}$.  As in the proof of Proposition~\ref{cor:types.as.pairs}, we see using Proposition~\ref{pro:filtration} and the results of~\cite{BS/00} that this inclusion of submodules is equivalent to the inclusion $I^\prime \subseteq I$.  Thus we see that $(I^\prime, \gamma) \leq_r (I, \gamma)$ if and only if $I^\prime \subseteq I$, and it is evident from the definition of $\preceq_r$ that this is equivalent to $(I^\prime, \gamma) \preceq_r (I, \gamma)$.  This implies our claim for the first, fourth, sixth, and seventh generating relations of Definition~\ref{def:covering.relations.ramified}.

The second generating relation states that $[(j_0, j_1 \dotdiv p^m)] \leq_r [(j_0, j_1)]$ if $p^m \preceq j_1$.  By the map $\Upsilon$ of Proposition~\ref{cor:types.as.pairs}, this is equivalent to $(I(j_0, r - 2j_1 - j_0 + 2 p^m), j_1 \dotdiv p^m) \leq_r (I(j_0, r - 2j_1 - j_0), j_1)$.  By Corollary~\ref{lem:carries.urlemma} we have
\begin{multline*}
I(j_0, r - 2j_1 - j_0 + 2 p^m) \subseteq I(j_0, r - 2j_1 - j_0 + 2 p^m) \cup I(r - 2j_1 - j_0, 2p^m) = \\
I(j_0, r - 2j_1 - j_0) \cup I(r - 2j_1, 2p^m),
\end{multline*}
and this implies $(I(j_0, r - 2j_1 - j_0 + 2 p^m), j_1 \dotdiv p^m) \preceq_r (I(j_0, r - 2j_1 - j_0), j_1)$ by the definition of $\preceq_r$.

The third generating relation is that $[(j_0 + p^m, j_1 \dotdiv p^m)] \leq_r [(j_0, j_1)]$ if $p^m \preceq j_1$, or, equivalently, that $(I(j_0 + p^m, r - 2j_1 - j_0 + p^m), j_1 \dotdiv p^m) \leq_r (I(j_0, r - 2j_1 - j_0), j_1)$.  Now, by Corollary~\ref{lem:carries.urlemma} we have
\begin{multline*}
 I(j_0 + p^m, r - 2j_1 + p^m) \subseteq I(j_0 + p^m, r - 2j_1 + p^m) \cup I(j_0, p^m) \cup I(r - 2j_1 - j_0, p^m) = \\
 I(r - 2j_1 - j_0, j_0) \cup I(r - 2j_1, 2 p^m) \cup I(p^m, p^m),
 \end{multline*}
 which is equivalent to $[(j_0 + p^m, j_1 \dotdiv p^m)] \preceq_r [(j_0, j_1)]$.

It remains to treat the fifth generating relation.  Abusing notation, we write the correspondence of Proposition~\ref{cor:types.as.pairs} as an equality.  Then it is clear from the definition of $\preceq_r$ that if $p^m \preceq j_1$, then
$$
[(r-2j_1, j_1 \dotdiv p^m)] = (I(r-2j_1, 2p^m), j_1 \dotdiv p^m) \preceq_r (\varnothing , j_1) = [(0,j_1)] = [(\infty, j_1)],
$$
which is exactly what we need.  We have now finished checking the generating relations of Definition~\ref{def:covering.relations.ramified} and can conclude that $\preceq_r$ refines $\leq_r$.
\end{proof}

The proof of the opposite inclusion of relations, namely that $\leq_r$ refines $\preceq_r$, is substantially more involved than that of Lemma~\ref{lem:one.refinement.partial.orders}.  The primary reason for this is that the correspondence $\Upsilon$ of Proposition~\ref{cor:types.as.pairs} need not be surjective.  The relation $\preceq_r$ is defined in terms of pairs $(I, \gamma)$, but when we work with such pairs, we must take care to remain inside the image of $\Upsilon$, so that we can translate to the equivalence classes of pairs $(j_0, j_1) \in \Theta$ in terms of which $\leq_r$ is defined.

\begin{dfn} \label{def:admissible.type}
A type $(I, \gamma) \in \mathcal{P}(\Z / f\Z) \times \Nt$ is called {\emph{$r$-admissible}} if $I$ is $(r-2\gamma)$-admissible in the sense of Definition~\ref{def:admissible.set}.
\end{dfn}

\begin{rem} \label{rem:admissible.types}
Equivalently, a type $(I, \gamma)$ is $r$-admissible if and only if it lies in the image of the correspondence $\Upsilon$ of Proposition~\ref{cor:types.as.pairs}, i.e. if there exists $(j_0, j_1) \in \Nt^2$ such that $(I, \gamma) = (I(j_0, r - 2j_1 - j_0), j_1)$.
\end{rem}

Denote the set of $r$-admissible types by $\mathcal{H}_r \subseteq \mathcal{P}(\Z / f\Z) \times \Nt$.  The restriction of $\preceq_r$ gives a partial order on $\mathcal{H}_r$.

To prove that $\leq_r$ refines $\preceq_r$, we need to show that if $(I^\prime, \gamma^\prime) \preceq_r (I, \gamma)$ are two $r$-admissible types, then they are connected in the poset $\mathcal{H}_r$ by a path of a specific form compatible with the generating relations of $\leq_r$.  The next claim is due to Bardoe and Sin~\cite[Corollary~4.1]{BS/00}; we reprove it here in our terminology.

\begin{lemma} \label{lem:bardoe.sin.chains}
Let $\gamma \in \Nt$ and let $(I^\prime, \gamma) \prec_r (I, \gamma)$ be two $r$-admissible types.  Then there exists $j \in I \setminus I^\prime$ such that $( I^\prime \cup \{ j \}, \gamma) \in \mathcal{H}_r$.
\end{lemma}
\begin{proof}
We may assume $r - 2 \gamma \neq q - 1$, since otherwise the claim is vacuously true.  Observe that $I^\prime \subset I$.  If $j \in I \setminus I^\prime$ and $(I^\prime \cup \{ j \}, \gamma) \not\in \mathcal{H}_r$ then it is immediate from Definition~\ref{def:admissible.type} that one of the following two options must hold:
\begin{itemize}
\item $(r - 2 \gamma)_{j+1} = 0$ and $j + 1 \not\in I^\prime$;
\item $(r - 2 \gamma)_{j} = p - 1$ and $j - 1 \not\in I^\prime$.
\end{itemize}
Since $(I,\gamma) \in \mathcal{H}_r$ by assumption, it follows that if our claim is false, then for all $j \in I \setminus I^\prime$ one of the following two alternatives must hold:
\begin{itemize}
\item $(r - 2 \gamma)_{j+1} = 0$ and $j + 1 \in I \setminus I^\prime$;
\item $(r - 2 \gamma)_j = p - 1$ and $j - 1 \in I \setminus I^\prime$.
\end{itemize}
If $j \in I \setminus I^\prime$ satisfies the first alternative, then $j + 1 \in I \setminus I^\prime$ necessarily also satisfies the same condition.  Hence $(r - 2 \gamma)_j = 0$ for all $j \in \Z / f\Z$.  If $j \in I \setminus I^\prime$ satisfies the second alternative, then $(r - 2 \gamma)_j = p - 1$ for all $j \in \Z / f\Z$ by a similar argument.  In either case, this contradicts the hypothesis $r - 2 \gamma \neq q - 1$.
\end{proof}

\begin{lemma} \label{lem:ramified.chain.jump}
Suppose that $(I^\prime, \gamma^\prime) \preceq_r (I, \gamma)$ are two $r$-admissible types.  Then there exist $j \in \Z / f\Z$ such that $\gamma^\prime_j < \gamma_j$ and a subset $I^\dpr \subseteq \Z / f\Z$ such that $(I^\dpr, \gamma^\prime + p^j) \in \mathcal{H}_r$ and
\begin{equation} \label{equ:triple.containment}
(I^\prime, \gamma^\prime) \preceq_r (I^\dpr, \gamma^\prime + p^j) \preceq_r (I, \gamma).
\end{equation}
Moreover, $I^\dpr$ may be taken to be the set $I^\prime \setminus I(r - 2 \gamma^\prime - 2 p^j, 2 p^j)$, except when $(r - 2 \gamma^\prime)_j = 2$ and $j - 1 \in I^\prime$ and $j \not\in I^\prime$ for all $j$ such that $\gamma^\prime_j < \gamma_j$.  In this exceptional case, we may take $I^\dpr = I^\prime \cup \{j, j+1, \dots, j+ \ell \}$, where $\ell \geq 0$ is maximal such that $(r - 2 \gamma^\prime - 2p^j)_{j + i} = 0$ and $j + i \not\in I^\prime$ for all $i \in [\ell]_0$.
\end{lemma}
\begin{proof}
For any $j$ such that $\gamma^\prime_j < \gamma_j$, observe that~\eqref{equ:triple.containment} amounts to the following three conditions:
\begin{enumerate}[label=(\alph*)]
\item $I^\prime \subseteq I \cup I(r - 2 \gamma, 2(\gamma \dotdiv \gamma^\prime)) \cup I(\gamma \dotdiv \gamma^\prime, \gamma \dotdiv \gamma^\prime)$;
\item $I^\prime \subseteq I^\dpr \cup I(r - 2\gamma^\prime - 2 p^j, 2 p^j)$;
\item $I^\dpr \subseteq I \cup I(r - 2 \gamma, 2 (\gamma \dotdiv \gamma^\prime \dotdiv p^j)) \cup I(\gamma \dotdiv \gamma^\prime \dotdiv p^j, \gamma \dotdiv \gamma^\prime \dotdiv p^j)$.
\end{enumerate}
By assumption, we have $\gamma^\prime \preceq \gamma \dotdiv p^j$, and hence $I(\gamma \dotdiv \gamma^\prime \dotdiv p^j, p^j) = \varnothing$.  Moreover, since $p > 2$, we have $I(p^j, p^j) = \varnothing$.
It follows from Lemma~\ref{lem:carries.urlemma} that
\begin{multline} \label{equ:insert.carry.relation}
I(\gamma \dotdiv \gamma^\prime \dotdiv p^j, \gamma \dotdiv \gamma^\prime \dotdiv p^j) \cup I(r - 2\gamma, 2(\gamma \dotdiv \gamma^\prime \dotdiv p^j)) \cup I(r - 2 \gamma^\prime - 2 p^j, 2 p^j) = \\ I(\gamma \dotdiv \gamma^\prime, \gamma \dotdiv \gamma^\prime) \cup I(r - 2\gamma, 2 (\gamma \dotdiv \gamma^\prime))
\end{multline}
and hence that the set $I^\dpr = I^\prime \setminus I(r - 2 \gamma^\prime - 2 p^j, 2 p^j)$ satisfies the three conditions above.
We now consider several cases.  

Suppose first that $I(r - 2 \gamma^\prime - 2 p^j, 2 p^j) \neq \varnothing$.  In this case, we have $(r - 2 \gamma^\prime - 2 p^j)_j \geq p - 2$.  Let $\ell \in [f-1]_0$ be maximal such that $(r - 2 \gamma^\prime - 2 p^j)_{j+i} = p - 1$ for all $i \in [\ell]$; recall that $[\ell] = \varnothing$ when $\ell = 0$.  In particular, if $r - 2 \gamma^\prime - 2 p^j = q - 1$ then $\ell = f-1$.  We then have $I(r - 2 \gamma^\prime - 2 p^j, 2 p^j) = \{ j + i : i \in [\ell]_0 \}$.  Set $I^\dpr = I^\prime \setminus I(r - 2 \gamma^\prime - 2 p^j, 2 p^j)$ and check that the type $(I^\dpr, \gamma^\prime + p^j)$ is $r$-admissible since $(I^\prime, \gamma^\prime)$ is.  Since we already know that $(I^\dpr, \gamma^\prime + p^j)$ satisfies~\eqref{equ:triple.containment}, we have established our claim in this case.

Now suppose that $I(r - 2\gamma^\prime - 2 p^j, 2 p^j) = \varnothing$.  Then the digits of $r-2 \gamma^\prime$ are the same as those of $r - 2 \gamma^\prime - 2 p^j$, except for $(r - 2 \gamma^\prime - 2 p^j)_j = (r - 2 \gamma^\prime)_j - 2$.  If we set $I^\dpr = I^\prime \setminus I(r - 2 \gamma^\prime - 2 p^j, 2 p^j) = I^\prime$, then the $r$-admissibility of $(I^\prime, \gamma^\prime)$ implies that of $(I^\prime, \gamma^\prime + p^j)$, except in the case where $(r - 2 \gamma^\prime - 2 p^j)_j = 0$ and $j \not\in I^\prime$ but $j - 1 \in I^\prime$.  

So suppose that this problematic situation holds for all $j$ such that $\gamma^\prime_j < \gamma_j$, and fix one such $j$.  Let $\ell \geq 0$ be maximal such that $(r - 2 \gamma^\prime - 2 p^j)_{j+i} = 0$ and $j + i \not\in I^\prime$ for all $i \in [\ell]_0$.  Note that this condition implies $\gamma^\prime_{j+i} = \gamma_{j+i}$ for all $i \in [\ell]$, and also that $\gamma^\prime_{j-1} = \gamma_{j-1}$ and hence $j - 1 \not\in I(\gamma \dotdiv \gamma^\prime \dotdiv p^j, \gamma \dotdiv \gamma^\prime \dotdiv p^j)$.  Set $I^\dpr = I^\prime \cup \{ j + i : i \in [\ell]_0 \}$ and observe that $(I^\dpr, \gamma^\prime + p^j)$ is $r$-admissible.
To complete the proof of this lemma, it remains to show that $I^\dpr$ satisfies the condition~(c) above.  By~\eqref{equ:insert.carry.relation} and the condition~(a), it suffices to show that 
\begin{equation} \label{equ:insert.carry.condition}
\{ j + i : i \in [\ell]_0 \} \subseteq I \cup I(r - 2\gamma, 2(\gamma \dotdiv \gamma^\prime \dotdiv p^j)) \cup I(\gamma \dotdiv \gamma^\prime \dotdiv p^j, \gamma \dotdiv \gamma^\prime \dotdiv p^j).
\end{equation}
If $j \in I(r - 2 \gamma, 2(\gamma \dotdiv \gamma^\prime \dotdiv p^j))$, then $\{ j + i : i \in [\ell]_0 \} \subseteq I(r - 2 \gamma, 2(\gamma \dotdiv \gamma^\prime \dotdiv p^j))$ by Remark~\ref{rem:admissible.types} and the $r$-admissibility of $(I^\dpr, \gamma^\prime + p^j)$, and we are done.
So suppose $j \not\in I(r - 2 \gamma, 2(\gamma \dotdiv \gamma^\prime \dotdiv p^j))$.  Then necessarily $(r-2 \gamma)_j = (2(\gamma \dotdiv \gamma^\prime \dotdiv p^j))_j = 0$.  Applying Remark~\ref{rem:admissible.types} again, we may conclude that $j - 1 \not\in I(r - 2 \gamma, 2(\gamma \dotdiv \gamma^\prime \dotdiv p^j))$, and we already know $j - 1 \not\in I(\gamma \dotdiv \gamma^\prime \dotdiv p^j, \gamma \dotdiv \gamma^\prime \dotdiv p^j)$.  But $j - 1 \in I^\prime$ by assumption, which forces $j - 1 \in I$ by~\eqref{equ:insert.carry.relation} and the condition~(a).  Since $(r - 2\gamma)_j = 0$ we obtain $j \in I$ by Remark~\ref{rem:admissible.types}.  We have thus established that $j$ is contained in the right-hand side of~\eqref{equ:insert.carry.condition}.  

Now we iterate this argument.  Suppose that $i \in [\ell]$ and it is known that $j + i^\prime \in I$ for all $i^\prime \in [i-1]_0$.  If $j + i \in I(r - 2 \gamma, 2(\gamma \dotdiv \gamma^\prime \dotdiv p^j))$, then $j + m \in I(r - 2 \gamma, 2(\gamma \dotdiv \gamma^\prime \dotdiv p^j))$ for all $i \leq m \leq \ell$, and we are done.  If not, then necessarily $(r - 2\gamma)_{j+i} = 0$.  Hence $j + i \in I$ by the $r$-admissibility of the type $(\gamma, I)$, and we may apply this argument again for $i + 1$.  After at most $\ell$ iterations, we complete the proof of the remaining case of the lemma.
\end{proof}

\begin{cor} \label{cor:ramified.chain.jump}
Let $(I^\prime, \gamma^\prime) \prec_r (I, \gamma)$ be two $r$-admissible types.  Then there exists a finite sequence
$$ (I_0, \gamma_0) \prec_r (I_1, \gamma_1) \prec_r \cdots \prec_r (I_s, \gamma_s)$$
of $r$-admissible types such that for each $i \in [s]$ one of the following statements holds:
\begin{enumerate}[label=(\alph*)]
\item $\gamma_{i} = \gamma_{i-1}$ and $I_i = I_{i-1} \cup \{ j \}$ for some $j \not\in I_{i-1}$, or
\item $\gamma_i = \gamma_{i - 1} + p^j$ for some $j \in [f-1]_0$ and $I_{i - 1} = I_i \cup I(r - 2 \gamma_{i}, 2 p^j)$.
\end{enumerate}
\end{cor}
\begin{proof}
We argue by induction on $\gamma \dotdiv \gamma^\prime$, with respect to the partial order $\preceq$ on $\Nt$.  If $\gamma \dotdiv \gamma^\prime = 0$, then our claims follows easily from Lemma~\ref{lem:bardoe.sin.chains}.  So assume that $\gamma \dotdiv \gamma^\prime > 0$ and obtain a sequence $(I^\prime, \gamma^\prime) \preceq_r (I^\dpr, \gamma^\prime + p^j) \preceq_r (I, \gamma)$ of $r$-admissible types as in Lemma~\ref{lem:ramified.chain.jump}.  If we may take $I^\dpr = I^\prime \setminus I(r - 2 \gamma^\prime - 2 p^j, 2 p^j)$, then the segment $(I^\prime, \gamma^\prime) \preceq_r (I^\dpr, \gamma^\prime + p^j)$ satisfies~(b), whereas the segment $(I^\dpr, \gamma^\prime + p^j) \preceq_r (I, \gamma)$ may be refined to a sequence with the desired properties by induction.  

In the remaining exceptional case of Lemma~\ref{lem:ramified.chain.jump}, we have $I^\dpr = I^\prime \cup \{ j, j+1, \dots, j+ \ell \}$ for a suitable $\ell \geq 0$.  In this case either $j + \ell + 1 \in I^\prime$ or $(r - 2 \gamma^\prime)_{j + \ell + 1} > 0$, so it is easy to check that $(I^\dpr, \gamma^\prime)$ is an $r$-admissible type.
Thus we obtain the refinement $(I^\prime, \gamma^\prime) \preceq_r (I^\dpr, \gamma^\prime) \preceq_r (I^\dpr,\gamma^\prime + p^j) \preceq_r (I, \gamma)$.  Since $I(r - 2(\gamma^\prime + p^j), 2 p^j) = \varnothing$ in this case, the second jump in this refinement satisfies~(b), whereas the first and third jumps may be refined to sequences with the desired properties by induction.
\end{proof}

The next lemma, whose proof is tedious but completely elementary, relates the refinements of Corollary~\ref{cor:ramified.chain.jump} with the types of specific basis elements of $V_{2,r}$.  This will finally allow us to relate our conclusions to the partial order $\leq_r$.

\begin{lemma} \label{lem:induction.step.jumps}
Let $(I^\prime, \gamma^\prime) \prec_r (I, \gamma)$ be two $r$-admissible types such that $\gamma^\prime \prec \gamma$.  Let $\alpha, \beta \in \Nt$ satisfy $I = I(\alpha, \beta)$.  Then there exists $j \in \Z / f\Z$ such that at least one of the following statements holds:
\begin{itemize}
\item
$(I(\alpha + 2 p^j, \beta), \gamma \dotdiv p^j) \in \mathcal{H}_r$ and $(I^\prime, \gamma^\prime) \preceq_r ( I(\alpha + 2 p^j, \beta), \gamma \dotdiv p^j)$.
\item
$(I(\alpha + p^j, \beta + p^j),\gamma \dotdiv p^j) \in \mathcal{H}_r$ and $(I^\prime, \gamma^\prime) \preceq_r (I(\alpha + p^j, \beta + p^j), \gamma \dotdiv p^j)$.
\item
$(I(\alpha, \beta + 2p^j), \gamma \dotdiv p^j) \in \mathcal{H}_r$ and $(I^\prime, \gamma^\prime) \preceq_r (I(\alpha, \beta + 2 p^j), \gamma \dotdiv p^j)$.
\end{itemize}
\end{lemma}
\begin{proof}
By Corollary~\ref{cor:ramified.chain.jump} there exists $j \in \Z / f\Z$ and an $r$-admissible type $(I^\dpr, \gamma \dotdiv p^j)$ such that $(I^\prime, \gamma^\prime) \preceq_r (I^\dpr, \gamma \dotdiv p^j) \preceq_r (I, \gamma)$.  Observe that the second relation amounts to $I^\dpr \subseteq I \cup I(\alpha + \beta, 2 p^j)$.  It suffices to show that at least one of the three sets $I(\alpha + 2 p^j, \beta)$, $I(\alpha + p^j, \beta + p^j)$, $I(\alpha, \beta + 2 p^j)$ is the largest subset $I^\dpr$ of $I \cup I(\alpha + \beta, 2 p^j)$ such that $(I^\dpr, \gamma \dotdiv p^j)$ is $r$-admissible.  We consider several cases, which exhaust all possibilities.

{\emph{Case 1: $\alpha_j < p - 2$ or $\beta_j < p - 2$}}.  If $\alpha_j < p - 2$, then $I(\alpha, 2 p^j) = \varnothing$.  Hence it follows from Lemma~\ref{cor:carries.triple.sum} that 
\begin{equation} \label{equ:case1}
I(\alpha + 2 p^j, \beta) = I(\alpha, \beta) \cup I(\alpha + \beta, 2 p^j).
\end{equation}  
In particular, $(I \cup I(\alpha + \beta, 2 p^j), \gamma \dotdiv p^j) = ( I(\alpha + 2 p^j, \beta), \gamma \dotdiv p^j)$ is itself an $r$-admissible type.  The case $\beta_j < p - 2$ is analogous.

{\emph{Case 2: $\alpha_j = \beta_j = p - 2$}}.  In this case $I(\alpha, p^j) = I(\beta, p^j) = \varnothing$.  By Corollary~\ref{lem:carries.urlemma} we may conclude that $I(\alpha, \beta) \cup I(\alpha + \beta, 2 p^j) = I(\alpha + p^j, \beta + p^j)$.  Hence $(I \cup I(\alpha + \beta, 2 p^j), \gamma \dotdiv p^j) = (I(\alpha + p^j, \beta + p^j), \gamma \dotdiv p^j)$ is an $r$-admissible type.

{\emph{Case 3: $\{ \alpha_j, \beta_j \} = \{ p - 1, p - 2 \}$ and $j - 1 \not\in I$}}.  Observe that $I(\alpha + \beta, 2 p^j) = \varnothing$, so~\eqref{equ:case1} implies that $I \cup I(\alpha + \beta, 2 p^j) = I$.  However, since $(\alpha + \beta + 2p^j)_j = p - 1$, the type $(I, \gamma \dotdiv p^j)$ fails to be admissible if $j - 1 \not\in I$ but $j \in I$.  In this case, it is clear that the largest subset $I^\dpr \subseteq I$ such that $(I^\dpr, \gamma \dotdiv p^j)$ is $r$-admissible is $I^\dpr = I \setminus \{ j + i : i \in [\ell]_0 \}$, where $\ell \in [f-1]_0$ is maximal such that $(\alpha + \beta + 2 p^j)_{j + i} = p - 1$ and $j + i \in I$ for all $i \in [\ell]_0$.  Observe that $\alpha_{j + i} = \beta_{j + i} = p - 1$ for all $i \in [\ell]$.  If $\alpha_{j + \ell + 1} = \beta_{j + \ell + 1} = p - 1$, then the maximality of $\ell$ forces $j + \ell + 1 \not\in I = I(\alpha, \beta)$, which is absurd.  If $\alpha_{j + \ell + 1} < p - 1$, then $I(\alpha, 2 p^j) = \{ j + i : i \in [\ell]_0 \}$.  By Lemma~\ref{cor:carries.triple.sum} we have $I = I \cup I(\alpha + \beta, 2 p^j) = I(\alpha, 2 p^j) \cup I(\alpha + 2 p^j, \beta)$.  Hence $I^\dpr \subseteq I(\alpha + 2 p^j, \beta)$.  But the type $(I(\alpha + 2 p^j, \beta), \gamma \dotdiv p^j)$ is manifestly $r$-admissible, so $I^\dpr = I(\alpha + 2 p^j, \beta)$ by the maximality of $I^\dpr$.
Similarly, if $\beta_{j + \ell + 1} < p - 1$ then $I^\dpr = I(\alpha, \beta + 2 p^j)$.

{\emph{Case 4: $\alpha_j = \beta_j = p - 1$ or ($\{ \alpha_j, \beta_j \} = \{ p - 1, p - 2 \}$ and $j - 1 \in I$)}}.  
Let $\ell \in [f-1]_0$ be maximal such that $\alpha_{j+i} = p - 1$ for all $i \in [\ell]$.  We may also assume that $\beta_{j + i} = p - 1$ for all $i \in [\ell]$; otherwise, we proceed analogously but switch the roles of $\alpha$ and $\beta$.  Note that $I(\alpha, 2 p^j) = \{ j + i : i \in [\ell]_0 \}$ and that $\{ j + i : i \in [\ell] \} \subseteq I(\alpha + 2 p^j, \beta)$.  

If $\alpha_j = p - 1$, then $j \in I(\alpha + 2 p^j, \beta)$.
Hence $I \cup I(\alpha + \beta, 2 p^j) = I(\alpha, 2 p^j) \cup I(\alpha + 2 p^j, \beta) = I(\alpha + 2 p^j, \beta)$, where the first equality comes from Lemma~\ref{cor:carries.triple.sum}.  

It remains only to treat the case $(\alpha_{j}, \beta_j) = (p-2, p-1)$ with $j - 1 \in I$.  Suppose first that $j - 1 \in I(\alpha, 2 p^j)$.  Then $\ell = f - 1$, hence $\alpha = q - 1 - p^j$ and $\beta = q - 1$, so that $I(\alpha, 2 p^j) = I(\alpha + 2p^j, \beta) = I(p^j, q - 1) = \Z / f\Z$.  Now assume $j - 1 \not\in I(\alpha, 2 p^j)$.  By assumption $j - 1 \in I \subseteq I(\alpha, 2 p^j) \cup I(\alpha + 2 p^j, \beta)$, so necessarily $j - 1 \in I(\alpha + 2 p^j, \beta)$.  This implies that $(\alpha + 2 p^j + \beta)_j = 0$, whence $j \in I(\alpha + 2 p^j, \beta)$ by Remark~\ref{rem:admissible.types}.  In either case, we have again shown that $I \cup I(\alpha + \beta, 2 p^j) = I(\alpha + 2 p^j, \beta)$.
\end{proof}

\begin{cor} \label{cor:other.direction.partial.orders}
The partial order $\leq_r$ is a refinement of $\preceq_r$.
\end{cor}
\begin{proof}
Suppose that $(I^\prime, \gamma^\prime) \preceq_r (I^\dpr, \gamma^\dpr)$ are $r$-admissible types.  
Repeatedly applying Lemma~\ref{lem:induction.step.jumps}, 
we may refine $(I^\prime, \gamma^\prime) \preceq_r (I, \gamma)$ to a sequence of steps, in which each step is of one of four sorts: either it is of the form $(I^\prime, \gamma) \preceq_r (I, \gamma)$ for some $I^\prime \subseteq I$, or there exist $\alpha, \beta, \gamma \in \Nt$ such that $\alpha + \beta = r - 2\gamma$ and one of the following three possibilities holds:
\begin{itemize}
\item $(I(\alpha + 2p^j, \beta), \gamma \dotdiv p^j) \prec_r (I(\alpha, \beta), \gamma)$, for $p^j \preceq \gamma$;
\item $(I(\alpha + p^j, \beta + p^j), \gamma \dotdiv p^j) \prec_r (I(\alpha, \beta), \gamma)$, for $p^j \preceq \gamma$;
\item $(I(\alpha, \beta + 2p^j), \gamma \dotdiv p^j) \prec_r (I(\alpha, \beta), \gamma)$, for $p^j \preceq \gamma$.
\end{itemize}
  
We already showed at the beginning of the proof of Lemma~\ref{lem:one.refinement.partial.orders} that $(I^\prime, \gamma) \preceq_r (I, \gamma)$ is equivalent to $(I^\prime, \gamma) \leq_r (I, \gamma)$.    

In the first case of the trichotomy, we have
\begin{multline*}
 (I(\alpha + 2p^j, \beta), \gamma \dotdiv p^j) = (I(\beta, r - 2\gamma - \beta + 2p^j), \gamma \dotdiv p^j) = [(\beta, \gamma \dotdiv p^j)] \leq_r [(\beta, \gamma)] = \\ (I(\beta, r - 2\gamma - \beta), \gamma) = (I(\alpha, \beta), \gamma)
 \end{multline*}
by the second generating relation of Definition~\ref{def:covering.relations.ramified}.  The third case is the same, with the roles of $\alpha$ and $\beta$ switched.

In the remaining second case, we have 
$$
(I(\alpha+ p^j, \beta + p^j), \gamma \dotdiv p^j) = [(\alpha + p^j, \gamma \dotdiv p^j)] \leq_r [(\alpha, \gamma)] = (I(\alpha, \beta), \gamma)
$$
by the third generating relation of $\leq_r$.  Since each step of the refinement of $(I^\prime, \gamma^\prime) \preceq_r (I^\dpr, \gamma^\dpr)$ is compatible with $\leq_r$, we conclude that $(I^\prime, \gamma^\prime) \leq_r (I^\dpr, \gamma^\dpr)$.
\end{proof}

Together, Lemma~\ref{lem:one.refinement.partial.orders} and Corollary~\ref{cor:other.direction.partial.orders} imply
Proposition~\ref{pro:equality.of.orders}.

\bibliographystyle{amsplain}
\providecommand{\bysame}{\leavevmode\hbox to3em{\hrulefill}\thinspace}
\providecommand{\MR}{\relax\ifhmode\unskip\space\fi MR }

\providecommand{\href}[2]{#2}

\end{document}